\documentclass[a4paper,12pt, twoside, pdftex]{amsart}

\usepackage{fullpage}
\usepackage{amsmath, amsthm, amssymb}
\usepackage{txfonts}
\usepackage{amsfonts}
\usepackage{prettyref}
\usepackage{mathrsfs}
\usepackage[all]{xy}
\usepackage[utf8]{inputenc}

\theoremstyle{plain}
\newtheorem{thm}{Theorem}[section]
\newtheorem{cor}[thm]{Corollary}
\newtheorem{lem}[thm]{Lemma}
\newtheorem{prop}[thm]{Proposition}
\newtheorem{dfn}[thm]{Definition}
\newtheorem*{acknowledgements}{Acknowledgements}
\newtheorem*{notations}{Notations and conventions}

\theoremstyle{remark}
\newtheorem{rmk}{Remark}[section]

\numberwithin{equation}{section}

\usepackage{tikz}
\usetikzlibrary{backgrounds}

\newrefformat{th}{Theorem~\ref{#1}}
\newrefformat{cr}{Corollary~\ref{#1}}
\newrefformat{lm}{Lemma~\ref{#1}}
\newrefformat{dl}{Definition-Lemma~\ref{#1}}
\newrefformat{df}{Definition~\ref{#1}}
\newrefformat{cl}{Claim~\ref{#1}}
\newrefformat{sl}{Sublemma~\ref{#1}}
\newrefformat{pr}{Proposition~\ref{#1}}
\newrefformat{cj}{Conjecture~\ref{#1}}
\newrefformat{st}{Step~\ref{#1}}
\newrefformat{sc}{Section~\ref{#1}}
\newrefformat{df}{Definition~\ref{#1}}
\newrefformat{rm}{Remark~\ref{#1}}
\newrefformat{q}{Question~\ref{#1}}
\newrefformat{pb}{Problem~\ref{#1}}
\newrefformat{cd}{Condition~\ref{#1}}
\newrefformat{eg}{Example~\ref{#1}}
\newrefformat{he}{Heore~\ref{#1}}
\newrefformat{fg}{Figure~\ref{#1}}
\newrefformat{tb}{Table~\ref{#1}}
\newrefformat{as}{Assumption~\ref{#1}}

\newcommand{\pref}{\prettyref}
\newcommand{\Ab}{\operatorname{Ab}}
\newcommand{\Acycl}{\operatorname{Acycl}}
\newcommand{\Add}{\operatorname{Add}}
\newcommand{\Alg}{\operatorname{Alg}}
\newcommand{\AR}{\operatorname{ar}}
\newcommand{\Art}{\operatorname{Art_\bfk}}
\newcommand{\can}{\operatorname{can}}
\newcommand{\Cat}{\operatorname{Cat}}
\newcommand{\coh}{\operatorname{coh}}
\newcommand{\Com}{\operatorname{Com}}
\newcommand{\Coder}{\operatorname{Coder}}
\newcommand{\CAlg}{\operatorname{CAlg}}
\newcommand{\Cone}{\operatorname{Cone}}
\newcommand{\ddef}{\operatorname{def}}
\newcommand{\Def}{\operatorname{Def}}
\newcommand{\Der}{\operatorname{Der}}
\newcommand{\Des}{\operatorname{Des}}
\newcommand{\dgCat}{\operatorname{dgCat}}
\newcommand{\DOT}{\operatorname{dot}}
\newcommand{\embr}{\operatorname{embr}}
\newcommand{\End}{\operatorname{End}}
\newcommand{\Ext}{\operatorname{Ext}}
\newcommand{\Gd}{\operatorname{Gd}}

\newcommand{\Ho}{\operatorname{Ho}}
\newcommand{\Hom}{\operatorname{Hom}}
\newcommand{\Hmo}{\operatorname{Hmo}}
\newcommand{\id}{\operatorname{id}}
\newcommand{\Ind}{\operatorname{Ind}}
\newcommand{\Inj}{\operatorname{Inj}}
\newcommand{\KS}{\operatorname{KS}}
\newcommand{\MC}{\operatorname{MC}}
\newcommand{\mmod}{\operatorname{mod}}

\newcommand{\Mod}{\operatorname{Mod}}
\newcommand{\Nat}{\operatorname{Nat}}
\newcommand{\Ob}{\operatorname{Ob}}
\newcommand{\PCom}{\operatorname{PCom}}
\newcommand{\PDes}{\operatorname{PDes}}
\newcommand{\Perf}{\operatorname{Perf}}
\newcommand{\Qch}{\operatorname{Qch}}
\newcommand{\QPr}{\operatorname{QPr}}
\newcommand{\Rng}{\operatorname{Rng}}
\newcommand{\Set}{\operatorname{Set}}
\newcommand{\sign}{\operatorname{sign}}
\newcommand{\Spec}{\operatorname{Spec}}
\newcommand{\Spf}{\operatorname{Spf}}
\newcommand{\Tor}{\operatorname{Tor}}
\newcommand{\Tw}{\operatorname{Tw}}

\newcommand{\cC}{\mathcal{C}}

\newcommand{\cH}{\mathcal{H}}
\newcommand{\cN}{\mathcal{N}}
\newcommand{\cP}{\mathcal{P}}
\newcommand{\cU}{\mathcal{U}}
\newcommand{\cV}{\mathcal{V}}
\newcommand{\cW}{\mathcal{W}}

\newcommand{\bN}{\mathbb{N}}
\newcommand{\bQ}{\mathbb{Q}}
\newcommand{\bZ}{\mathbb{Z}}

\newcommand{\bfc}{\mathbf{c}}
\newcommand{\bff}{\mathbf{f}}
\newcommand{\bfg}{\mathbf{g}}
\newcommand{\bfh}{\mathbf{h}}
\newcommand{\bfk}{\mathbf{k}}
\newcommand{\bfm}{\mathbf{m}}
\newcommand{\bfC}{\mathbf{C}}
\newcommand{\bfD}{\mathbf{D}}
\newcommand{\bfI}{\mathbf{I}}
\newcommand{\bfP}{\mathbf{P}}
\newcommand{\bfR}{\mathbf{R}}
\newcommand{\bfS}{\mathbf{S}}

\newcommand{\scrA}{\mathscr{A}}
\newcommand{\scrB}{\mathscr{B}}
\newcommand{\scrC}{\mathscr{C}}
\newcommand{\scrD}{\mathscr{D}}
\newcommand{\scrE}{\mathscr{E}}
\newcommand{\scrF}{\mathscr{F}}
\newcommand{\scrG}{\mathscr{G}}
\newcommand{\scrO}{\mathscr{O}}
\newcommand{\scrT}{\mathscr{T}}
\newcommand{\scrX}{\mathscr{X}}

\newcommand{\fraka}{\mathfrak{a}}
\newcommand{\frakb}{\mathfrak{b}}
\newcommand{\frakc}{\mathfrak{c}}
\newcommand{\fraki}{\mathfrak{i}}
\newcommand{\frakj}{\mathfrak{j}}
\newcommand{\frakm}{\mathfrak{m}}
\newcommand{\frakI}{\mathfrak{I}}
\newcommand{\frakS}{\mathfrak{S}}
\newcommand{\frakU}{\mathfrak{U}}
\newcommand{\frakZ}{\mathfrak{Z}}

\begin{document}

\title{Versal dg deformation of Calabi--Yau manifolds}
\author[H.~Morimura]{Hayato Morimura}
\address{SISSA, via Bonomea 265, 34136 Trieste, Italy}
\email{hmorimur@sissa.it}

\date{}
\pagestyle{plain}

\begin{abstract}
We prove the equivalence of
the deformation theory for a higher dimensional Calabi--Yau manifold
and
that for its dg category of perfect complexes
by giving a natural isomorphism of the deformation functors.
As a consequence,
the dg category of perfect complexes on a versal deformation of the original manifold provides a versal Morita deformation of its dg category of perfect complexes.
Besides the classical uniqueness up to \'etale neiborhood of the base,
we prove another sort of uniqueness of versal Morita deformations. 
Namely,
given a pair of derived-equivalent higher dimensional Calabi--Yau manifolds,
the dg categories of perfect complexes of their algebraic deformations over a common base,
which always exist,
become quasi-equivalent close to effectivizations.
Then the base change along the corresponding first order approximation yields quasi-equivalent versal Morita deformations.
We introduce the generic fiber of the versal Morita deformation as a Drinfeld quotient,
which is quasi-equivalent to the dg category of perfect complexes on the generic fiber of the versal deformation.
\end{abstract}

\maketitle

\section{Introduction}
The derived category of coherent sheaves on an algebraic variety is an intensively studied invariant
which carries rich information about geometric properties of the variety.
For instance,
given a smooth projective variety
either of whose canonical or anticanonical bundle is ample,
one can reconstruct the variety from its derived category
\cite{BO01}.
The condition guarantees the absence of nontrivial autoequivalences of the derived category.
Such autoequivalences often stem from the derived equivalence of nonisomorphic,
sometimes even nonbirational Calabi--Yau manifolds.
According to the homological mirror symmetry conjecture by Kontsevich,
derived-equivalent Calabi--Yau manifolds should share their mirror partner.
Usually,
the homological mirror symmetry is considered for families of K\"{a}hler manifolds.

A goal of this paper is to study the relationship between
deformations
and
the derived category
of a higher dimensional Calabi--Yau manifold.
There seems to be a consensus among some experts
that
deforming
an algebraic variety
and
its derived category
are essentially the same. 
Philosophically,
it is reasonable
since their Hochschild cohomology,
which in general is known to control deformations of a mathematical object,
are isomorphic.
However,
before
\cite{LV06}
we were not given the correct framework to study deformations of even
linear
nor
abelian categories.
To every second Hochschild cocycle on a smooth projective variety,
Toda associated the category of twisted coherent sheaves on the corresponding noncommutative scheme over the ring of dual numbers
\cite{Tod}.
In
\cite{DLL}
Dinh--Liu--Lowen showed that
Toda's construction indeed yields flat abelian first order deformations of the category of coherent sheaves on the variety in the sense of
\cite{LV06}.
We fill the gap between this point
and
the conclusion
stated below more precisely. 

Let
$X_0$
be a Calabi--Yau manifold of dimension more than two in the strict sense,
i.e.,
a smooth projective $\bfk$-variety with
$\omega_{X_0} \cong \scrO_{X_0}$
and
$H^i (\scrO_{X_0}) = 0$
for
$0 < i < \dim X_0$.
We denote by
$\Perf_{dg}(X_0)$
the dg category of perfect complexes on
$X_0$.
The deformation functor
\begin{align*}
\Def_{X_0}
\colon
\Art
\to
\Set
\end{align*}
sends each local artinian $\bfk$-algebra
$A \in \Art$
with residue field
$\bfk$
to the set of equivalence classes of $A$-deformations of
$X_0$
and each morphism
$B \to A$
in
$\Art$
to the map
$\Def_{X_0}(B) \to \Def_{X_0}(A)$
induced by the base change.
Consider another deformation functor
\begin{align*}
\Def^{mo}_{\Perf_{dg}(X_0)}
\colon
\Art
\to
\Set
\end{align*}
which sends each
$A \in \Art$
to the set of isomorphism classes of Morita $A$-deformations of
$\Perf_{dg}(X_0)$
and each morphism
$B \to A$
in
$\Art$
to the map
$\Def^{mo}_{\Perf_{dg}(X_0)}(B) \to \Def^{mo}_{\Perf_{dg}(X_0)}(A)$
induced by the derived dg functor
$- \otimes^L_B A$.
Our first main result claims that
the deformation theory for
$X_0$
is equivalent to
that for
$\Perf_{dg}(X_0)$
in the following sense.

\begin{thm} {\rm{(}\pref{thm:natisom}\rm{)}} \label{thm:natisomINTRO}
There is a natural isomorphism
\begin{align*}
\zeta \colon \Def_{X_0} \to \Def^{mo}_{\Perf_{dg}(X_0)}
\end{align*}
of deformation functors.
\end{thm}  

In particular,
Morita deformations of
$\Perf_{dg}(X_0)$
is controlled by the Kodaira--Spencer differential graded Lie algebra.   
To obtain
$\zeta$
we need to consider certain maximal partial curved dg deformations of
$\Perf_{dg}(X_0)$.
Curved dg deformations of a dg category is a special case of
curved $A_\infty$-deformations of an $A_\infty$-category.
Let
$(\fraka, \mu)$
be a dg category over
$\bfR \in \Art$
with a square zero extension
\begin{align*}
0 \to \bfI \to \bfS \to \bfR \to 0.
\end{align*}
Choose generators
$\epsilon = (\epsilon_1, \ldots, \epsilon_l)$    
of
$\bfI$
regarded as a free $\bfR$-module of rank
$l$.
By
\cite[Theorem 4.11]{Low08}
there is a bijection
\begin{align} \label{eq:classifyingINTRO}
H^2 \bfC (\fraka)^{\oplus l} \to \Def^{cdg}_\fraka(\bfS), \
\phi \mapsto \fraka_\phi = (\fraka[\epsilon], \mu + \phi \epsilon)
\end{align}
where
$\phi$
is a Hochschild cocycle.
In other words,
curved dg $\bfS$-deformations of
$\fraka$
are classified by the direct sum of the second Hochschild cohomology.

Assume that
$\fraka$
is an $\bfR$-linear category.
We denote by
$\Com^+(\fraka)$
the dg category of bounded below complexes of $\fraka$-objects.
Then by
\cite[Theorem 4.8]{Low08}
the characteristic morphism
\begin{align*}
\chi^{\oplus l}_\fraka
\colon
H^\bullet \bfC(\fraka)^{\oplus l}
\to
\frakZ^\bullet K^+(\fraka)^{\oplus l}
\end{align*}
maps
$\phi \in Z^2 \bfC(\fraka)^{\oplus l}$
to obstructions against deforming objects of
$K^+(\fraka)$
to objects of
$K^+(\fraka_\phi)$.
In particular,
for each
$C \in K^+(\fraka)$
there exists a lift to
$K^+(\fraka_\phi)$
if and only if
$\chi^{\oplus l}_\fraka (\phi)_C = 0$.
The characteristic morphism
$\chi^{\oplus l}_\fraka$
is induced by a $B_\infty$-section
\begin{align*}
\embr_\delta \colon \bfC(\fraka)^{\oplus l} \to \bfC(\Com^+(\fraka))^{\oplus l}
\end{align*}
of the canonical projection,
which is a quasi-isomorphism of $B_\infty$-algebras
\cite[Theorem 3.22]{Low08}.
Hence
\pref{eq:classifyingINTRO}
induces another bijection
\begin{align}
H^2 \bfC (\fraka)^{\oplus l} \to \Def^{cdg}_{\Com^+(\fraka)}(\bfS), \
\phi
\mapsto
\Com^+(\fraka)_{\embr_\delta (\phi)}
=
(\Com^+(\fraka)[\epsilon], \embr_\delta(\mu + \phi \epsilon)).
\end{align}
From the proof of
\cite[Theorem 4.8]{Low08}
it follows that
$\chi^{\oplus l}_\fraka(\phi)_C = 0$
if and only if
the curvature element
$(\embr_\delta(\mu + \phi \epsilon))_{0, C}$
vanishes
for each
$C \in \Com^+(\fraka)$.
Hence any full dg subcategory of
$\Com^+(\fraka)$
spanned by object
$C$
with
$\chi^{\oplus l}_\fraka(\phi)_C = 0$
dg deforms along the restriction of
$\embr_\delta(\phi)$.

Let
$X$
be an $\bfR$-deformation of
$X_0$
and
$X_\phi$
its deformation along a cocycle
$\phi \in HH^2(X)^{\oplus l}
=
H^1(\scrT_{X/\bfR})^{\oplus l}$.
The above argument can be adapted to our setting
so that
for each
$E \in \Perf_{dg}(X)$
the curvature element vanishes
if and only if
there exists a lift of
$E \in \Perf(X)$
to
$\Perf(X_\phi)$.
With a little more effort
one can apply
\cite[Proposition 3.12]{KL}
to obtain

\begin{thm} {\rm{(}\pref{thm:mocdg}\rm{)}}
There is a bijection
\begin{align*}
\Def^{mo}_{\Perf_{dg}(X)}(\bfS) \to \Def^{cdg}_{\Perf_{dg}(X)}(\bfS)
\end{align*}
between
the set of isomorphism classes of Morita $\bfS$-deformations
and
the set of isomorphism classes of curved dg $\bfS$-deformations
of
$\Perf_{dg}(X)$.
\end{thm}

In particular,
giving curved dg $\bfS$-deformations of
$\Perf_{dg}(X)$
is equivalent to giving its Morita $\bfS$-deformations.
Consider the dg category
$\Perf_{dg}(X_\phi)$
of perfect complexes on
$X_\phi$.
It defines a Morita $\bfS$-deformation of
$\Perf_{dg}(X)$.
Let
\begin{align*}
\frakm(\phi)
=
\Perf_{dg}(X_\phi) \otimes^L_\bfS \bfR
\end{align*}
be the image of the derived base change.
Then any $h$-flat resolution
$\overline{\Perf_{dg}(X_\phi)}$
defines a dg deformation
$\overline{\Perf_{dg, \Gamma}(X_\phi)}$
of
$\frakm(\phi)$,
where
$\overline{\Perf_{dg, \Gamma}(X_\phi)}$
is the full dg subcategory of
$\overline{\Perf_{dg}(X_\phi)}$
consisting the collection of one chosen lift of each object in
$\frakm(\phi)$.
There is an isomorphism
\begin{align*}
HH^2(X)^{\oplus l}
\cong
H^2 \bfC(\Perf_{dg}(X))^{\oplus l}
\end{align*}
induced by the $B_\infty$-section
\begin{align*}
\embr_\delta
\colon
\bfC(\Inj(\Qch(X)))^{\oplus l}
\to
\bfC(\Com^+(\Inj(\Qch(X))))^{\oplus l},
\end{align*}
where
$\Inj(\Qch(X)) \subset \Qch(X)$
is the full $\bfR$-linear subcategory of injective objects.
We denote by
$\embr_\delta(\phi)$
the image of
$\phi$
under the isomorphism,
which defines another dg $\bfS$-deformation
$\frakm(\phi)_{\embr_\delta(\phi)}$
of
$\frakm(\phi)$
along
$\embr_\delta(\phi)$.
Deformations
and
taking the dg category of perfect complexes
intertwine in the following sense.

\begin{thm} {\rm{(}\pref{thm:Intertwine2}\rm{)}}
There is an isomorphism
\begin{align*}
\overline{\Perf_{dg, \Gamma} (X_\phi)}
\simeq
\frakm(\phi)_{\embr_\delta(\phi)}
\end{align*}
of dg $\bfS$-deformations of
$\frakm(\phi)$.
In particular,
the Morita $\bfS$-deformation
$\Perf_{dg} (X_\phi)$
defines a maximal partial dg $\bfS$-deformaiton of
$\Perf_{dg} (X)$
along
$\embr_\delta(\phi)$.
\end{thm}

This is the key to prove
\pref{thm:natisomINTRO}. 
Unwinding Toda's construction,
from
\cite[Theorem 5.12]{DLL}
we obtain an equivalence
$\Qch(X)_\phi \simeq \Qch(X_\phi)$
of Grothendieck abelian categories,
where
$\Qch(X)_\phi$
is the flat abelian $\bfS$-deformation of
$\Qch(X)$
along
$\phi \in Z^2 \bfC_{ab}(\Qch(X))^{\oplus l}$.
Here,
we use the same symbol
$\phi$
to denote the image under the isomorphism
\begin{align*}
HH^2(X)^{\oplus l}
\cong
H^2 \bfC_{ab}(\Qch(X))^{\oplus l}.
\end{align*}
Via the induced equivalence 
\begin{align*}
D_{dg}(\Qch(X_\phi))
\simeq 
D_{dg}(\Qch(X)_\phi)
\end{align*}
we regard
$\Perf_{dg}(X_\phi)$
as the full dg subcategory of compact objects of
$D_{dg}(\Qch(X)_\phi)$.
Based on the idea in the proof of
\cite[Theorem 4.15]{Low08},
we compare the dg structure on
$\overline{\Perf_{dg, \Gamma} (X_\phi)}$
with that on
$\frakm(\phi)_{\embr_\delta(\phi)}$.

Working with Morita deformations of
$\Perf_{dg}(X_0)$,
by
\cite[Corollary 5.7]{Coh}
we may apply
\cite[Theorem 1.2]{BFN}
to obtain reductions.
In particular,
given a deformation
$(X_B, i_B) \in \Def_{X_0}(B)$
and
a morphism
$B \to A$
in
$\Art$,
there is a Morita equivalence
\begin{align*}
\Perf_{dg}(X_B) \otimes^L_B A
\simeq_{mo}
\Perf_{dg}(X_B) \otimes^L_B \Perf_{dg}(A)
\simeq_{mo}
\Perf_{dg}(X_A)
\end{align*}
of $A$-linear dg categories,
where
$- \otimes^L_B -$
is the derived pointwise tensor product of dg categories. 
Further application of
\cite[Theorem 1.2]{BFN}
shows that
any universal formal family for
$\Def^{mo}_{\Perf_{dg}(X_0)}$
is effective.
If we ignore the set theoretical issues,
the deformation functor
$\Def^{mo}_{\Perf_{dg}(X_0)}$
can naturally be extended to a functor defined on the category
$\Alg^{aug}(\bfk)$
of augmented noetherian $\bfk$-algebras.
Although we do not know
whether it would be locally of finite presentation
(colimit preserving),
one can always construct a versal Morita deformation  
via geometric realization in the following sense.

\begin{cor} {\rm{(}Corollary \pref{cor:algebraizable}\rm{)}} \label{cor:algebraizableINTRO}
Any effective universal formal family for
$\Def^{mo}_{\Perf_{dg}(X_0)}$
is algebraizable.
In particular,
an algebraization is given by
$\Perf_{dg}(X_S)$
where
$(\Spec S, s, X_S)$
is a versal deformation of
$X_0$. 
\end{cor}  

The versal Morita deformation
$\Perf_{dg}(X_S)$
may be regarded as a family of Morita deformations of
$\Perf_{dg}(X_0)$.
More generally,
for such a family determined by an enough nice $S$-scheme
$X_S$
we introduce its generic fiber as follows.

\begin{dfn}
Let
$X_S$
be a smooth separated scheme
over a noetherian connected regular affine $\bfk$-scheme
$\Spec S$
whose closed points are
$\bfk$-rational.
Then the
\emph{dg categorical generic fiber}
of
$\Perf_{dg}(X_S)$
is the Drinfeld quotient
\begin{align*}
\Perf_{dg}(X_S) / \Perf_{dg}(X_S)_0,
\end{align*}
where
$\Perf_{dg}(X_S)_0 \subset \Perf_{dg}(X_S)$
is the full dg subcategory of perfect complexes with $S$-torsion cohomology.
\end{dfn}

We impose a technical assumption on
$S$
to include also the case
where
$S$
is a formal power series ring.
The Drinfeld quotient is a natural dg enhancement of the categorical generic fiber introduced in
\cite{Morb},
which is in turn based on the categorical general fiber by Huybrechts--Macr\`i--Stellari
\cite{HMS11}.
Taking the generic fiber
and
the dg category of perfect complexes 
intertwine in the following sense.

\begin{prop} {\rm{(}Proposition \pref{prop:dgCGF}\rm{)}}
Let
$X_S$
be a smooth separated scheme
over a noetherian connected regular affine $\bfk$-scheme
$\Spec S$
whose closed points are
$\bfk$-rational.
Then there is a quasi-equivalence
\begin{align*}
\Perf_{dg}(X_S) / \Perf_{dg}(X_S)_0 \simeq_{qeq} \Perf_{dg}(X_{Q(S)})
\end{align*}
where
$Q(S)$
is the quotient field of
$S$
and
$X_{Q(S)}$
is the generic fiber of
$X_S$.
\end{prop}  

Another goal of this paper is to show the uniqueness of versal Morita deformations
with respect to geometric realizations.
Recall that
up to \'etale neighborhood of the base
versal deformations 
of
$X_0$
are unique.
Namely,
if
$(\Spec S, s, X_S)$
$(\Spec S^\prime, s^\prime, X_{S^\prime})$
are two versal deformations of
$X_0$,
then there is another versal deformation
$(\Spec S^{\prime \prime}, s^{\prime \prime}, X_{S^{\prime \prime}})$
such that
$(\Spec S^{\prime \prime}, s^{\prime \prime})$
is an \'etale neighborhood of
$s, s^\prime$
in
$\Spec S, \Spec S^\prime$
respectively
and
$X_{S^{\prime \prime}}$
is the pullback along the corresponding \'etale morphisms.
The deformation functor
$\Def_{X_0}$
has an effective universal formal family
$(R, \xi)$,
where
$R$
is a regular complete local noetherian $\bfk$-algebra.
Choose an isomorphism 
$R \cong \bfk \llbracket t_1, \ldots, t_d \rrbracket$
and
let 
$T = \bfk [t_1, \ldots, t_d]$
with
$d = \dim_\bfk \text{H}^1 (\scrT_{X_0})$.
There is a filtered inductive system
$\{ R_i \}_{i \in I}$
of finitely generated $T$-subalgebras of
$R$
whose colimit is
$R$.
Then
$(\Spec S, s)$
is an \'etale neighborhood of
$t$
in
$\Spec T$
with
$t$
corresponding to the maximal ideal
$(t_1, \ldots, t_d) \subset T$,
and
$X_S$
is the pullback of a deformation
$X_{R_j}$ 
of
$X_0$
along a first order approximation
$R_j \to S$
of
$R_j \hookrightarrow R$
for sufficiently large
$j \in I$.
Hence the ambiguity of
$X_S$ 
stems from the choice of
$j \in I$,
besides the choice of \'etale neighborhoods.

In
\cite{Mora}
the author constructed smooth projective versal deformations
$X_S, X^\prime_S$
of
$X_0, X^\prime_0$
over a common nonsingular affine variety
$\Spec S$,
while deforming simultaneously the Fourier--Mukai kernel connecting deformations of
$X_0, X^\prime_0$.
By
Corollary
\pref{cor:algebraizableINTRO}
we have two versal Morita deformations
$\Perf_{dg}(X_S), \Perf_{dg}(X^\prime_S)$
of
$\Perf_{dg}(X_0)$.
\pref{thm:natisomINTRO}
together with the construction of versal deformations
suggests that
$\Perf_{dg}(X_S), \Perf_{dg}(X^\prime_S)$
should be determined only by
quasi-equivalent universal formal families
and
the same sufficiently large index
$j \in I$.
From this observation we arrive at our second main result.

\begin{thm} {\rm{(}\pref{thm:inherit}\rm{)}} \label{thm:inheritINTRO}
Let
$X_0, X^\prime_0$
be derived-equivalent Calabi--Yau manifolds of dimension more than two
and
$\cP_0 \in D^b(X_0 \times_\bfk X^\prime_0)$
the Fourier--Mukai kernels.
Then there exists an index
$j \in I$
such that
for all
$k \geq j$
the integral functors
\begin{align*}
\Phi_{\cP_k}
\colon
\Perf(X_{R_k})
\to
\Perf(X^\prime_{R_k})
\end{align*}
defined by deformations
$\cP_k$
of
$\cP_0$
are equivalences of triangulated categories of perfect complexes.
In particular,
the dg categories
$\Perf_{dg}(X_{R_k}), \Perf_{dg}(X^\prime_{R_k})$
of perfect complexes are quasi-equivalent.
\end{thm}

\pref{thm:inheritINTRO}
tells us that,
given two algebraic Morita deformations
$\Perf_{dg}(X_{R_k}), \Perf_{dg}(X^\prime_{R_k})$
geometrically realized by algebraic deformations
$X_{R_k}, X^\prime_{R_k}$
of two derived-equivalent higher dimensional Calabi--Yau manifolds
$X_0, X^\prime_0$,
if
$X_{R_k}, X^\prime_{R_k}$
are enough close to effectivizations
$X_R, X^\prime_R$
then
$\Perf_{dg}(X_{R_k}), \Perf_{dg}(X^\prime_{R_k})$
are Morita equivalent.
The base change along the homomorphism
$R_k \to S$
yields Morita equivalent versal Morita deformations
$\Perf_{dg}(X_S), \Perf_{dg}(X^\prime_S)$.
In other words,
up to Morita equivalence
the versal Morita deformation
$\Perf_{dg}(X_S)$
does not depend on the choice of geometric realizations in the following sense.

\begin{cor} {\rm{(}Corollary \pref{cor:inherit}\rm{)}} \label{cor:inheritINTRO}
Let
$X_0, X^\prime_0$
be derived-equivalent Calabi--Yau manifolds of dimension more than two
and
$X_S, X^\prime_S$
their smooth projective versal deformations
over a common nonsingular affine $\bfk$-variety
$\Spec S$.
Let
$\{ R_i \}_{i \in I}$
be a filtered inductive system of finitely generated $T$-subalgebras
$R_i \subset R$
whose colimit is
$R$
with
\begin{align*}
T = \bfk [t_1, \ldots, t_d], \
R \cong \bfk \llbracket t_1, \ldots, t_d \rrbracket, \
d = \dim_\bfk H^1(\scrT_{X_0}).
\end{align*}
Assume that
$X_S, X^\prime_S$
correspond to a first order approximation
$R_j \to S$
of
$R_j \hookrightarrow R$
for sufficiently large
$j \in I$.
Then
$X_S, X^\prime_S$
are derived-equivalent.
In particular,
the dg categories
$\Perf_{dg}(X_S), \Perf_{dg}(X^\prime_S)$
of perfect complexes are quasi-equivalent.
\end{cor}

The uniqueness result also holds for the dg categorical generic fiber.
Corollary
\pref{cor:inheritINTRO}
slightly improves
\cite[Theorem 1.1]{Mora},
which extends the derived equivalence from special to general fibers.
Here,
the advantage is that
we do not have to shrink the base
$\Spec S$
as long as the construction passes enough close to effectivizations.
In particular,
beginning with a pair of general fibers,
one obtains the derived equivalence of special fibers contained in the versal deformations.
Hence the above corollary partially provides a method for the opposite direction,
i.e.,
how to extend the derived equivalence from general to special fibers.

\begin{notations}
We work over an algebraically closed field $\bf{k}$ of characteristic $0$ throughout this paper.
For an augmented $\bf{k}$-algebra
$A$
by
$\frakm_A$
we denote its augmentation ideal.
All higher dimensional Calabi--Yau manifolds we treat
are
smooth projective $\bfk$-varieties
$X_0$
of dimension more than two with
$\omega_{X_0} \cong \scrO_{X_0}$
and
$H^i (\scrO_{X_0}) = 0$
for
$0 < i < \dim X_0$.
\end{notations}

\begin{acknowledgements}
The author is supported by SISSA PhD scholarships in Mathematics.
The author would like to thank Yukinobu Toda for pointing out mistakes in earlier version.
\end{acknowledgements}

\section{Hochschild cohomology of relatively smooth proper schemes} \label{sec:Hochsch}
In this section,
we review various kinds of complexes
whose cohomology controls deformations of associated mathematical objects,
mainly following the exposition from
\cite[Section 2, 3]{DLL}.
We always assume that
all algebras have units,
morphism of algebras preserve units,
and
modules are unital.
In the sequel,
we fix
a local artinian $\bfk$-algebra
$\bfR$
with residue field
$\bfk$
and
its square zero extension
\begin{align*}
0 \to \bfI \to \bfS \to \bfR \to 0,
\end{align*}
and
choose generators
$\epsilon = (\epsilon_1, \ldots, \epsilon_l)$    
of
$\bfI$
regarded as a free $\bfR$-module of rank
$l$.
For smooth proper $\bfR$-schemes,
we explain the correspondence between
its relative Hochschild cohomology
and
cohomology of the Gerstenharber--Shack complex
associated with its restricted structure sheaf. 
  
\subsection{Relative Hochschild cohomology of schemes}
Let
$X$
be a smooth proper $\bfR$-scheme.
We denote by
$\Delta_\bfR \colon X_\bfR \hookrightarrow X \times_\bfR X$
the relative diagonal embedding.
The
\emph{relative Hochschild cohomology}
is defined as the graded $\bfR$-algebra
\begin{align*}
HH^\bullet (X/\bfR)
=
\Ext^\bullet_{X \times_\bfR X}(\scrO_{\Delta_\bfR}, \scrO_{\Delta_\bfR})
\cong
\Ext^\bullet_X(\Delta^*_\bfR \scrO_{\Delta_\bfR}, \scrO_X).
\end{align*}
Here,
the multiplication in
$HH^\bullet (X/\bfR)$
is given
by the composition in
$D^b(X \times_\bfR X)$. 
Then the natural map
$\bfR \to \End_{X \times_\bfR X}(\scrO_{\Delta_\bfR})$
induces the $\bfR$-algebra structure.
There is a quasi-isomorphism
\begin{align*}
\Delta^*_\bfR \scrO_{\Delta_\bfR}
\cong
\bigoplus_i \Omega^i_{X/\bfR}[i]
\end{align*}
called the
\emph{relative Hochschild--Kostant--Rosenberg isomorphism},
which induces an isomorphism
\begin{align*}
\Ext^\bullet_X(\Delta^*_\bfR \scrO_{\Delta_\bfR}, \scrO_X)
\to
\Ext^\bullet_X(\bigoplus_i \Omega^i_{X/\bfR}[i], \scrO_X)
\to
\bigoplus_i H^{\bullet - i}(X, \wedge^i \scrT_{X/\bfR})
\end{align*}
where
$\scrT_{X/\bfR}$
is the relative tangent sheaf
and
$\Omega_{X/\bfR}$
is its dual.
We also call the compositions
\begin{align*}
\text{I}_{X/\bfR}^\text{HKR}
\colon
\Ext^n_{X \times_\bfR X}(\scrO_{\Delta_\bfR}, \scrO_{\Delta_\bfR})
=
HH^n (X/\bfR)
\to
HT^n (X/\bfR)
=
\bigoplus_{p+q=n} H^p(X, \wedge^q \scrT_{X/\bfR})
\end{align*}
the relative Hochschild--Kostant--Rosenberg isomorphisms.

\subsection{Hochschild cohomology of algebras}
Let
$A = (A, m)$
be an $\bfR$-algebra
and
$M$
an $A$-bimodule.
The  
\emph{Hochschild complex}
$\bfC(A, M)$  
has
$\bfC^n(A, M) = \Hom_\bfR (A^{\otimes n}, M)$
as its $n$-th term
and
$d^n_{Hoch}
\colon
\bfC^n(A, M) \to \bfC^{n+1}(A, M)$,
called the
\emph{Hochschild differential},
as its differential
which is given by
\begin{align*}
d^n_{Hoch}(\phi)(a_n, a_{n-1}, \ldots, a_0)
= \ 
&a_n \phi (a_{n-1}, \ldots, a_0) \\
&+ \sum^{n-1}_{i=0}(-1)^{i+1} \phi (a_n, \ldots, a_{n-i}a_{n-i-1}, \ldots, a_0) \\
&+ (-1)^{n+1} \phi (a_n, \ldots, a_1)a_0.
\end{align*} 

A cochain
$\phi \in \bfC^n(A, M)$
is
\emph{normalized}
if
$\phi (a_{n-1}, \ldots, a_0) = 0$
whenever
$a_i = 1$
for some
$0 \leq i \leq n-1$.
The normalized cochains form a subcomplex
$\bar{\bfC} (A, M)$
quasi-isomorphic to
$\bfC (A, M)$
via the inclusion.
When
$M = A$,
we call
$\bfC(A) = \bfC (A, A)$
the
\emph{Hochschild complex}
and
$H^n \bfC(A)$
the $n$-th
\emph{Hochschild cohomology}
of
$A$.
Note that
the multiplication
$m$
on
$A$
belongs to
$\bfC^2(A)$.

The direct sum of the second normalized Hochschild cohomology of
$A$
classifies $\bfS$-deformations of
$A$
up to equivalence.
Recall that
an
\emph{$\bfS$-deformation}
of
$A$
is an $\bfS$-algebra
$(\bar{A}, \bar{m})
=
(A[\epsilon] = A \otimes_\bfR \bfS, m + \bfm \epsilon)$
with
$\bfm \in \bfC^2(A)^{\oplus l}$
such that
the unit of
$\bar{A}$
is the same as that of $A$. 
Two deformations
$(\bar{A}, \bar{m}), (\bar{A}^\prime, \bar{m}^\prime)$
are
\emph{equivalent}
if there is an isomorphism of the form 
$1+ \bfg \epsilon \colon \bar{A} \to \bar{A}^\prime$
with
$\bfg \in \bfC^1(A)^{\oplus l}$.
We denote by
$\Def^{alg}_A(\bfS)$   
the set of equivalence classes of $\bfS$-deformations of
$A$.
It is known that
there is a bijection
\begin{align*}
H^2 \bar{\bfC}(A)^{\oplus l}
\to
\Def^{alg}_A(\bfS), \
\bfm \mapsto (A[\epsilon], m +\bfm \epsilon), \
\bfm \in Z^2 \bar{\bfC}(A)^{\oplus l}.
\end{align*}

\subsection{Simplicial cohomology of presheaves}
Let
$\frakU$
be a small category
and
$\cN(\frakU)$
its simplicial nerve.
We write 
\begin{align*}
\sigma
=
(d \sigma = U_0 \xrightarrow{u_1} U_1 \xrightarrow{u_2} \cdots \xrightarrow{u_p} U_p \xrightarrow{u_{p+1}} U_{p+1} = c \sigma)
\end{align*}
for a $(p+1)$-simplex
$\sigma \in \cN_{p+1}(\frakU)$.
Let
$(\scrF, f), (\scrG, g)$
be presheaves of $\bfR$-modules with restriction maps
$f^u \colon \scrF(U) \to \scrF(V), g^u \colon \scrG(U) \to \scrG(V)$
for
$u \colon V \to U$
in
$\frakU$.
We write
$f^\sigma$
for the map
$f^{u_{p+1} \ldots u_2 u_1}
\colon
\scrF(U_{p+1}) \to \scrF(U_0)$.
Consider a complex whose $p$-th term is
\begin{align*}
\bfC^p_{simp}(\scrG, \scrF)
=
\prod_{\tau \in \cN_p(\frakU)}
\Hom_\bfR \left( \scrG(c\tau), \scrF(d \tau) \right).
\end{align*}
and
whose differential
$d^p_{simp}$
is defined as follows.
Recall that
we have the maps
\begin{align*}
\partial_i \colon \cN_{p+1}(\frakU) \to \cN_p(\frakU), \
\sigma \mapsto \partial_i \sigma,
\end{align*}
for
$i = 0, 1, \ldots, p+1$
given by
\begin{align*}
\begin{gathered}
\partial_i \sigma
=
(U_0 \xrightarrow{u_1} \cdots U_{i-1} \xrightarrow{u_{i+1} u_i} U_{i+1} \xrightarrow{u_{i+1}} \cdots \xrightarrow{u_p} U_p \xrightarrow{u_{p+1}} U_{p+1}), \
i \neq 0, p+1, \\
\partial_0 \sigma
=
(U_1 \xrightarrow{u_2} U_2 \xrightarrow{u_3} \cdots \xrightarrow{u_p} U_p \xrightarrow{u_{p+1}} U_{p+1}), \\
\partial_{p+1} \sigma
=
(U_0 \xrightarrow{u_1} U_1 \xrightarrow{u_2} \cdots \xrightarrow{u_p} U_p).
\end{gathered}
\end{align*}
Each $\partial_i$ induces a map
\begin{align*}
d_i
\colon
\bfC^p_{simp}(\scrG, \scrF)
\to
\bfC^{p+1}_{simp}(\scrG, \scrF), \ 
\phi = (\phi^\tau)_\tau \mapsto d_i \phi = ((d_i \phi)^\sigma)_\sigma
\end{align*}
given by
\begin{align*}
\begin{gathered}
(d_i \phi)^\sigma
=
\phi^{\partial_i \sigma}, \ i \neq 0, p+1, \\
(d_0 \phi)^\sigma
=
f^{u_1} \circ \phi^{\partial_0 \sigma}, \\
(d_{p+1} \phi)^\sigma
=
\phi^{\partial_{p+1} \sigma} \circ g^{u_{p+1}}.\end{gathered}
\end{align*}
Then one defines
\begin{align*}
d_{simp}
=
\sum^{p+1}_ {i=0} (-1)^i d_i
\colon
\bfC^p_{simp}(\scrG, \scrF)
\to
\bfC^{p+1}_{simp}(\scrG, \scrF).
\end{align*}

When
$\scrG$
is the constant presheaf
$\bfR$,
we call
$H^p(\frakU, \scrF)
=
H^p\bfC_{simp}(\scrF)
=
H^p\bfC_{simp}(\bfR, \scrF)$ 
the
\emph{simplicial presheaf cohomology}
of $\scrF$.
A $(p+1)$-simplex
$\sigma \in \cN_{p+1} (\frakU)$
is
\emph{degenerate}
if
$u_i = 1_{U_i}$
for some
$1 \leq i \leq p+1$.
A $p$-cochain
$\phi = (\phi^\tau)_\tau \in \bfC^p(\scrG, \scrF)$
is
\emph{reduced}
if
$\phi^\tau = 0$
whenever $\tau$ is degenerate.
All $0$-cochains are reduced by convention.
The reduced cohains are preserved by $d_{simp}$
and
form a subcomplex
$\bfC^\prime_{simp}(\scrG, \scrF)$,
which is quasi-isomorphic to
$\bfC_{simp}(\scrG, \scrF)$
by
\cite[Proposition 2.9]{DLL}.

The direct sum of the first reduced simplicial presheaf cohomology of
$\scrF$
classifies $\bfS$-deformations of
$\scrF$
up to equivalence.
Recall that
an
\emph{$\bfS$-deformation}
of
$\scrF$
is a presheaf of $\bfS$-modules
$(\bar{\scrF}, \bar{f})
=
(\scrF[\epsilon], f + \bff \epsilon)$
with
$\bff \in \bfC^1_{simp}(\scrF, \scrF)^{\oplus l}$.
Two deformations
$(\bar{\scrF}, \bar{f}), (\bar{\scrF}^\prime, \bar{f}^\prime)$
are
\emph{equivalent}
if there is an isomorphism of the form
$1 + \bfg \epsilon$
with
$\bfg \in \bfC^0_{simp}(\scrF, \scrF)^{\oplus l}$.
We denote by
$\Def^{psh}_\scrF(\bfS)$
the set of equivalence classes of $\bfS$-deformations of
$\scrF$.

\begin{lem} {\rm{(}\cite[Proposition 2.11]{DLL}\rm{)}} 
Let
$(\scrF, f)$
be a presheaf of $\bfR$-modules.
Then there is a bijection
\begin{align*}
H^1 \bfC^\prime_{simp}(\scrF, \scrF)^{\oplus l}
\to
\Def^{psh}_\scrF(\bfS), \ 
\phi \mapsto (\scrF[\epsilon], f + \bff \epsilon), \
\bff \in Z^1 \bfC^\prime_{simp}(\scrF, \scrF)^{\oplus l}.
\end{align*}
Another cocycle
$\bff^\prime \in Z^1 \bfC^\prime_{simp}(\scrF, \scrF)^{\oplus l}$
maps to an equivalent deformation
if and only if
there is an element
$\bfg \in \bfC^{\prime 0}_{simp}(\scrF, \scrF)$
satisfying
$\bff^\prime - \bff = d_{simp} (\bfg)$.
\end{lem}

\subsection{Gerstenharber--Schack complexes}
Let
$\frakU$
be a small category
and
$(\scrA, m, f)$
a presheaf of $\bfR$-algebras on
$\frakU$.
The 
\emph{Gerstenharber--Schack complex}
$\bfC_{GS}(\scrA)$
introduced in \cite{GS88}
is the total complex of the double complex
whose $(p,q)$-term
for
$p, q \geq 0$
is
\begin{align*}
\bfC^{p,q}_{GS}(\scrA)
=
\prod_{\tau \in \cN_p(\frakU)}\Hom_\bfR(\scrA(c\tau)^{\otimes q}, \scrA(d\tau)),
\end{align*}
where
we regard
$\scrA(d\tau)$
as an
$\scrA(c\tau)$-bimodule
via
$f^\tau \colon \scrA(c\tau) \to \scrA(d\tau)$.
When
$q$
is fixed,
we have
\begin{align*}
\bfC^{\bullet, q}_{GS}(\scrA)
=
\bfC_{simp}(\scrA^{\otimes q}, \scrA)
\end{align*}
endowed with the simplicial differential
$d_{simp}$
horizontally.
When
$p$
is fixed,
we have
\begin{align*}
\bfC^{p, \bullet}_{GS}
=
\prod_{\tau \in \cN_p(\frakU)} \bfC \left( \scrA(c\tau), \scrA(d\tau) \right)
\end{align*}
endowed with the product Hochschild differential
$d_{Hoch}$
vertically.
The differential 
\begin{align*}
d^n_{GS}(\scrA) \colon \bfC^n_{GS}(\scrA) \to \bfC^{n+1}_{GS}(\scrA)
\end{align*}
is defined as
$d^n_{GS}
=
(-1)^{n+1} d_{simp} + d_{Hoch}$.

A cochain
$\phi = (\phi^\tau)_\tau \in \bfC^{p,q}_{GS}(\scrA)$
is
\emph{normalized}
if
$\phi^\tau$
is normalized for each $p$-simplex
$\tau$,
and
it is
\emph{reduced}
if
$\phi^\tau = 0$
whenever
$\tau$
is degenerate.
The normalized cochains form a subcomplex
$\bar{\bfC}_{GS}(\scrA)$
of
$\bfC_{GS}(\scrA)$
called the
\emph{normalized Hochschild complex}
of
$\scrA$,
and
the normalized reduced cochains
form a subcomplex
$\bar{\bfC}^\prime_{GS}(\scrA)$
of
$\bar{\bfC}_{GS}(\scrA)$
called the
\emph{normalized reduced Hochschild complex}
of
$\scrA$.
These three complexes are quasi-isomorphic
via the inclusions.
Eliminating the bottom row from
$\bfC_{GS}(\scrA)$,
one obtains a subcomplex
$\bfC_{tGS}(\scrA)$
called the
\emph{truncated Hochschild complex}.
There is a short exact sequence
\begin{align*}
0 \to \bfC_{tGS}(\scrA) \to \bfC_{GS}(\scrA) \to \bfC_{simp}(\scrA) \to 0.
\end{align*}
Since
$\bfR$
is commutative,
one can apply
\cite[Proposition 2.14]{DLL}
to see that
the sequence splits
and
we have
\begin{align*}
\bfC_{GS}(\scrA)
=
\bfC_{tGS}(\scrA) \oplus \bfC_{simp}(\scrA).
\end{align*}
Similarly,
we have
\begin{align*}
\bar{\bfC}^\prime_{GS}(\scrA)
=
\bar{\bfC}^\prime_{tGS}(\scrA) \oplus \bfC^\prime_{simp}(\scrA).
\end{align*} 

The direct sum of the second normalized reduced Gerstenharber--Schack cohomology of
$\scrA$
classifies twisted $\bfS$-deformations of
$\scrA$
up to equivalence.
Recall that
a
\emph{twisted presheaf}
$\scrA = (\scrA, m, f, c, z)$
of $\bfR$-algebras on
$\frakU$
consists of the following data:
\begin{itemize}
\item
for each
$U \in \frakU$
an $\bfR$-algebra
$(\scrA(U), m^U)$,
\item
for each
$u \colon V \to U$
in
$\frakU$
a homomorphism of $\bfR$-algebras
$f^u \colon \scrA (U) \to \scrA (V)$,
\item
for each pair
$u \colon V \to U, \ v \colon W \to V$
in
$\frakU$
an invertible element
$c^{u,v} \in \scrA(W)$
satisfying for any
$a \in \scrA(U)$
\begin{align*}
c^{u,v} f^v (f^u (a)) = f^{uv} (a) c^{u,v}.
\end{align*}
\item
for each
$U \in \frakU$
an invertible element
$z^U \in \scrA(U)$
satisfying for any
$a \in \scrA(U)$
\begin{align*}
z^U a = f^{1_U}(a) z^U.
\end{align*}
\end{itemize}
Moreover, these data must satisfy 
\begin{align*}
\begin{gathered}
c^{u,vw} c^{v,w} = c^{uv, w} f^w (c^{u,v}), \\
c^{u, 1_V} z^V = 1, \ \ c^{1_U, u} f^u (z^U)= 1
\end{gathered}
\end{align*}
for each triple 
$u \colon V \to U, \ v \colon W \to V, \ w \colon T \to W$
in
$\frakU$.
When
$c^{u,v}, z^U$
are central for all
$u, v$
and
$U$,
we call
$\scrA$
a
\emph{twisted presheaf with central twists}
and
denote by
$|\scrA| = (\scrA, m, f)$
the underlying ordinary presheaf. 

For twisted sheaves   
$\scrA = (\scrA, m, f, c, z),
\scrA^\prime
=
(\scrA^\prime, m^\prime, f^\prime, c^\prime, z^\prime)$
of $\bfR$-algebras on
$\frakU$,
a
\emph{morphism}
$(g, h) \colon \scrA \to \scrA^\prime$
consists of the following data:
\begin{itemize}
\item
for each
$U \in \frakU$
a homomorphism of $\bfR$-algebras
$g^U \colon \scrA(U) \to \scrA^\prime(U)$,
\item
for each
$u \colon V \to U$
in
$\frakU$
an invertible element 
$h^u \in \scrA^\prime (V)$.
\end{itemize}
Moreover, these data must satisfy 
\begin{align*}
\begin{gathered}
m^{\prime V} (g^V f^u(a), h^u)
=
m^{\prime V}(h^u, f^{\prime u} (g^U (a))), \\
m^{\prime W}(h^{uv}, c^{\prime u, v})
=
m^{\prime W} (g^W (c^{u,v}), h^v, f^{\prime v} (h^u)), \\
m^{\prime U} (h^{1_U}, z^{\prime U})
=
g^U (z^U)
\end{gathered} 
\end{align*}
for all 
$u,  v$
and
$a \in \scrA(U)$.
Morphisms can be composed
and
the identity
$1_\scrA$
is given by
$g^U = 1_{\scrA(U)}$
and
$h^u =1 \in \scrA(V)$.
When
$g^U$
are isomorphisms of $\bfR$-algebras for all
$U$,
we call
$(g, h)$ an
\emph{isomorphism}.
Any twisted presheaf 
$(\scrA, m, f, c, z)$
is isomorphic to the one of the form
$(\scrA^\prime, m^\prime, f^\prime, c^\prime, 1)$.

Let 
$\scrA = (\scrA, m, f, c)$
be a twisted presheaf of $\bfR$-algebras on
$\frakU$.
A
\emph{twisted $\bfS$-deformation}
of
$\scrA$
is a twisted presheaf
\begin{align*}
\overline{\scrA}
=
(\overline{\scrA}, \overline{m}, \overline{f}, \overline{c})
=
(\scrA[\epsilon], m + \bfm \epsilon , f + \bff \epsilon, c + \bfc \epsilon)
\end{align*}
of $\bfS$-algebras 
such that
$(\overline{\scrA}(U), \overline{m}^U)$
is an $\bfS$-deformation of
$(\scrA(U), m^U)$
for each
$U \in \frakU$
with
\begin{align*}
(\bfm, \bff, \bfc)
\in
\bfC^2_{GS}(\scrA)^{\oplus l}
=
\bfC^{0,2}_{GS}(\scrA)^{\oplus l}
\oplus
\bfC^{1,1}_{GS}(\scrA)^{\oplus l}
\oplus
\bfC^{2,0}_{GS}(\scrA)^{\oplus l}.
\end{align*}
Two twisted deformations
$(\overline{\scrA}, \overline{m}, \overline{f}, \overline{c}),
(\overline{\scrA^\prime}, \overline{m^\prime}, \overline{f^\prime}, \overline{c^\prime})$
are
\emph{equivalent}
if there is an isomorphism of the form 
$(1 + \bfg \epsilon, 1 + \bfh \epsilon)$
with
\begin{align*}
(\bfg, \bfh)
\in
\bfC^1_{GS}(\scrA)^{\oplus l}
=
\bfC^{0,1}_{GS}(\scrA)^{\oplus l} \oplus \bfC^{1,0}_{GS}(\scrA)^{\oplus l}.
\end{align*}
We denote by
$\Def^{tw}_\scrA(\bfS)$
the set of equivalence classes of twisted $\bfS$-deformations of
$\scrA$.

When
$c = 1$,
A
\emph{presheaf $\bfS$-deformation}
of
$\scrA$
is a twisted $\bfS$-deformation with
$\bfc = 0$.
Two presheaf deformations
$(\overline{\scrA}, \overline{m}, \overline{f}),
(\overline{\scrA^\prime}, \overline{m^\prime}, \overline{f^\prime})$
are
\emph{equivalent}
if there is an isomorphism of the form 
$1 + \bfg \epsilon$
with
$\bfg \in \bfC^{0,1}_{GS} (\scrA)^{\oplus l}$.
We denote by
$\Def^{psh}_\scrA(\bfS)$
the set of equivalence classes of presheaf $\bfS$-deformations of
$\scrA$.

\begin{lem} {\rm{(}\cite[Theorem 2.21]{DLL}\rm{)}} \label{lem:CLASStw} 
Let
$(\scrA, m, f)$
be a presheaf of $\bfR$-algebras on
$\frakU$.
Then there is a bijection
\begin{align*}
H^2 \bar{\bfC^\prime}_{GS}(\scrA)^{\oplus l}
\to
\Def^{tw}_\scrA(\bfS), \ 
(\bfm, \bff, \bfc)
\mapsto
(\scrA[\epsilon], m + \bfm \epsilon , f + \bff \epsilon, c + \bfc \epsilon), \
(\bfm, \bff, \bfc) \in Z^2 \bar{\bfC}^\prime_{GS}(\scrA)^{\oplus l}.
\end{align*}
Another cocycle
$(\bfm^\prime, \bff^\prime, \bfc^\prime)
\in
Z^2 \bar{\bfC}^\prime_{GS}(\scrA)^{\oplus l}$
maps to an equivalent deformation
if and only if
there is an element
$(\bfg, \bfh) \in \bar{\bfC}^{\prime 1}_{GS}(\scrA)^{\oplus l}$
satisfying
$(\bfm^\prime, \bff^\prime, \bfc^\prime) - (\bfm, \bff, \bfc)
=
d_{GS}(\bfg, \bfh)$.
In particular,
there is a bijection
\begin{align*}
H^2 \bar{\bfC^\prime}_{tGS}(\scrA)^{\oplus l}
\to
\Def^{psh}_\scrA(\bfS), \ 
(\bfm, \bff)
\mapsto
(\scrA[\epsilon], m + \bfm \epsilon , f + \bff \epsilon).
\end{align*}
\end{lem}

\subsection{Hodge decomposition}
In the group algebra
$\bQ S_n$
of the $n$-th symmetric group
$S_n$,
there is a collection of pairwise orthogonal idempotents
$e_n(r)$
for 
$1 \leq r \leq n$
such that
$\sum^n_{r=1} e_n(r) = 1$
\cite[Theorem 1.2]{GS87}.
Put
$e_n(0) = 0$,
$e_0(0) = 1 \in \bQ$,
and
$e_n(r) = 0$
for
$r > n$.
Let
$A$
be a commutative $\bfR$-algebra
and
$M$
a symmetric $A$-bimodule.
The subcomplex
$\bfC(A, M)_r \subset \bfC(A, M)$
whose $n$-th term is
$\bfC(A, M)e_n(r)$
gives rise to a Hodge decomposition
\begin{align*}
\bfC(A, M) = \bigoplus_{r \in \bN} \bfC(A, M)_r.
\end{align*}

Assume that
$\scrA$
is a presheaf of commutative $\bfR$-algebras.
Then the Hodge decomposition
\begin{align*}
\Hom(\scrA(c\tau)^{\otimes q}, \scrA(d\tau))
=
\bigoplus^q_{r=0} \Hom(\scrA(c\tau)^{\otimes q}, \scrA(d\tau))_r
\end{align*}
induces a decomposition of the double complex
$\bfC^{n-q, q}_{GS}(\scrA)$
preserved by
$d_{Hoch}$
and
$d_{simp}$.
Hence one obtains a
\emph{Hodge decomposition}
\begin{align*}
\bfC_{GS}(\scrA) = \bigoplus_{r \in \bN}\bfC_{GS}(\scrA)_r.
\end{align*} 
Taking cohomology yields a decomposition for
$H^\bullet \bfC_{GS}(\scrA)$.

Assume further the following.
\begin{itemize}
\item
The restriction map
$f^u \colon \scrA(U) \to \scrA(V)$
is a flat epimorphism of rings for each 
$u \colon V \to U$.
\item
The algebra
$\scrA(U)$
is essentially of finite type
and
smooth $\bfR$-algebra for each
$U$.
\end{itemize}
Recall that
a homomorphism of rings is called an epimorphism
if it is an epimorphism in the category of noncommutative rings.
For instance,
every surjective homomorphism of commutative rings is an epimorphism.
Then one obtains the
\emph{presheaf of differential}
$\Omega_\scrA \colon \frakU^{op} \to \Mod(\scrA)$
with
$\Omega_\scrA(U) = \Omega_{\scrA(U)/\bfR}$.
Since we have a canonical isomorphism
$\scrA(V) \otimes_{\scrA(U)}\Omega_{\scrA(U)}
\cong
\Omega_{\scrA(V)}$
by the first additional assumption,
the induced restriction maps
$\scrT_{\scrA(U)/\bfR} \to \scrT_{\scrA(V)/\bfR}$
yield the
\emph{tangent presheaf}
$\scrT_\scrA \colon \frakU^{op} \to \Mod(\scrA)$
with
$\scrT_\scrA(U) = \scrT_{\scrA(U)/\bfR}$.
From the second additional assumption it follows that
antisymmetrizations
$\wedge^n \scrT_{\scrA(U)} \to H^n \bfC(\scrA(U))$
are isomorphisms.

\begin{lem} {\rm{(}\cite[Theorem 3.3]{DLL}\rm{)}} \label{lem:GSsimp}
Let
$\frakU$
be a small category
and
$\scrA \colon \frakU^{op} \to \CAlg(\bfR)$
a presheaf of commutative algebras.
Assume that
the algebra
$\scrA(U)$
is essentially of finite type
and
smooth $\bfR$-algebra for each
$U$.
Assume further that
the restriction map
$f^u \colon \scrA(U) \to \scrA(V)$
is a flat epimorphism of rings for each 
$u \colon V \to U$.
Then there is a canonical bijection
\begin{align} \label{eq:GSsimp} 
H^n\bfC_{GS}(\scrA)
=
\bigoplus^n_{r=0} H^n\bfC_{GS}(\scrA)_r
\cong
\bigoplus_{p+q=n} H^p(\frakU, \wedge^q \scrT_\scrA).
\end{align}
\end{lem}

From the proof,
one sees that
any Gerstenharber--Shack cohomology class
$\frakc_{GS}$
is represented by a normalized reduced
\emph{decomposable}
cocycle
$\theta_{0,n}, \theta_{1, n-1}, \ldots, \theta_{n,0}$
in the sense that
$\theta_{n-r, r}$
are reduced
and
belong to
$\bar{\bfC}^{n-r, r} (\scrA)_r$.
Each
$\theta_{n-r, r} = (\theta^\tau_{n-r, r})_{\tau \in \cN_{n-r}(\frakU)}$
lifts to a unique simplicial cocycle
$\Theta_{n-r, r}
=
(\Theta^\tau_{n-r, r})_{\tau \in \cN_{n-r}(\frakU)}
\in
\bfC^{\prime n-r}_{simp}(\wedge^r \scrT_\scrA)$.
The image of
$\frakc_{GS}$
under the bijection is the cohomology class
$\frakc_{simp}$
represented by
$\Theta_{0,n}, \Theta_{1, n-1}, \ldots, \Theta_{n,0}$.

\subsection{Comparison with relative Hochschild cohomology}
We describe the relation between
simplicial cohomology
and
$\mathrm{\breve{C}ech}$ cohomology
for a presheaf
$\scrF \colon \frakU^{op} \to \Mod(\bfR)$
in the case where
$\frakU$
is a poset with binary meets.
We use the symbol
$\cap$
to denote meets in
$\frakU$.
For a $p$-sequence
$\tau
=
(U^\tau_0, U^\tau_1, \ldots, U^\tau_p) \in \frakU^{p+1}$
we denote by
$\cap \tau$
the meet of all coordinates of
$\tau$.
The
$\mathrm{\breve{C}ech}$ complex
$\breve{\bfC}(\scrF)$
of
$\scrF$
has
\begin{align*}
\breve{\bfC}^p(\scrF)
=
\prod_{\tau \in \frakU^{p+1}} \scrF(\cap \tau)
\end{align*}
as the $p$-th term with the usual differentials.
A
$\mathrm{\breve{C}ech}$ cochain
$\psi = (\psi^\tau)_\tau$
is
\emph{alternating}
if
$\psi^\tau = 0$  
whenever two coordinates of
$\tau$
are equal,
and
$\psi^{\tau s} = (-1)^{\sign (s)} \psi^\tau$
for any permutation
$s$
of the set
$\{ 0,1, \ldots, p\}$.
Here,
we regard
$\tau$
as a set theoretic map
$\{0,1, \ldots, p\} \to \frakU$.
The alternating
$\mathrm{\breve{C}ech}$ cochains
form a subcomplex
$\breve{\bfC}^\prime(\scrF)$
which is quasi-isomorphic to
$\breve{\bfC}(\scrF)$
via the inclusion.

To a $p$-sequence
$\tau$,
one associates a $p$-simplex
\begin{align*}
\bar{\tau}
=
(d \bar{\tau} = \cap^p_{j=0}U^\tau_j \to \cap^p_{j=1}U^\tau_j \to \cdots \to \cap^p_{j=p-1}U^\tau_j \to U^\tau_p = c \bar{\tau}). 
\end{align*}  
Conversely,
any $p$-simplex
$\mu$
can be regarded as a $p$-sequence
$\tilde{\mu}$
by forgetting the inclusions.
Define a map
$\delta_i \colon \frakU^{p+1} \to \frakU^p$
for
$i = 1, \ldots, p$
as
\begin{align*}
\delta_i \tau
=
(U^\tau_0, \ldots, U^\tau_{i-2}, U^\tau_{i-1} \cap U^\tau_i, U^\tau_{i+1}, \ldots, U^\tau_p).
\end{align*}
There are morphisms
$\iota \colon \bfC^\prime_{simp} \to \breve{\bfC}^\prime (\scrF), \
\pi \colon \breve{\bfC}^\prime (\scrF) \to \bfC^\prime_{simp}$
of complexes defined as
\begin{align*}
\begin{gathered}
\iota(\phi)^\tau
=
\sum_{s \in \frakS_{p+1}} (-1)^{\sign(s)} \phi^{\overline{\tau s}}, \ 
\phi \in \bfC^\prime_{simp}(\scrF), \
\tau \in \frakU^{p+1}, \\
\pi(\psi)^\mu
=
\psi^{\tilde{\mu}}, \ 
\psi \in \breve{\bfC}^\prime(\scrF), \
\mu \in \cN_p(\frakU),
\end{gathered}
\end{align*}
which induce mutually inverse isomorphisms between
$H^\bullet (\frakU, \scrF)$
and
$\breve{H}^\bullet (\frakU, \scrF)$
\cite[Lemma 3.9]{DLL}.

Now,
for a smooth proper $\bfR$-scheme
$X$
we give an alternative description of the relative Hochschild cohomology.
As explained above,
we have
$HH^\bullet (X/\bfR) \cong HT^\bullet (X/\bfR)$.
Choose a finite affine open cover
$\frakU$
closed under intersections.
By definition
$\frakU$
is semi-separating,
i.e.,
$\frakU$
is closed under finite intersections. 
For every quasi-coherent sheaf 
$\scrF$
on
$X$,
one can apply
\cite[Lemma 3.9]{DLL}
and
Leray's theorem
\cite[Theorem 4.5]{Har}
to obtain
\begin{align} \label{eq:simpsheaf}
H^\bullet (\frakU, \scrF |_\frakU)
\cong
\breve{H}^\bullet (\frakU, \scrF |_\frakU)
\cong
\breve{H}^\bullet (\frakU, \scrF)
\cong
H^\bullet (X, \scrF)
\end{align}
with
$\scrF |_\frakU$
regarded as a presheaf on
$\frakU$.
Since
$X$
is smooth over
$\bfR$
and
open immersions
$V \hookrightarrow U$
in
$\frakU$
define flat epimorphisms
$\scrO_X(U) \to \scrO_X(V)$,
combining
\pref{eq:GSsimp}
with
\pref{eq:simpsheaf},
we obtain

\begin{lem} {\rm{(}\cite[Corollary 3.4]{DLL}\rm{)}} 
Let
$X$
be a smooth proper $\bfR$-scheme
with a finite affine open cover
$\frakU$
closed under intersections. 
Let
$\scrO_X|_\frakU, \scrT_{X/\bfR}|_\frakU$
be the restrictions of
$\scrO_X, \scrT_{X/\bfR}$
to
$\frakU$
respectively.
Then there are canonical isomorphisms
\begin{align} \label{eq:GSHT} 
H^n \bfC_{GS}(\scrO_X|_\frakU)
=
\bigoplus^n_{r=0} H^n\bfC_{GS}(\scrO_X|_\frakU)_r
\cong
\bigoplus_{p+q=n} H^p(\frakU, \wedge^q \scrT_{X/\bfR}|_\frakU)
\cong
HH^n (X/\bfR),
\end{align}
where the first isomorphism respects the Hodge decomposition. 
\end{lem}

\section{Deformations of relatively smooth proper schemes}
In this section,
we review the classical deformation theory of schemes.
The main reference is
\cite{Har10}.
We explain how deformations of smooth proper $\bfk$-varieties extend to Toda's construction
\cite{Tod},
which can be adapted to deformations of relatively smooth proper schemes along square zero extensions in a straightforward way.
When the original scheme is a deformation of a higher dimensional Calabi--Yau manifold,
Toda's construction gives the category of quasi-coherent sheaves on deformations of the Calabi--Yau manifold.

\subsection{Deformations of schemes}
Let
$X$
be a
$\bfk$-scheme
and
$A$
a local artinian $\bfk$-algebra
with residue field
$\bfk$.
An
\emph{$A$-deformation}
of
$X$
is a pair
$(X_A, i_A)$,
where
$X_A$
is a scheme flat over
$A$
and
$i_A \colon X \hookrightarrow X_A$
is a closed immersion
such that
the induced map
$X \to X_A \times_A \bfk$
is an isomorphism.
Two deformations
$(X_A, i_A), (X^\prime_A, i^\prime_A)$
are
\emph{equivalent}
if there is an $A$-isomorphism
$X_A \to X^\prime_A$
compatible with
$i_A, i^\prime_A$.
The deformation functor
\begin{align*}
\Def_X \colon \Art \to \Set
\end{align*}
sends each
$A \in \Art$
to the set of equivalence classes of $A$-deformations of
$X$.

Assume that
$X$
is projective over
$\bfk$.
Then
$\Def_X$
satisfies Schlessinger's criterion
and
there exists
a
\emph{miniversal formal family}
$(R, \xi)$
for
$\Def_X$,
where
$R$
is a complete local noetherian $\bfk$-algebra with residue field
$\bfk$,
and
$\xi = \{ \xi_n \}_n$
belongs to the limit
\begin{align*}
\widehat{\Def}_X (R)
=
\displaystyle \lim_{\longleftarrow} \Def_X (R / \frakm^n_R)
\end{align*}
of the inverse system
\begin{align*}
\cdots \to
\Def_X (R/\frakm_R^{n+2})
\to \Def_X (R/\frakm_R^{n+1})
\to \Def_X (R/\frakm_R^n)
\to \cdots
\end{align*}
induced by the natural quotient maps
$R/\frakm^{n+1}_R \to R/\frakm^n_R$.
The formal family
$\xi$
corresponds to a natural transformation 
\begin{align*}
h_R = \Hom_{\operatorname{\bfk-alg}} (R, -) \to \Def_X,
\end{align*}
which sends each
$g \in h_R (A)$
factorizing through
$R \to R / \frakm^{n+1}_R \xrightarrow{g_n} A$
to
$\Def_X (g_n) (\xi_n)$.

Let
$X_n$
be the schemes which define
$\xi_n$.
There is a noetherian formal scheme
$\scrX$
over
$R$
such that
$X_n \cong \scrX \times_R R / \frakm^{n+1}_R$
for each
$n$.
By abuse of notation,
we use the same symbol
$\xi$
to denote
$\scrX$.
Thus any scheme which defines an equivalence class
$[X_A, i_A]$
can be obtained as the pullback of
$\xi$
along some morphism of noetherian formal schemes
$\Spec A \to \Spf R$.
If
$X$
has no infinitesimal automorphisms
which restrict to the identity of
$X$,
then every equivalence class
$[X_A, i_A]$
becomes just a deformation
$(X_A, i_A)$
and
we have a natural isomorphism
$h_R \cong \Def_X$.
In this case,
we call
$\Def_X$
\emph{prorepresentable}
and
$(R, \xi)$
a
\emph{universal formal family}
for
$\Def_X$.

\subsection{Algebraization}
Let
$X$
be a projective $\bfk$-variety.
We call a miniversal formal family
$(R, \xi)$
for
$\Def_X$
\emph{effective}  
when there exists a scheme
$X_R$
flat and of finite type over
$R$
whose formal completion along the closed fiber
$X$
is isomorphic to
$\xi$.
By
\cite[Theorem III5.4.5]{GD61}
the family
$(R, \xi)$
is effective
if deformations of any invertible sheaf on
$X$
are unobstructed.
Note that
this is the case,
for instance,
if we have
$H^2 (\scrO_X) = 0$.
From the proof,
one sees that
$X_R$
is projective over
$R$.
We will call such
$X_R$
an
\emph{effectivization}
of
$\xi$.

The deformation functor
$\Def_X$
can naturally be extended to a functor defined on the category
$\Alg^{aug}(\bfk)$
of augmented noetherian $\bfk$-algebras.
By abuse of notation,
we use the same symbol
$\Def_X$
to denote the extended functor,
which sends each
$(\bfP, \frakm_\bfP) \in \Alg^{aug}(\bfk)$
to the set of equivalence classes of deformations over
$(\bfP, \frakm_\bfP)$.
Since the functor
$\Def_X$
is locally of finite presentation,
by
\cite[Theorem 1.6]{Art69b}
the miniversal formal family is
\emph{algebraizable},
i.e.,
there exists a triple
$(S, s, X_S)$
where
$S$
is an algebraic $\bfk$-scheme
with a distinguished closed point
$s \in S$,
and
$X_S$
is a flat and of finite type $S$-scheme
whose formal completion along the closed fiber
$X$
over
$s$
is isomorphic to
$\xi$. 
We call the scheme
$X_S$
a
\emph{versal deformation}
over
$S$.
When there exists a versal deformation,
we say that
the miniversal formal family
$(R, \xi)$
is
\emph{algebraizable}.

\subsection{Deformations of higher dimensional Calabi--Yau manifolds}
\label{subsec:DefCY}
Here,
we focus on a special case
where
several interesting results hold.
Let
$X_0$
be a Calabi--Yau manifold of dimension more than two.
Then the deformation functor
$\Def_{X_0}$
has an effective universal formal family
$(R, \xi)$.
Since deformations of Calabi--Yau manifolds are unobstructed,
the complete local noetherian ring
$R$
is regular
and
we have
\begin{align*}
R \cong \bfk \llbracket t_1, \ldots, t_d \rrbracket
\end{align*}
with
$d =\dim_{\bfk} H^1 (\scrT_{X_0})$.
Every $A$-deformation of
$X_0$
is smooth projective over
$A$,
as we have

\begin{lem} {\rm{(}\cite[Lemma 2.4]{Mora}\rm{)}}
The effectivization
$X_R$
for
$(R, \xi)$
is regular
and
the morphism
$\pi_R \colon X_R \to \Spec R$
is smooth of relative dimension
$\dim X_0$.
\end{lem}

Now,
we briefly recall the construction of
$X_S$.
Consider the extended functor 
\begin{align*}
\Def_{X_0}
\colon
\Alg^{aug}(\bfk)
\to
\Set.
\end{align*}
Fix an isomorphism
$R \cong \bfk \llbracket t_1, \ldots, t_d \rrbracket$.
Let
$T = \bfk [ t_1, \ldots, t_d ]$
and
$t \in \Spec T$
be the closed point corresponding to maximal ideal
$(t_1, \ldots, t_d)$.
There is a filtered inductive system
$\{ R_i \}_{i \in I}$
of finitely generated $T$-subalgebras of
$R$
whose colimit is
$R$.
Since
$\Def_{X_0}$
is locally of finite presentation,
$[ X_R, i_R ]$
is the image of some element
$\zeta_i
\in
\Def_{X_0} \left( \left( R_i, \frakm_{R_i} \right) \right)$
by the canonical map
$\Def_{X_0} \left( \left( R_i, \frakm_{R_i} \right) \right)
\to
\Def_{X_0} (R)$.
By
\cite[Corollary 2.1]{Art69a}
there exists an \'etale neighborhood
$\Spec S$
of
$t$
in
$\Spec T$
with first order approximation
$\varphi \colon R_i \to S$
of
$R_i \hookrightarrow R$.
Let
$[ X_S, i_S ]$
be the image of
$\zeta_i$
by the map
$\Def_{X_0} (\varphi)$.
From miniversality of
$(R, \xi)$,
it follows that
the formal completion of
$X_S$
along the closed fiber
$X_0$
over
$s \in \Spec S$
is isomorphic to
$\xi$,
where
$s$
is the distinguished closed point mapping to
$t$.
By construction,
$\Spec S$
is a nonsingular affine $\bfk$-variety 
and
$X_S$
is flat of finite type over
$S$.
Exploiting inherited
smoothness
and
projectivity
of
$X_R$
by terms in the projective system
$\{ X_{R_i } \}_{i \in I}$
for sufficiently large indices,
one can show

\begin{lem} {\rm{(}\cite[Lemma 2.3]{Mora}\rm{)}} \label{lem:smprojVD}
Let
$X_0$
be a Calabi--Yau manifold of dimension more than two.
Then there exists a nonsingular affine $\bfk$-variety
$\Spec S$
with a versal deformation
$X_S$
which is smooth projective of relative dimension
$\dim X_0$
over
$S$.
\end{lem}

\subsection{$T^i$ functors}
Let 
$A \to B$
be a ring homomorphism
and
$M$
a $B$-module.
Define the groups
$T^i (B/A, M)$
for
$i = 0, 1, 2$
as the $i$-th cohomology of the complex
$\Hom_B(L_\bullet, M)$,
where
\begin{align*}
L_\bullet = L_2 \xrightarrow{d_2} L_1 \xrightarrow{d_1} L_0
\end{align*}
is the
\emph{cotangent complex}.
When the ring homomorphism
$A \to B$
is a surjection with kernel
$J$,
$L_\bullet$
is given as follows.
Choose a free $A$-module
$P$
and
a surjection
$j \colon P \to J$
with kernel
$Q$.
We have two short exaxt sequences
\begin{align*}
0 \to J \to A \to B \to 0, \
0 \to Q \to P \xrightarrow{j} J \to 0.
\end{align*}
Let
$P_0$
be the submodule of
$P$
generated by all relations of the form
$j(a)b - j(b)a$
for
$a,b \in P$.
From
$j(P_0)=0$
it follows
$P_0 \subset Q$. 
Take
$L_2 = Q / P_0$,
$L_1 = P \otimes_A B$,
and
$L_0 = 0$.
Note that
$L_2$
is a $B$-module.
Indeed,
for
$a \in J$
there is an element
$a^\prime \in P$
such that
$a = j(a^\prime)$.
Then we have
$ax \equiv j(x)a^\prime \equiv 0$
modulo
$P_0$
for
$x \in Q$.
The differential
$d_2 \colon L_2 \to L_1$
is the map induced by the inclusion
$Q \to P$
and
$d_1 = 0$.
By
\cite[Lemma 3.2]{Har10}
the $B$-modules
$T^i (B/A, M)$
do not depend on the choice of
$P$
up to isomorphism.

\begin{lem} {\rm{(}\cite[Theorem 3.4]{Har10}\rm{)}}
Let
$A \to B$
be a homomorphism of rings.
Then
\begin{align*}
T^i (B/A, -) \colon \Mod(B) \to \Mod(B), \ i = 0, 1, 2
\end{align*}
define covariant additive functor.
\end{lem}

The construction of
$T^i$
functors is compatible with localization
and
one obtains sheaves
$\scrT^i(X/Y, \scrF), \ i = 0, 1, 2$
for any morphism of $\bfk$-schemes
$f \colon X \to Y$
and
any quasi-coherent $\scrO_X$-module
$\scrF$
\cite[Exercise 3.5]{Har10}.
The sections of
$\scrT^i(X/Y, \scrF)$
over
$U = \Spec B \subset f^{-1}(V)$
give
$T^i(B/A, M)$,
where
$V = \Spec A \subset Y$
and
$\scrF |_U = \tilde{M}$
for some $B$-module
$M$.

\subsection{Infinitesimal extension of schemes}
Let
$X$
be a scheme of finite type over
$\bfR$
and
$\scrF$
a coherent sheaf on
$X$.
An
\emph{infinitesimal extension}
of
$X$
by
$\scrF$
is a pair
$(Y, \frakI)$,
where
$Y$
is a scheme of finite type over
$\bfS$
and
$\frakI \subset \scrO_Y$
is an ideal sheaf
such that
$\frakI^2 = 0$,
$(Y, \scrO_Y / \frakI) \cong (X, \scrO_X)$,
and
$\frakI \cong \scrF$
as an $\scrO_X$-module.
Two infinitesimal extensions
$(Y, \frakI), (Y^\prime, \frakI^\prime)$
are
\emph{equivalent}
if there is an isomorphism
$\scrO_Y \to \scrO_{Y^\prime}$
which makes the diagram
\begin{align*} 
\begin{gathered}
\xymatrix{
0 \ar[r] & \scrF \ar[r] \ar[d]^{\id} & \scrO_Y \ar[r] \ar[d] & \scrO_X \ar[r] \ar[d]^{\id} & 0 \\
0 \ar[r] & \scrF \ar[r] & \scrO_{Y^\prime} \ar[r] & \scrO_X \ar[r] & 0
}
\end{gathered}
\end{align*}
commute.
The trivial extension is a sheaf
$\scrO_X \oplus \scrF$
of abelian group endowed with the ring structure by
\begin{align*}
(a, f) \cdot (a^\prime, f^\prime)
=
(aa^\prime, af^\prime + a^\prime f).
\end{align*}
 
Assume that
$X$
is smooth proper over
$\bfR$.
Recall that
for the square zero extension
\begin{align} \label{eq:ses31}
0 \to \bfI \to \bfS \to \bfR \to 0
\end{align}
we have
$\bfI \cong \bfR^{\oplus l}$
as an $\bfR$-module.
Note that
given an infinitesimal extension
$(Y, \frakI)$
of
$X$
by
$\scrO^{\oplus l}_{X}$, 
$Y$
is flat over
$\bfS$
since
$\scrO_X$
is flat over
$\bfR$
and
$\scrO^{\oplus l}_X \to \scrO_Y$
is injective
\cite[Proposition 2.2]{Har10}.
Below,
we collect fundamental results necessary to describe the relation between
deformations
and
extensions
of schemes.

\begin{lem} {\rm{(}\cite[Exercise 4.7]{Har10}\rm{)}} \label{lem:relativeILP}
Let
$X$
be a smooth $\bfR$-scheme
and
$g \colon Y \to X$
a morphism from an affine $\bfR$-scheme
$Y$
to
$X$,
and
$i_\bfS \colon Y \hookrightarrow Y^\prime$
an $\bfS$-deformation of
$Y$. 
Then
$g$
lifts to a morphism
$h \colon Y^\prime \to X$
such that
$h \circ i_\bfS = g$.
\end{lem}

\begin{lem} {\rm{(}\cite[Proposition 3.6, Exercise 5.2]{Har10}\rm{)}} \label{lem:AUTOext}
Let
$A \to B$
be a homomorphism of rings,
$M$
a $B$-module,
and
$B^\prime$
an extension of
$B$
by
$M$.
Then the automorphism group of
$B^\prime$
is given by 
\begin{align*}
T^0(B/A, M)
=
\Hom_B(\Omega_{B/A}, M)
=
\Der_A(B, M).
\end{align*}
\end{lem}

\begin{lem} {\rm{(}\cite[Theorem 5.1]{Har10}\rm{)}} \label{lem:CLASSext}
Let
$A \to B$
be a homomorphism of rings
and
$M$
a $B$-module.
Then there is a bijection
between
the set of equivalence classes of
$B$
by
$M$
and
the group
$T^1(B/A, M)$.
The trivial extension corresponds to the zero element.
\end{lem}

\begin{lem} {\rm{(}\cite[Theorem 4.11]{Har10}\rm{)}} \label{lem:smoothness}
Let
$f \colon X \to Y$
be an of finite type morphism of noetherian $\bfk$-schemes.
Then
$f$
is smooth
if and only if
it is flat
and
$\scrT^1(X/Y, \scrF) = 0$
for every coherent $\scrO_X$-module
$\scrF$.
\end{lem}

Now,
we are ready to show relevant results to our setting.

\begin{lem} \label{lem:LT}
Let
$X$
be a smooth separated $\bfR$-scheme.
Then every $\bfS$-deformation
$(Y, j)$
of
$X$
is locally trivial.
\end{lem}
\begin{proof}
Since
$Y$
is flat over
$\bfS$,
\pref{eq:ses31}
induces a short exact sequence 
\begin{align*}
0 \to
\scrO^{\oplus l}_X
\to
\scrO_Y
\to
\scrO_X
\to 0,
\end{align*}
which defines
equivalence classes of infinitesimal extensions of coordinate rings on affine open subschemes of
$X$.
Let
$i_\bfS \colon \Spec B \hookrightarrow \Spec A$
be the induced deformation of any affine open subscheme.
Since
$\Spec B$
is smooth over
$\bfR$,
by
\pref{lem:relativeILP}
the identity
$\Spec B \to \Spec B$
lifts to a morphism
$h \colon \Spec A \to \Spec B$
such that
$h \circ i_\bfS = \id$.
The lift
$h$
induces a morphism
$\Spec A \to \Spec B \times_\bfR \bfS$
of schemes flat of finite type over
$\bfS$.
Now,
one can apply 
\cite[Exercise 4.2]{Har10}
to see that
the induced morphism is an isomorphism.
\end{proof}

\begin{prop} \label{prop:CLASSscheme}
Let
$X$
be a smooth separated $\bfR$-scheme.
Then there is a bijection
\begin{align*}
\Def_X(\bfS) \cong H^1(X, \scrT_{X/\bfR})^{\oplus l},
\end{align*}
where
$\scrT_{X/\bfR}$
is the relative tangent sheaf on
$X$.
\end{prop}
\begin{proof}
Let
$(Y, j)$
be an $\bfS$-deformation of
$X$.
Take an affine open cover
$\frakU = \{ U_i \}_{i \in I}$
of
$X$.
By
\pref{lem:LT}
we may assume that
the induced deformations
$U_i \hookrightarrow V_i \subset Y$
are trivial.
Choose isomorphisms
$\varphi_i \colon U_i \times_\bfR \bfS \to V_i$
and
write
$\varphi_{ij}$
for the composition
$\varphi^{-1}_j \circ \varphi_i$
on
$U_{ij} \times_\bfR \bfS$,
where the intersections
$U_{ij} = U_i \cap U_j $
are again affine
as
$X$
is separated over
$\bfR$.
Let
$\Spec B = U_{ij}$
and
$\Spec A = U_{ij} \times_\bfR \bfS$.
According to
\pref{lem:AUTOext},
the set of automorphisms of extensions
$\tilde{A}$
of
$\tilde{B}$
by
$\tilde{B}^{\oplus l}$
bijectively corresponds to
$T^0(B/\bfR, B^{\oplus l})
\cong
\Hom (\Omega_{B/\bfR}, B)^{\oplus l}$. 
Then
$\{ \varphi_{ij} \}_{i,j \in I}$
define a collection 
$\{ \theta_{ij} \}_{i,j \in I}$
of sections 
$\theta_{ij}
\in
H^0(U_{ij}, \scrT_{X/\bfR})^{\oplus l}$
on
$U_{ij}$.
One checks
$\theta_{ij} + \theta_{jk} + \theta_{ki} = 0$
and
$\{ \theta_{ij} \}_{i,j \in I}$
is a $\rm{\breve{C}ech}$ $1$-cocycle
with respect to
$\frakU$. 
Another choice of isomorphisms
$\varphi^\prime_i
\colon
U_i \times_\bfR \bfS \to V_i$
yields a collection
$\{ \varphi^\prime_{ij} \}_{i,j \in I}$
of automorphisms
such that
$\varphi^\prime_{ij}
=
(\varphi^{-1}_j \circ \varphi^\prime_j)^{-1}
\circ
\varphi_{ij}
\circ
(\varphi^{-1}_i \circ \varphi^\prime_i)$.
It follows 
$\theta^\prime_{ij}
=
\theta_{ij} + \alpha_i - \alpha_j$
for some sections
$\alpha_i
\in
H^0(U_i, \scrT_{X/\bfR})^{\oplus l}$.
Thus we obtain a well defined assignment 
\begin{align*}
\Def_X(\bfS) \to H^1(X, \scrT_{X/\bfR})^{\oplus l}, \
[Y, j]
\mapsto
\{ \theta_{ij} \}_{i,j \in I},
\end{align*}
as 
$\{ \theta_{ij} \}_{i,j \in I}$
does not depends on $\frakU$.

Conversely,
an element of
$H^1(X, \scrT_{X/\bfR})^{\oplus l}$
can be represented by $\rm{\breve{C}ech}$ $2$-cocycle
$\{ \theta_{ij} \}_{i,j \in I}$
with respect to
$\frakU$. 
As explained above,
the cocycle define automorphisms of the trivial deformations
$U_{ij} \times_\bfR \bfS$,
which glue to yield a global deformation
$(Y^\prime, j^\prime)$
of
$X$.
Clearly,
this construction gives the inverse assignment.
\end{proof}

\begin{cor} \label{cor:Bridge31}
There is a canonical bijection between
$\Def_X(\bfS)$
and
the set of equivalence classes of infinitesimal extensions of
$X$
by
$\scrO^{\oplus l}_X$.
\end{cor}
\begin{proof}
By
\pref{lem:CLASSext}
and 
\pref{lem:smoothness}
any extension of
$X$
by
$\scrO^{\oplus l}_X$
is locally trivial.
Then
due to
\pref{lem:AUTOext}
the claim follows from the same argument as in the proof of
Proposition
\pref{prop:CLASSscheme}.
\end{proof}

\subsection{Toda's construction}
Let
$X_0$
be a smooth projective $\bfk$-variety.
In
\cite{Tod}
Toda constructed the category of
$\tilde{\alpha}$-twisted sheaves on the noncommutative scheme 
$(X_0, \scrO^{(\beta, \gamma)}_{X_0})$
over the ring of dual numbers for each
$[\phi_0] \in HT^2 (X_0) $
represented by a cocycle
\begin{align*}
(\alpha_0, \beta_0, \gamma_0)
\in
H^2(\scrO_{X_0}) \oplus H^1(\scrT_{X_0}) \oplus H^0 (\wedge^2 \scrT_{X_0}).
\end{align*}
Here,
we apply his idea to a smooth proper $\bfR$-scheme
$X$
and
$[\phi] \in HT^2 (X/\bfR)^{\oplus l}$
represented by
\begin{align*}
(\alpha, \beta, \gamma)
=
\left(
(\alpha^1, \ldots, \alpha^l),
(\beta^1, \ldots, \beta^l),
(\gamma^1, \ldots, \gamma^l)
\right)
\in
H^2(\scrO_X)^{\oplus l}
\oplus
H^1(\scrT_{X/\bfR})^{\oplus l}
\oplus
H^0 (\wedge^2 \scrT_{X/\bfR})^{\oplus l}.
\end{align*}
Take a finite affine open cover
$\frakU = \{ U_i \}^N_{i=1}$
of
$X$
and
let
$\frakU \times_\bfR \bfS
=
\{ U_i \times_\bfR \bfS \}^N_{i=1}$.
Consider the extension of
$X$
by
$\scrO^{\oplus l}_X$
whose equivalence class corresponds to
$\beta$,
giving rise to an classical $\bfS$-deformation
$X_\beta$
of
$X$
by Corollary
\pref{cor:Bridge31}.
We modify the multiplication on
$\scrO_X \oplus \cC(\frakU, \scrO^{\oplus l}_X)$
as 
\begin{align*}
&(a, \{ b^1_i \}, \ldots, \{ b^l_i \})
*_\gamma
(c, \{ d^1_i \}, \ldots, \{ d^l_i \}) \\
=
&(ac, \{ a d^1_i + b^1_i c + \gamma^1_i(a, c) \}, \ldots, \{ a d^l_i + b^l_i c + \gamma^l_i(a, c) \}),
\end{align*}
where
$\gamma^j
\colon
\scrO_X \times \scrO_X \to \scrO_X$
are regarded as bidifferential operators.  
We denote by
$X_{(\beta, \gamma)}
=
(X_\beta, \scrO^\gamma_{X_\beta})$
the resulting noncommutative $\bfS$-scheme.
By the standard argument,
one sees that
up to isomorphism
the scheme does not depend on the choice of
$\frakU$
and
$\mathrm{\breve{C}ech}$ representative of
$\gamma$.
From
$\alpha$
one obtains an element
\begin{align*}
\tilde{\alpha}
=
\{ 1 - \alpha^1_{ i_0 i_1 i_2} \epsilon_1 - \cdots - \alpha^l_{i_0 i_1 i_2} \epsilon_l \}_{i_0 i_1 i_2}
\in
\bfC^2(X_\beta, Z(\scrO^\gamma_{X_\beta})^*), 
\end{align*}
which is a cocycle.
Then $\tilde{\alpha}$-twisted sheaves on
$X_{(\beta, \gamma)}$
form a category
$\Mod(X_{(\beta, \gamma)}, \tilde{\alpha})$.
By the similar argument to
\cite[Lemma 1.2.3, 1.2.8]{Cal00},
one sees that
up to equivalence
the category does not depend on the choice of
$\frakU$
and
$\mathrm{\breve{C}ech}$
representative of
$\alpha$.
We denote by
$\Qch(X, \phi)$
the full abelian subcategory
spanned by $\tilde{\alpha}$-twisted quasi-coherent sheaves.

Assume that
$X$
is an $\bfR$-deformation of a higher dimensional Calabi--Yau manifold.
Then we have
\begin{align*}
HT^2(X/\bfR) = H^1(\scrT_{X/\bfR}).
\end{align*}
In this case,
Toda's construction yields nothing but  the category of quasi-coherent sheaves on the $\bfS$-deformation of
$X$
along
$\phi$.

\begin{prop} \label{prop:QchCYDef}
Let
$X_0$
be a Calabi--Yau manifold with
$\dim X_0 >2$
and
$X$
an $\bfR$-deformation of
$X_0$.
Then for every cocycle
$\phi \in HT^2(X/\bfR)^{\oplus l} = H^1(\scrT_{X/\bfR})^{\oplus l}$
we have
\begin{align*}
\Qch(X, \phi) = \Qch(X_\phi),
\end{align*}
where
$X_\phi$
is the $\bfS$-deformation of
$X$
along
$\phi$.
\end{prop}

\section{Deformations of linear and abelian categories}
In this section,
we review the deformation theory of linear and abelian categories
developed by
Lowen
and
van den Bergh in
\cite{LV06},
introducing the fundamental notion of flatness.
As explained there,
when considering only flat nilpotent deformations over a certain class of rings,
one avoids any set theoretic issue by choosing sufficiently large universe.
Moreover,
both
linear
and
abelian deformations
reduce to strict linear deformations
without affecting the deformation theory up to equivalence.
Along square zero extensions,
flat deformations of
linear
and
abelian categories
are controlled by the second Hochschild cohomology of the corresponding linear categories.

\subsection{Universes}
First,
we need to extend the Zermelo--Fraenkel axioms of the set theory
to avoid foundational issues in the deformation theory of categories.
One solution is the theory of universes
introduced by Grothendieck  
with the axiom of choice
and
the universe axiom.
A
\emph{universe}
is a set
$\cU$
with the following properties:
\begin{itemize}
\item
if
$x \in \cU$
and   
$y \in x$
then
$y \in \cU$,
\item
if
$x, y \in \cU$
then
$\{ x, y \} \in \cU$, 
\item
if
$x \in \cU$
then
the powerset
$\cP(x)$
of $x$ is in $\cU$,
\item
if
$(x_i)_{i \in I}$
is a family of objects of
$\cU$
indexed by an element of
$\cU$
then
$\bigcup_{i \in I} x_i \in \cU$,
\item
if
$U \in U$
and
$f \colon U \to \cU$
is a function
then
$\{ f(x) \ | \ x \in U \} \in \cU$.
\end{itemize}
A universe
$\cU$
containing
$\bN$
is a model for the Zermelo--Fraenkel axioms of the set theory
with the axiom of choice.
Since the known nonempty universe only contains finite sets,
the
\emph{universe axiom}
is added,
which imposes every set to be an element of a universe. 

Consider the category
$\cU-\Set$
whose objects are elements of
$\cU$
and
whose morphisms are ordinary maps between sets in
$\cU$. 
The category
$\cU-\Cat$
consists of categories
whose objects and morphisms respectively form sets 
being an element of
$\cU$. 
Similarly,
by requiring the underlying sets to belong to
$\cU$,
we obtain categories with a structure
such as
abelian groups and rings.  
We call a category 
\emph{$\cU$-small}
when its objects and morphisms respectively form sets
with the same cardinality
as an element of
$\cU$,
and
\emph{essentially $\cU$-small}
when it is equivalent to a $\cU$-small category.
A
\emph{$\cU$-category}
is a category
whose $\Hom$-sets have the the same cardinality
as an element of
$\cU$.
The axiom of choice allows us
to replace a $\cU-$category
$\scrC$
by an equivalent category
$\scrC^\prime$
with
$\Ob(\scrC) = \Ob(\scrC^\prime)$
and
$\scrC^\prime (C, D) \in \cU$
for all
$C, D \in \Ob(\scrC^\prime)$.
When
$\scrC$
is abelian
with a generator,
we call
$\scrC$
\emph{$\cU$-Grothendieck}.
Every $\cU$-Grothendieck category
$\scrC$
admits $\cU$-small colimits
and
$\cU$-small filtered colimits are exact in
$\scrC$.

Throughout the paper,
we work with a fixed universe
$\cU$
containing
$\bN$.
All the notion based on universes will be with respect to
$\cU$
and
all the related symbols will be tacitly prefixed by
$\cU$.
By taking
$\cU$
sufficiently large,
we may assume all categories to be small.
Unless otherwise specified,
we will be free from any issue caused by the choice of universes.     

\subsection{Flatness}
The notion of flatness for abelian categories
was introduced in
\cite{LV06}. 
For a while,
we temporarily drop the assumption on
$\bfR$
and
$\bfS$
imposed at the beginning of
\pref{sec:Hochsch}.
Let
$\bfR$
be a commutative ring.
An
\emph{$\bfR$-linear category}
is a category
$\fraka$
enriched over the abelian category
$\Mod(\bfR)$
of $\bfR$-modules.
Namely,
$\fraka$
is a pre-additive category 
together with a ring map
$\rho \colon \bfR \to \Nat(1_\fraka, 1_\fraka)$
inducing a ring map
$\rho_A \colon \bfR \to \fraka(A,A)$
for each
$A \in \fraka$
and
an action of
$\bfR$
on each Hom-set.

Assume that
$\bfR$
is coherent,
i.e.,
any finitely generated ideal is finitely presented as an $\bfR$-module.
Typical examples are given by noetherian rings. 
We denote by
$\mmod(\bfR)$
the full abelian subcategory of finitely presented $\bfR$-modules.
Let
$\scrC$
be an $\bfR$-linear abelian category.
We call an object
$C \in \scrC$
\emph{flat}
if the natural finite colimit preserving functor
$(-) \otimes_\bfR C \colon \mmod (\bfR) \to \scrC$
is exact,
and
\emph{coflat}
if the natural finite limit preserving functor
$\Hom_\bfR (-, C) \colon \mmod (\bfR) \to \scrC$
is exact.

An $\bfR$-linear category
$\fraka$
is
\emph{flat} 
if its Hom-sets are flat $\bfR$-modules.
Namely,
the functors
$- \otimes_\bfR \fraka(A, A^\prime)
\colon
\mmod(\bfR)
\to
\Mod(\bfR)$
are exact for all
$A, A^\prime \in \fraka$.
An $\bfR$-linear abelian category
$\scrC$
is
\emph{flat}
if for each
$Y \in \mmod(\bfR)$
the functor
$\Tor^\bfR_1(Y, -) \colon \scrC \to \scrC$
is co-effaceble,
i.e.,
for each
$C \in \scrC$
there is an epimorphism
$f \colon C^\prime \to C$
with
$\Tor^\bfR_1(Y, f) = 0$
\cite[Proposition 3.1]{LV06}.
Here,
$\Tor^\bfR_i(Y, -)$
is the left derived functor of the finite colimit preserving functor
$Y \otimes_\bfR (-) \colon \scrC \to \scrC$.
The flatness has the following characterizations
\cite[Proposition 3.3, 3.4, 3.6, 3.7]{LV06}.
\begin{itemize}
\item
$\scrC$ is flat if and only if $\scrC^{op}$ is flat.
\item
$\scrC$ is flat if and only if all injectives in $\scrC$ are coflat.
\item
$\scrC$ is flat if and only if $\Ind(\scrC)$ is flat.
\item
$\fraka$ is flat if and only if the abelian category $\Mod(\fraka)$ is flat. 
\end{itemize}
Here,
$\Ind(\scrC)$
is the category of ind-objects,
i.e.,
the full subcategory of
$\Mod(\scrC)$
consisting of left exact functors,
where
$\Mod(\scrC)$
is the category of covariant additive functors from
$\scrC$
to the category
$\Ab$
of abelian groups.
Note that
we are assuming all categories to be small in our fixed universe
$\cU$.

\subsection{Base change} \label{subsec:BC}
We fix a homomorphism  
$\theta \colon \bfS \to \bfR$
of commutative rings.
For an $\bfR$-module
$M$,
by
$\overline{M}$
we denote
$M$
regarded as an $\bfS$-module via
$\theta$.
Let
$\fraka$
be an $\bfR$-linear category.
We have the category
$\overline{\fraka}$
with
$\Ob(\overline{\fraka}) = \Ob(\fraka)$
and
$\overline{\fraka}(A, A^\prime) = \overline{\fraka(A, A^\prime)}$.
For an $\bfS$-linear category
$\frakb$,
we denote by
$\frakb \otimes_\bfS \bfR$
the $\bfR$-linear category
with
$\Ob(\frakb \otimes_\bfS \bfR) = \Ob(\frakb)$
and
$(\frakb \otimes_\bfS \bfR) (B, B^\prime)
=
\frakb(B, B^\prime) \otimes_\bfS \bfR$.
The functor
$(-) \otimes_\bfS \bfR$
is left adjoint to
$\overline{(-)}$
in the sense that 
there is a natural isomorphism
\begin{align*}
\overline{\Add(\bfR)(\frakb \otimes_\bfS \bfR, \fraka)}
\simeq
\Add(\bfS)(\frakb, \overline{\fraka})
\end{align*}
of $\bfS$-linear categories,
where
$\Add(\bfS)$
is the category of $\bfS$-linear functors.

Let
$(\frakb, \rho)$
be an $\bfS$-linear category.
We have the category
$\frakb_\bfR$
of $\bfR$-linear objects
whose objects are pairs
$(B, \varphi)$
where
$B \in \frakb$
and
$\varphi \colon \bfR \to \frakb(B,B)$
is a ring map with
$\varphi \circ \theta = \rho_B$,
and
whose morphisms are those of
$\frakb$
compatible with the ring maps.
An object
$B \in \frakb$
belongs to
$\frakb_\bfR$
if and only if
$1_B$
is annihilated by the kernel of
$\theta$.
Taking $\bfR$-linear objects defines a functor
$(-)_\bfR \colon \frakb \to \frakb_\bfR$,
which is right adjoint to
$\overline{(-)}$
in the sense that
there is a natural isomorphism
\begin{align*}
\overline{\Add(\bfR)(\fraka, \frakb_\bfR)}
\simeq
\Add(\bfS)(\overline{\fraka}, \frakb)
\end{align*}
of $\bfS$-linear categories.
If
$\scrD$
is an $\bfS$-linear abelian category,
then
$\scrD_\bfR$
is also abelian
and
by
\cite[Proposition 4.2]{LV06}
the forgetful functor
$\scrD_\bfR \to \scrD$
is exact.
From
$\left( \Mod(\bfS) \right)_\bfR \simeq \Mod(\bfR)$,
it follows
\begin{align*}
\overline{\Add(\bfR)(\fraka, \Mod(\bfR))}
\simeq
\Mod(\fraka)
\end{align*}
for any $\bfR$-linear category
$\fraka$.

\begin{lem} {\rm{(}\cite[Proposition 4.4(1)]{LV06}\rm{)}} \label{lem:Modlinab}
Let
$\frakb$
be an $\bfS$-linear category.
Then there is an equivalence
$\Mod(\frakb \otimes_\bfS \bfR)
\to
\Mod(\frakb)_\bfR$
of $\bfR$-linear categories
which makes the diagram
\begin{align*} 
\begin{gathered}
\xymatrix{
\Mod(\frakb \otimes_\bfS \bfR) \ar[r]^{\simeq} \ar[d]_{} & \Mod(\frakb)_\bfR \ar[d]^{} \\
\Mod(\frakb) \ar[r]^{\id} & \Mod(\frakb)
}
\end{gathered}
\end{align*}
commutes,
where
the left vertical arrow is the dual to
$\frakb \to \frakb \otimes_\bfS \bfR$
and
the right vertical arrow is the forgetful functor.
\end{lem}

\subsection{Deformations of linear categories}
Let
$\fraka$
be an $\bfR$-linear category.
A
\emph{linear $\bfS$-deformation}
of
$\fraka$
is an $\bfS$-linear category
$\frakb$
together with an $\bfS$-linear functor
$\frakb \to \overline{\fraka}$
inducing an equivalence
$\frakb \otimes_\bfS \bfR \to \fraka$.
Two deformations
$f \colon \frakb \to \overline{\fraka}, \ 
f^\prime \colon \frakb^\prime \to \overline{\fraka}$
are
\emph{equivalent}
if there is an equivalence
$\Phi \colon \frakb \to \frakb^\prime$
of $\bfS$-linear categories
such that
$f^\prime \circ \Phi$
is natural isomorphic to
$f$.
When
$\frakb$
is flat over
$\bfR$,
we call the deformation
$\frakb$
\emph{flat}.
We denote by
$\Def^{lin}_\fraka (\bfS)$   
the set of equivalence classes of flat linear $\bfS$-deformations of
$\fraka$.
The notation will be justified below with respect to the choice of universe.
When
$\frakb \otimes_\bfS \bfR \to \fraka$
is an isomorphism,
we call the deformation
$\frakb$
\emph{strict}.
Two strict linear deformations
$f \colon \frakb \to \overline{\fraka}, \ 
f^\prime \colon \frakb^\prime \to \overline{\fraka}$
are
\emph{equivalent}
if there is an isomorphism
$\Phi \colon \frakb \to \frakb^\prime$
of $\bfS$-linear categories
such that
$f^\prime \circ \Phi = f$.
We denote by
$\Def^{s-lin}_\fraka(\bfS)$   
the set of equivalence classes of strict flat linear $\bfS$-deformations of
$\fraka$.
Also this notation will be justified below.

\subsection{Deformations of abelian categories}
Let
$\scrC$
be an $\bfR$-linear abelian category.
An
\emph{abelian $\bfS$-deformation}
of
$\scrC$
is an $\bfS$-linear abelian category
$\scrD$
together with an $\bfS$-linear functor
$\overline{\scrC} \to \scrD$
inducing an equivalence
$\scrC \to \scrD_\bfR$.
When
$\scrD$
is flat over
$\bfR$,
we call the deformation
$\scrD$
flat.
Two deformations
$g \colon \overline{\scrC} \to \scrD,
g^\prime \colon \overline{\scrC} \to \scrD^\prime$
are
\emph{equivalent}
if there is an equivalence
$\Psi \colon \scrD \to \scrD^\prime$
of $\bfS$-linear abelian categories
such that
$\Psi \circ g^\prime$
is natural isomorphic to
$g$.
We denote by
$\Def^{ab}_\scrC(\bfS)$   
the set of equivalence classes of flat abelian $\bfS$-deformations of
$\scrC$.
The notation will be justified below
with respect to the choice of universe.
When
$\scrC \to \scrD_\bfR$
is an isomorphism,
we call the deformation
$\scrD$
\emph{strict}.
Two strict abelian deformations
$g \colon \overline{\scrC} \to \scrD, 
g^\prime \colon \overline{\scrC} \to \scrD^\prime$
are
\emph{equivalent}
if there is an isomorphism
$\Psi \colon \scrD \to \scrD^\prime$
of $\bfS$-linear abelian categories
such that
$\Psi \circ g^\prime = g$.
We denote by
$\Def^{s-ab}_\scrC(\bfS)$   
the set of equivalence classes of strict flat abelian $\bfS$-deformations of
$\scrC$.
Also this notation will be justified below.

Assume that
$\theta \colon \bfS \to \bfR$
is a homomorphism of coherent commutative rings
with $\bfR$ being finitely presented as an $\bfS$-module.
Then the bifunctors
\begin{align*}
(-) \otimes_\bfS (-) \colon \scrD \times \mmod (\bfS) \to \scrD, \ 
\Hom_\bfS (-, -) \colon \mmod (\bfS) \times \scrD \to \scrD
\end{align*}
yield respectively left and right adjoint
\begin{align*}
(-) \otimes_\bfS \bfR \colon \scrD \to \scrD_\bfR \simeq \scrC, \ 
\Hom_\bfS(\bfR, -) \colon \scrD \to \scrD_\bfR \simeq \scrC
\end{align*}
to the natural inclusion functor
$\scrC \simeq \scrD_\bfR \hookrightarrow \scrD$
\cite[Proposition 4.3]{LV06}.
They agree with the adjoints in the Section
\pref{subsec:BC}.

\subsection{Flat nilpotent deformations of categories}
Assume further that
$\theta$
is surjective.
Then for an $\bfS$-linear abelian category
$\scrD$
the forgetful functor
$\scrD_\bfR \to \scrD$
is fully faithful.
When the kernel 
$I = \ker \theta$
is nilpotent,
we call both linear and abelian $\bfS$-deformations
\emph{nilpotent}.
From now on, 
we restrict our attention to flat nilpotent deformations.
The following properties of $\bfR$-linear category
$\fraka$
and
$\bfR$-linear abelian category
$\scrC$
are respectively preserved under flat nilpotent linear and abelian deformations
\cite[Proposition 6.7, 6.9, Theorem 6.16, 6.29, 6.36]{LV06}.
\begin{itemize}
\item
$\fraka, \scrC$
are essentially small.
\item
$\scrC$
has enough injectives.
\item
$\scrC$
is a Grothendieck category.
\item
$\scrC$
is a locally coherent Grothendieck category.
\end{itemize}
Here,
we call
$\scrC$
\emph{locally coherent Grothendieck}
when it is Grothendieck
and
generated by a small abelian subcategory of finitely presented objects.

\subsection{Deformation pseudofunctors}
In order to be careful about our choices of universes,
we temporarily make them explicit in the notation.
Let
$\cU$
be a universe containing the field
$\bfk$.
We denote by
$\cU-\Rng^0$
the category whose objects are coherent commutative $\cU$-rings
and
whose morphisms are surjective ring maps
with finitely generated nilpotent kernels. 
We are interested in the category
$\cU-\Rng^0 / \bfk$.
Fix some other universe $\cW$.
A
\emph{deformation pseudofunctor}
is a pseudofunctor
$D \colon \cU-\Rng^0 / \bfk \to \cW-\Gd$.
Two deformation pseudofunctors
$D, D^\prime$
are
\emph{equivalent}
if there is a pseudonatural transformation
$\mu \colon D \to D^\prime$
such that
for each
$\bfR \in \cU-\Rng^0 / \bfk$
we have an equivalence
$D(\bfR) \to D^\prime(\bfR)$
of categories. 
For any enlargement
$\cU^\prime$
of
$\cU$,
the canonical functor
\begin{align*}
\cU-\Rng^0 / \bfk \to \cU^\prime-\Rng^0 / \bfk
\end{align*}
is an equivalence of categories
\cite[Proposition 8.1]{LV06}.
Thus the deformation pseudofunctor is independent of the choice of
$\cU$ up to equivalence.

Let 
$\fraka$
be a flat $\bfk$-linear $\cU$-category
and
$\scrC$
a flat $\bfk$-linear abelian $\cU$-category.
Fix a universe
$\cV$
such  that
$\fraka, \scrC$
are essentially $\cV$-small
and
$\cU \in \cV$.
For 
$\bfR \in \cU-\Rng^0 / \bfk$
we consider the groupoid
$\cV-\ddef^{lin}_\fraka(\bfR)$ 
whose objects are flat linear $\bfR$-deformations of $\fraka$
belonging to
$\cV$,
and
whose morphisms are equivalences of deformations
up to natural isomorphism of functors.
Also we consider the groupoid
$\cV-\ddef^{ab}_\scrC(\bfR)$ 
whose objects are flat abelian $\bfR$-deformations of
$\scrC$
belonging to $\cV$,
and
whose morphisms are equivalences of deformations
up to natural isomorphism of functors.
Enlarging the universe
$\cW$
if necessary,
we may assume that   
$\cV \in \cW$
and
we obtain deformation pseudofunctors
\begin{align*}
\cV-\ddef^{lin}_\fraka, \cV-\ddef^{ab}_\scrC
\colon
\cU-\Rng^0 / \bfk
\to
\cW-\Gd. 
\end{align*}
The universe
$\cW$
is a purely technical device
which guarantees
$\cV-\ddef^{lin}_\fraka, \cV-\ddef^{ab}_\scrC$
taking values in categories.
Moreover,
whether two deformation pseudofunctors are equivalent
is preserved under enlargement of
$\cW$. 
On the other hand,
by
\cite[Proposition 8.3]{LV06}
for any enlargement
$\cV^\prime \in \cW$
of
$\cV$,
the canonical pseudonatural transformations
\begin{align*}
\cV-\ddef^{lin}_\fraka \to \cV^\prime-\ddef^{lin}_\fraka, \
\cV-\ddef^{ab}_\scrC \to \cV^\prime-\ddef^{ab}_\scrC
\end{align*}
define equivalences of deformation pseudofunctors.

In summary,
as long as we consider flat nilpotent deformations,
the choice of universe does not affect deformation pseudofunctors
up to equivalence.
Thus we simply write
$\ddef^{lin}_\fraka, \ddef^{ab}_\scrC$ 
for deformation pseudofunctors.
Since they have small skeletons
\cite[Theorem 8.4, 8.5]{LV06},
we also write 
$\Def^{lin}_\fraka, \Def^{ab}_\scrC$
for deformation functors
\begin{align*}
\cV-\Def^{lin}_\fraka, \cV-\Def^{ab}_\scrC
\colon
\cU-\Rng^0 / \bfk
\to
\cW-\Set 
\end{align*}
which take values in sets.

Finally,
we collect relevant results on deformations of linear and abelian categories.

\begin{lem} {\rm{(}\cite[Theorem 8.16]{LV06}\rm{)}} \label{lem:linab1}
Let
$\bfS \to \bfR$
be a morphism in
$\cU-\Rng^0 / \bfk$
and
$\fraka$
an essentially small flat $\bfR$-linear category.
Then there is a bijection
\begin{align*}
\Def^{lin}_\fraka(\bfS) \to \Def^{ab}_{\Mod(\fraka)}(\bfS), \
\frakb \mapsto \Mod(\frakb).
\end{align*}
In particular,
deformations of a module category are module categories.
\end{lem}

\begin{lem} {\rm{(}\cite[Theorem 8.17]{LV06}\rm{)}} \label{lem:linab2}
Let
$\bfS \to \bfR$
be a morphism in
$\cU-\Rng^0 / \bfk$
and
$\scrC$
an essentially small flat $\bfR$-linear abelian category with enough injectives.
Then there is a bijection
\begin{align*}
\Def^{lin}_{\Inj(\scrC)}(\bfS) \to \Def^{ab}_\scrC(\bfS), \
\frakj \mapsto (\mmod(\frakj))^{op}.
\end{align*}
\end{lem}

\begin{lem} {\rm{(}\cite[Proposition B.3]{LV06}\rm{)}} \label{lem:s-linlin}
Let
$\bfS \to \bfR$
be a morphism in
$\cU-\Rng^0 / \bfk$
and
$\fraka$
an essentially small flat $\bfR$-linear category.
Then the map
\begin{align*}
\Def^{s-lin}_\fraka(\bfS) \to \Def^{lin}_\fraka(\bfS)
\end{align*}
induced by the canonical pseudofunctor
\begin{align*}
\ddef^{s-lin}_\fraka \to \ddef^{lin}_\fraka
\end{align*}
is bijective.
\end{lem}

Now,
let again
$\bfR$
be the fixed local artinian $\bfk$-algebra
with residue field
$\bfk$,
the square zero extension
\begin{align*}
0 \to \bfI \to \bfS \to \bfR \to 0,
\end{align*}
and
the chosen generators
$\epsilon = (\epsilon_1, \ldots, \epsilon_l)$    
of
$\bfI$
regarded as a free $\bfR$-module of rank
$l$.

\begin{lem} {\rm{(}\cite[Proposition 4.2]{Low08}\rm{)}} \label{lem:CLASSlin}
Let
$(\fraka, m)$
be a flat $\bfR$-linear category with compositions  
$m$.
Then there is a bijection
\begin{align} \label{eq:CLASSlin}
H^2 \bfC(\fraka)^{\oplus l}
\to
\Def^{lin}_\fraka(\bfS), \
\phi
\mapsto
(\fraka[\epsilon], m + \phi \epsilon), \
\phi \in Z^2 \bfC(\fraka)^{\oplus l}.
\end{align}
Another cocycle
$\phi^\prime \in Z^2 \bfC(\fraka)^{\oplus l}$
maps to an isomorphic linear deformation
if and only if
there is an element
$h \in \bfC^1(\fraka)$
satisfying
$\phi^\prime - \phi = d_m (h)$.
\end{lem}

\begin{lem} {\rm{(}\cite[Theorem 3.1]{LV06a}\rm{)}} \label{lem:CLASSab}
Let
$\scrC$
be a flat $\bfR$-linear abelian category.
Then there is a bijection
\begin{align} \label{eq:CLASSab}
H^2 \bfC_{ab}(\scrC)^{\oplus l}
\to
\Def^{ab}_\scrC(\bfS).
\end{align}
\end{lem}

Here,
$\bfC(\fraka)$
is the Hochschild object associated with
$\fraka$.
The compositions
$m$
is an element of
\begin{align*}
\prod_{A_0, A_1, A_2 \in \fraka} [\fraka(A_1, A_2) \otimes_\bfR \fraka(A_0, A_1), \fraka(A_0, A_2)]^0
\end{align*}
rather than the ring map
$\rho$
defining the $\bfR$-linear structure. 
For an $\bfR$-linear abelian category
$\scrC$,
the associated Hochschild object is defined as
$\bfC_{ab}(\scrC) = \bfC_{sh}(\Ind(\Inj(\scrC)))$,
where
$\bfC_{sh}(\Ind(\Inj(\scrC)))$
is the Shukla complex associated with
$\Ind(\Inj(\scrC))$.
Note that
we have
$H^* \bfC_{sh}(\fraka) = H^* \bfC(\fraka)$.
We will review the definitions in \pref{sec:dgDefPerf}.

\subsection{Examples}
Let
$X$
be a smooth proper $\bfR$-scheme.
Since it is noetherian,
the category
$\Qch(X)$
of quasi-coherent sheaves on
$X$
has enough injectives.
We denote by
$\fraki = \Inj(\Qch(X))$
the full $\bfR$-linear subcategory of injective objects.
Since
$X$
is flat separated,
by
\cite[Proposition 4.28, 4.30(2)]{DLL}
the $\bfR$-linear abelian category
$\Qch(X)$
is flat.
From
\cite[Proposition 2.9(6)]{LV06}
it follows that
the $\bfR$-linear category 
$\fraki$
is flat.
Then by
\pref{lem:CLASSlin}
and
\pref{lem:CLASSab}
or
\pref{lem:linab2}
both
flat linear $\bfS$-deformations of
$\fraki$
and
flat abelian $\bfS$-deformations of
$\Qch(X)$
are classified by
$H^2 \bfC(\fraki)^{\oplus l}$.

\section{The category of quasi-coherent sheaves}
In this section,
we review an alternative description of Toda's construction
in terms of
the descent category of the category of twisted quasi-coherent presheaves
over
the restricted structure sheaf,
following the exposition from
\cite[Section 4,5]{DLL}.
It follows that,
for square zero extension of relatively smooth $\bfR$-schemes,
Toda's construction coincides with the deformation of the category of quasi-coherent sheaves along the corresponding Hochschild cocycle.
As a consequence,
deforming the category of the quasi-coherent sheaves
is equivalent to
deforming the complex structure
for higher dimensional Calabi--Yau manifolds.
In particular,
deformations of the category of quasi-coherent sheaves are given by the category of quasi-coherent sheaves on deformations. 

\subsection{Descent categories}
Let
$\frakU$
be a small category
and
$\Cat(\bfR)$
the category of small $\bfR$-linear categories and $\bfR$-linear functors.
A
\emph{prestack}
$\scrA$
is a pseudofunctor
$\frakU^{op} \to \Cat(\bfR)$
consists of the following data:
\begin{itemize}
\item
for each
$U \in \frakU$
an $\bfR$-linear category
$\scrA(U)$,
\item
for each
$u \colon V \to U$
in
$\frakU$
an $\bfR$-linear functor
$f^u \colon \scrA (U) \to \scrA (V)$,
\item
for each pair
$u \colon V \to U, \ v \colon W \to V$
in
$\frakU$
a natural isomorphism
$c^{u,v} \colon f^v f^u \to f^{uv}$,
\item
for each
$U \in \frakU$
a natural isomorphism
$z^U \colon 1_{\scrA(U)} \to f^{1_U}$.
\end{itemize}
Moreover, these data must satisfy 
\begin{align*}
\begin{gathered}
c^{u,vw}(c^{v,w} \circ f^u) = c^{uv, w} (f^w \circ c^{u,v}), \\
c^{u, 1_V} (z^V \circ f^u) = 1, \ \ c^{1_U, u} (f^u \circ z^U) = 1
\end{gathered}
\end{align*}
for each triple 
$u \colon V \to U, \ v \colon W \to V, \ w \colon T \to W$
in
$\frakU$.
With
$\scrA(U)$
regarded as one-objected categories,
a twisted presheaf
$\scrA$
of $\bfR$-algebras provides an example of a prestack of $\bfR$-linear category. 

For prestacks   
$\scrA = (\scrA, m, f, c, z),
\scrA^\prime
=
(\scrA^\prime, m^\prime, f^\prime, c^\prime, z^\prime)$
of
$\bfR$-linear category on
$\frakU$,
a
\emph{morphism}
$(g, h) \colon \scrA \to \scrA^\prime$
is a pseudonatural transformation
which consists of the following data:
\begin{itemize}
\item
for each
$U \in \frakU$
an $\bfR$-linear functor
$g^U \colon \scrA(U) \to \scrA^\prime(U)$,
\item
for each
$u \colon V \to U$
in
$\frakU$
a natural isomorphism 
$h^u \colon f^{\prime u} g^U \to g^V f^u$.
\end{itemize}
Moreover, these data must satisfy 
\begin{align*}
\begin{gathered}
h^{uv} (c^{\prime u,v} \circ g^U)
=
(g^W \circ c^{u,v})(h^v \circ f^u)(f^{\prime v} \circ h^u), \\
h^{1_U} (z^{\prime U} \circ g^U)
=
g^U \circ z^U.
\end{gathered} 
\end{align*}
for each pair 
$u \colon V \to U, \ v \colon W \to V$
in
$\frakU$.
When
$\scrA$
is a twisted presheaf of $\bfR$-algebras,
morphisms of twisted presheaves of $\bfR$-algebras
coincide with
morphisms of prestacks.

A
\emph{pre-descent datum}
in a prestack
$\scrA$
is a collection
$(A_U)_U$
of objects
$A_U \in \scrA(U)$
with a morphism
$\varphi_u \colon f^u A_U \to A_V$
in
$\scrA(V)$
for each
$u \colon V \to U$
in
$\frakU$,
which satisfies
\begin{align*}
\varphi_v f^v \varphi_u = \varphi_{uv} c^{u,v, A_U} 
\end{align*}
given an additional
$v \colon W \to V$
in
$\frakU$.   
A
\emph{morphism}
of pre-descent data
$g \colon (A_U)_U \to (A^\prime_U)_U$
is a collection
$(g_U)_U$
of compatible morphisms
$g_U \colon A_U \to A^\prime_U$.
Pre-descent data and their morphisms form a category
$\PDes (\scrA)$
equipped with a canonical functor
\begin{align*}
\pi_V \colon \PDes (\scrA) \to \scrA(V), \
(A_U)_U \mapsto A_V.
\end{align*}
When all
$\varphi_u$
are isomorphisms,
$(A_U)_U$
is called a
\emph{descent datum}
and
we denote by
$\Des (\scrA)$
the full subcategory of descent data.
Given limits and colimits in each
$\scrA(U)$ 
preserved by all
$f^u \colon \scrA(U) \to \scrA(V)$,
there exist ones in
$\Des (\scrA)$
preserved by all
$\pi_V \colon \Des(\scrA) \to \scrA(V)$
\cite[Proposition 4.5(3)]{DLL}.
In particular,
if each category
$\scrA(U)$
is abelian
and
all
$f^u$
are exact, 
then
$\Des(\scrA)$
is abelian
and
$\pi_V$
are exact. 

\subsection{The category of quasi-coherent modules over a prestack}
The
\emph{category of right quasi-coherent modules}
over a prestack
$\scrA$
is defined as
\begin{align*}
\Qch (\scrA) = \Qch^r (\scrA) = \Des (\Mod_\scrA),
\end{align*}
where
$\Mod_\scrA$
is the
\emph{associated prestack}
with a prestack
$\scrA$
given by
\begin{align*}
\Mod_\scrA
=
\Mod^r_\scrA
\colon
\frakU^{op} \to \Cat(\bfR), \
U \mapsto \Mod_\scrA (U) = \Mod(\scrA(U)),
\end{align*}
whose restriction functor
\begin{align*}
-\otimes_u \scrA(V) \colon \Mod (\scrA(U)) \to \Mod (\scrA(V))
\end{align*}
is the unique colimit preserving functor extending
$f^u \colon \scrA(U) \to \scrA(V)$.
The functor sends each
$F \in \Mod (\scrA(U))$
to an $\bfR$-linear functor
$F \otimes_u \scrA(V) \colon \scrA(V)^{op} \to \Mod(\bfR)$
such that
\begin{align*}
F \otimes_u \scrA(V) (B)
=
\bigoplus_{A \in \scrA(U)} F(A) \otimes_\bfR \scrA(V) (B, f^u A) / \sim
\end{align*}
for each
$B \in \scrA(V)$.
Here,
$\sim$
denotes the equivalence relation defined as
\begin{align*}
F(a) (x) \otimes y \sim x \otimes f^u (a) y 
\end{align*}
for
$x \in F(A^\prime), \ y \in \scrA(V)(B, f^u A)$,
and
$a \colon A \to A^\prime$
in
$\scrA(U)$.

In the case
where
$F = \scrA(U)(-, A^\prime)$
for some
$A^\prime \in \scrA(U)$,
$f^u$
induces an isomorphism
\begin{align*}
\theta^u_{A^\prime} \colon \scrA(U)(-, A^\prime) \otimes_u \scrA(V)
\to
\scrA(V)(-, f^u A^\prime).
\end{align*}
If
$u = 1_U$,
then
$z^U \colon 1_{\scrA(U)} \to f^{1_U}$
induces an isomorphism 
\begin{align*}
\Mod(z)^U \colon 1_{\Mod_\scrA(U)} \to -\otimes_{1_U}\scrA(U). 
\end{align*}
Since we have 
\begin{align*}
F \otimes_u \scrA(V) \otimes_v \scrA(W)(C)
=
\bigoplus_{A \in \scrA(U), B \in \scrA(V)} F(A) \otimes_\bfR \scrA(V)(B, f^u A) \otimes_\bfR \scrA(W)(C, f^v B) / \sim,
\end{align*}
$\theta^v_{f^u A}$
and
$c^{u,v}$
induce another isomorphism
\begin{align*}
\Mod(c^{u,v}) \colon -\otimes_u \scrA(V) \otimes_v \scrA(W) \to -\otimes_{uv} \scrA(W).
\end{align*}
When
$\scrA$
is a twisted presheaf of $\bfR$-algebras,
$\Mod (\scrA(U))$
coincides with the category of right $\scrA(U)$-modules
whose restriction functor is the ordinary tensor product
and
$\Mod(c)^{u,v}, \Mod (z)^U$
are respectively given by
\begin{align*}
\begin{gathered}
\Mod(c)^{u,v}_M \colon M \otimes_u \scrA(V) \otimes_v \scrA(W) \to M \otimes_{uv} \scrA(W), \
m \otimes a \otimes b \mapsto m \otimes c^{u,v} f^v (a)b, \\
\Mod (z)^U_M \colon M \to M \otimes_{1_U} \scrA(U), \
m \mapsto m \otimes z^U
\end{gathered}
\end{align*}
for any right
$\scrA(U)$-module
$M$.

\subsection{The category of twisted quasi-coherent presheaves over a twisted presheaf}
Let
$\scrA$
be a presheaf of $\bfR$-algebras on
$\frakU$.
We denote by
$\Pr(\scrA|_U)$
the category of presheaves of right $\scrA|_U$-modules on
$\frakU / U$,
where
$\scrA|_U$
is the induced presheaf on
$\frakU / U$
with
$\scrA|_U (V \to U) = \scrA (V)$
for
$U \in \frakU$
and
$u \colon V \to U$
in
$\frakU$.
Each
$u \colon V \to U$
in
$\frakU$
induces a functor
$u^*_{\Pr} \colon \Pr(\scrA|_U) \to \Pr(\scrA|_V)$.
Since we have
$v^*_{\Pr} u^*_{\Pr} = (uv)^*_{\Pr}$
and
$(1_U)^*_{\Pr} = 1_{\Pr(\scrA|_U)}$
given an additional
$v \colon W \to V$
in
$\frakU$,
the assignments
$U \mapsto \Pr(\scrA|_U)$
and
$u \mapsto u^*_{\Pr}$
define a functor
\begin{align*}
\Pr(\scrA) \colon \frakU^{op} \to \Cat (\bfR).
\end{align*}

Let
$M$
be a right $\scrA(U)$-module.
Then
$\tilde{M}(u)
\coloneqq
M \otimes_u \scrA(V)
=
M \otimes_{\scrA(U)} \scrA(V)$
is a right $\scrA(V)$-module
with 
$\scrA(V)$
regarded as a left $\scrA(U)$-module via
$f^u$.
Suppose that
$u^\prime \colon V^\prime \to U$
satisfies
$uv^\prime = u^\prime$
for
$v^\prime \colon V^\prime \to V$.
We have the right $\scrA(U)$-module homomorphism
$1_M \otimes f^{v^{\prime}} \colon \tilde{M}(u) \to \tilde{M}(u^\prime)$. 
The assignments
$u \mapsto \tilde{M}(u)$
and
$f^{v^{\prime}} \mapsto 1_M \otimes f^{v^{\prime}}$
define a presheaf
$\tilde{M}$
of right $\scrA(U)$-modules on
$\frakU/U$.
Any $\scrA(U)$-module homomorphism
$g \colon M \to N$
induces a natural transform
$\tilde{g} = \{ \tilde{g}^u \coloneqq g \otimes 1_{\scrA(V)} \}_u$.
Thus the assignments
$M \mapsto \tilde{M}$
and
$g \mapsto \tilde{g}$
define a functor
\begin{align} \label{eq:Q}
Q^U \colon \Mod(\scrA(U)) \to \Pr(\scrA|_U).
\end{align}
We have the canonical isomorphism
\begin{align*}
\can^{u,v}_M
\colon
M \otimes_u \scrA(V) \otimes_v \scrA(W))
\to
M \otimes_{uv} \scrA(W), \
m \otimes a \otimes b \mapsto m \otimes f^v (a)b.
\end{align*}
By
\cite[Lemma 4.10]{DLL}
the functor
$Q^U$
is fully faithful
and
there is a natural isomorphism
\begin{align} \label{eq:tau}
\tau^u \colon u^*_{\Pr} Q^U \to Q^V (- \otimes_u \scrA(V))
\end{align}
induced by
$(\can^{u,v}_M)^{-1}$.
A
\emph{quasi-coherent presheaf}
over
$\scrA_U$
is defined as the essential image of some
$\scrA(U)$-module
$M$
by
$Q^U$.
We denote by
$\QPr (\scrA|_U)$
the category of quasi-coherent presheaves over
$\scrA|_U$.

When
$\scrA$
is a twisted presheaf with central twists
$c$,
one can adapt
$\Mod(c)^{u,v}$
to
$\Pr(c)^{u,v}$
as follows.
For
$\scrF \in \Pr(|\scrA||_U)$
and
$w \colon T \to W$
in
$\frakU/W$
the central invertible element
$f^w(c^{u,v})$
in
$\scrA(T)$
gives an automorphism
\begin{align*}
f^w(c^{u,v})_r \colon \scrF(uvw) \to \scrF(uvw), \
m \mapsto m f^w(c^{u,v})
\end{align*}
inducing an isomorphism
\begin{align*}
\Pr(c)^{u,v}_\scrF
\colon
v^*_{\Pr} u^*_{\Pr} (\scrF)
\to (uv)^*_{\Pr} (\scrF)
\end{align*}
in
$\Pr(|\scrA||_U)$.
Since we have
\begin{align*}
\Pr(c)^{u, vw} \Pr (c)^{v,w}
=
\Pr(c)^{uv, w} w^*_ {\Pr} (\Pr (c)^{u,v}),
\end{align*}
the assignments
$U \mapsto \Pr(|\scrA||_U)$
and
$u \mapsto u^*_{\Pr}$
define a prestack
\begin{align*}
{\Pr}_\scrA \colon \frakU^{op} \to \Cat (\bfR)
\end{align*}
whose twist functor is given by
$\Pr(c)$
and
$z$
is given by the identity.
Restricting to the essential images
$\QPr(|\scrA||_U)$,
we obtain another prestack
$\QPr_\scrA$.
The
\emph{category of right twisted quasi-coherent presheaves}
over a twisted presheaf
$\scrA \colon \frakU^{op} \to \Alg(\bfR)$
with central twists is defined as
\begin{align*}
\QPr (\scrA) = \Des (\QPr_\scrA).
\end{align*}

\begin{lem} {\rm{(}\cite[Theorem 4.12]{DLL}\rm{)}} \label{lem:Bridge1}
Let
$\scrA \colon \frakU^{op} \to \Alg( \bfR)$
be a twisted presheaf with central twists.
Then
$Q= (Q^U, \tau^u)_{U, u}$
defines an equivalence
\begin{align*}
\Qch (\scrA) \simeq \QPr (\scrA)
\end{align*}
of $\bfR$-linear categories,
where
$Q^U, \tau^u$
are given by
\pref{eq:Q},
\pref{eq:tau}
respectively.
\end{lem}

\subsection{Deformations of the restricted structure sheaves}
Let
$\scrA \colon \frakU^{op} \to \Cat (\bfR)$
be a flat prestack.
Recall that
$\scrA$
is
\emph{flat}
if $\bfR$-modules
$\scrA(U)(A, A^\prime)$
are flat for all
$U \in \frakU$
and
$A, A^\prime \in \scrA(U)$.
An
\emph{$\bfS$-deformation}
of
$\scrA$
is a flat prestack
$\scrB \colon \frakU^{op} \to \Cat (\bfS)$
together with an equivalence of prestacks
$\scrB \otimes_\bfS \bfR \to \scrA$,
i.e.,
for each
$U$
there is a morphism of prestacks inducing an equivalence
$\scrB(U) \otimes_\bfS \bfR \to \scrA(U)$
of $\bfR$-linear categories
\cite[Proposition 4.7]{DLL}.
Two deformations
$\scrB, \scrB^\prime$
are
\emph{equivalent}
if there is an equivalence
$\scrB \to \scrB^\prime$
of prestacks compatible with equivalences
$\scrB \otimes_\bfS \bfR \to \scrA, \
\scrB^\prime \otimes_\bfS \bfR \to \scrA$.
We denote by
$\Def^{tw}_\scrA(\bfS)$
the set of equivalence classes of $\bfS$-deformations of
$\scrA$.
When
$\scrB \otimes_\bfS \bfR \to \scrA$
is an isomorphism of prestacks,
i.e.,
for each
$U$
there is a morphism of prestacks inducing an isomorphism
$\scrB(U) \otimes_\bfS \bfR \to \scrA(U)$
of $\bfR$-linear categories,
we call the deformation
$\scrB$
\emph{strict}.
Two strict deformations
$\scrB, \scrB^\prime$
are
\emph{equivalent}
if there is an isomorphism
$\scrB \to \scrB^\prime$
of prestacks inducing the identity on
$\scrA$.
We denote by
$\Def^{s-tw}_\scrA(\bfS)$
the set of equivalence classes of strict $\bfS$-deformations of
$\scrA$.
Recall that
twised presheaves of $\bfR$-algebras can be regarded as a prestack.
Due to the lemma below,
as long as we consider equivalence classes of twisted deformations of flat presheaves,
we may restrict our attention to strict twisted deformations.

\begin{lem} {\rm{(}\cite[Proposition 5.9]{DLL}\rm{)}}
Let
$(\scrA, m, f, c, z)$
be a flat prestack of $\bfR$-linear categories on
$\frakU$.
Then the canonical map
\begin{align*}
\Def^{s-tw}_\scrA(\bfS) \to \Def^{tw}_\scrA(\bfS)
\end{align*}
is bijective.
\end{lem}

Let
$\frakU$
be a finite poset with binary meets.
Then any prestack on
$\frakU$
is
\emph{quasi-compact}
since
$\frakU$
is finite.
A prestack
$\scrA \colon \frakU^{op} \to \Cat(\bfR)$
is
\emph{right semi-separated}
if the associated prestack
$\Mod_\scrA$
is of affine localizations.
Namely,
for all
$U, V, W \in U$
with
$v \colon V \to U, w \colon W \to U$
in
$\frakU$
and the pullback diagram
\begin{align*} 
\begin{gathered}
\xymatrix{
V \cap W \ar[r]^{\bar{w}} \ar[d]_{\bar{v}} & V \ar[d]^{v} \\
W \ar[r]^{w} & U
}
\end{gathered}
\end{align*}
the following conditions are satisfied.
\begin{itemize}
\item
The category
$\Mod_\scrA(U)$
is Grothendieck abelian. 
\item
The functor
$v^* \colon \Mod_\scrA(U) \to \Mod_\scrA(V)$
is exact.
\item
The functor
$v^*$
admits a fully faithful exact right adjoint
$v_* \colon \Mod_\scrA(V) \to \Mod_\scrA(U)$.
\item
There are natural isomorphisms
\begin{align*}
(v_*v^*)(w_*w^*)
\cong
(v \bar{w})_*(v \bar{w})^*
\cong
(w_*w^*)(v_*v^*).
\end{align*}
\end{itemize}
A presheaf  
$\scrA \colon \frakU^{op} \to \Alg(\bfR)$
is
\emph{right semi-separated}
if so is
$\scrA$
with
$\scrA(U)$
regarded as one-objected categories.  
Every right semi-separated prestack is
\emph{geometric},
i.e.,
the restriction functor
\begin{align*}
- \otimes_u \scrA(V) \colon \Mod(\scrA(U)) \to \Mod(\scrA(V)) 
\end{align*}
is exact.
Note that
for any geometric prestack
$\scrA \colon \frakU^{op} \to \Cat(\bfR)$
on a small category
$\Qch(\scrA)$
is a Grothendieck abelian category
\cite[Theorem 4.14]{DLL}.

Let
$X$
be a smooth proper $\bfR$-scheme.
Choose a finite affine open cover
$\frakU$
closed under intersections.
We denote by
$\scrO_X|_\frakU$
the restricted structure sheaf to
$\frakU$.
Since
$U \cap V$
is affine
as
$X$
is separated,
we have isomorphisms of $\scrO_X (U)$-modules
\begin{align*}
\scrO_X (V) \otimes_{\scrO_X (U)} \scrO_X (W)
\cong
\scrO_X (U \cap V)
\cong
\scrO_X (W) \otimes_{\scrO_X (U)} \scrO_X (V)
\end{align*}
for all
$U, V, W \in U$
with
$V, W \subset U$.
Since pushforwards along open immersions
$V \hookrightarrow U$
of affine schemes are fully faithful,
by
\cite[Lemma 3.1]{DLL} 
the restriction maps
$\scrO_X (U) \to \scrO_X (V)$
are flat epimorphism of rings.
Then one can apply
\cite[Proposition 4.28]{DLL}
to see that
the presheaf
$\scrO_X |_\frakU \colon \frakU^{op} \to \Alg(\bfR)$
is right semi-separated.
Since
$\scrO_X(U)$
are flat $\bfR$-modules,
the category
\begin{align*}
\Qch(\scrO_X |_\frakU)
\simeq
\QPr(\scrO_X |_\frakU)
\simeq
\Qch(X)
\end{align*}
is flat over
$\bfR$
and
Grothendieck abelian
\cite[Proposition 4.30]{DLL}.  

\begin{lem} {\rm{(}\cite[Theorem 5.10]{DLL}\rm{)}} \label{lem:Bridge2}
Let
$X$
be a smooth proper $\bfR$-scheme
with a finite affine open cover
$\frakU$
closed under intersections.
Then
every twisted $\bfS$-deformations of the restricted structure sheaf
$\scrO_X |_\frakU$
is a quasi-compact semi-separated presheaf on
$\frakU$
and
there is a bijection
\begin{align*}
\Def^{tw}_{\scrO_X |_\frakU}(\bfS)
\to
\Def^{ab}_{\Qch(X)}(\bfS), \
(\scrO_X |_\frakU)_\phi \mapsto \Qch((\scrO_X |_\frakU)_\phi),
\end{align*}
where
$\phi \in H^2 \bfC_{GS}(\scrO_X |_\frakU)^{\oplus l}$
is a cocycle
and
$(\scrO_X |_\frakU)_\phi$
is the twisted $\bfS$-deformation of
$\scrO_X |_\frakU$
along
$\phi$.
In particular,
the category of right quasi-coherent modules over a twisted deformation of
$\scrO_X |_\frakU$
is given by an abelian deformation of the category
$\Qch(X)$
of quasi-coherent sheaves.   
\end{lem}

\subsection{Toda's construction revisited}
Let
$\frakU$
be a small category
and
$(\scrA, m, f)$
a presheaf of $\bfR$-algebras on
$\frakU$.
The
\emph{simplicial complex of presheaves}
associated with
$\scrA$
is the complex
$(\scrA^\bullet, \varphi^\bullet)$
defined as follows.
Consider the presheaf of algebras 
$\scrA^n = (\scrA^n, m^n, f^n)$
for
$n \geq 0$
given by
\begin{align*}
\scrA^n(U)
=
\prod_{\tau \in \cN_n(\frakU / U)} \scrA|_U(\tau)
\end{align*}
endowed with the product algebra structure
$m^{n,U}$.
Here,
$\tau \in \cN_n(\frakU / U)$
is identified with the object
$d \tau \to U \in \frakU / U$
by composing all morphisms of $\tau$,
and 
the restriction map
\begin{align*}
f^{n,U} \colon \scrA^n(U) \to \scrA^n(V), \ 
(a^\tau)_\tau \mapsto (a^{u \sigma})_\sigma
\end{align*}
is induced by the natural map
$\cN_n(\frakU / V) \to \cN_n(\frakU / U), \ 
\sigma \to  u \sigma$.
Define morphisms of presheaves
$\varphi^n \colon \scrA^n \to \scrA^{n+1}$
as
\begin{align*}
\varphi^{n,U}
\colon
\prod_{\tau \in \cN_n(\frakU / U)} \scrA|_U (\tau)
\to
\prod_{\sigma \in \cN_{n+1}(\frakU / U)} \scrA|_U (\sigma), \
(a^\tau)_\tau
\mapsto
\left( f^{u^\sigma_1} (a^{\partial_0 \sigma}) + \sum^{n+1}_{i = 1} (-1)^i a^{\partial_i \sigma} \right)_\sigma 
\end{align*}
which specialize to
\begin{align*}
\varphi^{0,U}
\colon
\prod_{u \colon V \to U} \scrA(V)
\to
\prod_{u \colon V \to U, v \colon W \to V} \scrA(W), \
(a^u)_u
\mapsto
f^v (a^u) - a^{uv}. 
\end{align*}
Then one obtains the complex
$(\scrA^\bullet, \varphi^\bullet)$
with
$\ker(\varphi^0) \cong \scrA$
\cite[Lemma 2.12]{DLL}. 

Using a part of the simplicial complex of presheaves,
one can give an alternative description of Toda's construction.
Since
$\bfR$
is commutative,
by
\cite[Proposition 2.14]{DLL}
every normalized reduced cocycle
\begin{align*}
\phi
=
(\bfm, \bff, \bfc)
=
(m_1, \ldots, m_l, f_1, \ldots, f_l, c_1, \ldots, c_l)
\in
\bar{\bfC}^{\prime 0,2}_{GS}(\scrA)^{\oplus l}
\oplus
\bar{\bfC}^{\prime 1,1}_{GS}(\scrA)^{\oplus l}
\oplus
\bar{\bfC}^{\prime 2,0}_{GS}(\scrA)^{\oplus l}.
\end{align*}
admits a weak decomposition
\begin{align*}
(\bfm, \bff, \bfc)
=
(\bfm, \bff, 0)
+
(0, 0, \bfc)
\in
\bar{\bfC}^{\prime 2}_{tGS}(\scrA)^{\oplus l}
\oplus
\bar{\bfC}^{\prime 2}_{simp}(\scrA)^{\oplus l}.
\end{align*}
From
\cite[Proposition 2.24]{DLL}
it follows that
the twisted $\bfS$-deformation
$\scrA_\phi$
of
$\scrA$
along
$\phi$
has central twists
and
the underlying presheaf
$|\scrA_\phi|$
is the presheaf $\bfS$-deformation of
$\scrA$
along
$|\phi| = (\bfm, \bff, 0)$.
Consider the morphism
$F \colon \scrA \oplus (\scrA^0)^{\oplus l} \to (\scrA^1)^{\oplus l}$
of presheaves defined as
\begin{align*}
\begin{gathered}
F^U
\colon
\scrA(U) \oplus \prod_{u \colon V \to U} \scrA^0 (V)^{\oplus l}
\to
\prod_{u \colon V \to U, v \colon W \to V} \scrA^1 (W)^{\oplus l}, \\
(a, (b^u_1, \ldots, b^u_l)_u)
\mapsto
(f^v_1 u^*(a) + v^*(b^u_1) - b^{uv}_1, \ldots, f^v_l u^*(a) + v^*(b^u_l) - b^{uv}_l)_{u, v}, 
\end{gathered}
\end{align*}
where we denote
$f^u$
by
$u^*$
for clarity. 
Define the multiplication on
$\scrA \oplus (\scrA^0)^{\oplus l}$
as
\begin{align*}
&(a, (b^u_1, \ldots, b^u_l)_u)
\cdot
(a^\prime, (b^{\prime u}_1, \ldots, b^{\prime u}_l)_u) \\
=
&(aa^\prime, (u^*(a)b^{\prime u}_1 + b^u_1 u^*(a^\prime) + m_1(u^*(a), u^*(a^\prime)), \ldots, u^*(a)b^{\prime u}_l + b^u_l u^*(a^\prime) + m_l(u^*(a), u^*(a^\prime)))_u). 
\end{align*}
With the scalar given by
\begin{align*}
(\lambda + \kappa_1 \epsilon_1 + \ldots + \kappa_l \epsilon_l)(a, (b^u_1, \ldots, b^u_l)_u)
=
(\lambda a, (\kappa_1 u^*(a) + \lambda b^u_1, \ldots, \kappa_l u^*(a) + \lambda b^u_l)_u),
\end{align*}
$\scrA \oplus (\scrA^0)^{\oplus l}$
becomes an $\bfS$-algebra.
Then the morphism
$G \colon |\scrA_\phi| \to \scrA \oplus (\scrA^0)^{\oplus l}$
of presheaves of $\bfS$-algebras defined as
\begin{align*}
\begin{gathered}
G^U
\colon
|\scrA_\phi|(U) \to \scrA(U) \oplus \scrA^0(U)^{\oplus l}, \\
a + b^u_1 \epsilon_1 + \cdots + b^u_l \epsilon_l
\mapsto
(a, (f^u_1(a) + u^*(b_1), \ldots, f^u_l(a) + u^*(b_l))_u)
\end{gathered}
\end{align*}
yields an exact sequence
\begin{align*}
0 
\to
|\scrA_\phi|
\xrightarrow{G}
\scrA \oplus (\scrA^0)^{\oplus l}
\xrightarrow{F}
(\scrA^1)^{\oplus l}.
\end{align*}

Consider the case
where
$\scrA$
is the restricted structure sheaf
$\scrO_X |_\frakU$ 
of a smooth proper $\bfR$-scheme
$X$.
Fix a finite affine open cover
$\frakU$
of
$X$
closed under intersections.
As explained above,
$\scrO_X |_\frakU$
gives a quasi-compact right semi-separated presheaf of $\bfR$-algebras.
Since
$\scrO_X(U)$
is smooth $\bfR$-algebra for each
$U \in \frakU$,
we may assume further that
$\phi = (\bfm, \bff, \bfc)$
is decomposable.
We use the same symbol
$\phi$
to denote the cocycle
\begin{align*}
(\alpha, \beta, \gamma)
\in
H^2(\scrO_X)
\oplus
H^1(\scrT_{X/\bfR})
\oplus
H^0(\wedge^2 \scrT_{X/\bfR})
\end{align*}
which is the image of
$(\bfm, \bff, \bfc)$
under the bijection
\pref{eq:GSHT}.
Then
$\phi$
defines the $\bfS$-linear abelian category
$\Qch(X, \phi)$
obtained by Toda's construction.

\begin{lem} {\rm{(}\cite[Theorem 5.12]{DLL}\rm{)}} \label{lem:Bridge3}
For a smooth proper $\bfR$-scheme
$X$
with a finite affine open cover
$\frakU$
closed under intersections,
let
$(\scrO_X |_\frakU)_\phi$
be the twisted $\bfS$-deformation of the restricted structure sheaf
$\scrO_X |_\frakU$
along a normalized reduced decomposable cocycle
\begin{align*}
\phi
=
(\bfm, \bff, \bfc)
\in
\bar{\bfC}^{\prime 0,2}_{GS}(\scrO_X |_\frakU)^{\oplus l}
\oplus
\bar{\bfC}^{\prime 1,1}_{GS}(\scrO_X |_\frakU)^{\oplus l}
\oplus
\bar{\bfC}^{\prime 2,0}_{GS}(\scrO_X |_\frakU)^{\oplus l},
\end{align*}
which maps to a cocycle
\begin{align*}
(\alpha, \beta, \gamma)
\in
H^2(\scrO_X)
\oplus
H^1(\scrT_{X/\bfR})
\oplus
H^0(\wedge^2 \scrT_{X/\bfR})
\end{align*}
under the bijection
\pref{eq:GSHT}.
Then there is an equivalence
\begin{align*}
\Qch((\scrO_X |_\frakU)_\phi) \simeq \Qch(X, \phi)
\end{align*}
of $\bfS$-linear Grothendieck abelian categories,
where
$\Qch(X, \phi)$
is the abelian category obtained by Toda's construction from
$\Qch(X)$
along
$(\alpha, \beta, \gamma)$.
\end{lem}

By
\pref{eq:GSHT},
\pref{lem:CLASStw},
and
\pref{lem:Bridge2}
we obtain a bijection
\begin{align} \label{eq:HTab}
HT^2(X/\bfR)^{\oplus l}
\cong
HH^2(X/\bfR)^{\oplus l}
\cong
HH^2_{ab}(\Qch(X))^{\oplus l}.
\end{align}
Let
$\Qch(X)_\phi$
the flat abelian $\bfS$-deformation of
$\Qch(X)$
along the image of
$(\alpha, \beta, \gamma)$
under
\pref{eq:HTab}.
Combining
\pref{lem:Bridge2}
and
\pref{lem:Bridge3},
we obtain

\begin{prop} 
For a smooth proper $\bfR$-scheme
$X$,
let
$\Qch(X)_\phi$
be the flat abelian $\bfS$-deformation of
$\Qch(X)$
and
$\Qch(X, \phi)$
the abelian category obtained by Toda's construction from
$\Qch(X)$
corresponding to
$[\phi] \in HH^2(X/\bfR)^{\oplus l}$
via the isomorphism
\pref{eq:HTab}.
Then there is an equivalence
\begin{align*}
\Qch(X)_\phi \simeq \Qch(X, \phi)
\end{align*}
of $\bfS$-linear Grothendieck abelian categories.
\end{prop}

Now,
we return to our setting.
Let
$X_0$
be a Calabi--Yau manifold with
$\dim X_0 > 2$
and
$(X, i_\bfR)$
an $\bfR$-deformation of
$X_0$.
Since we have
\begin{align*}
HT^2(X/\bfR)
=
H^2(\scrO_X/\bfR) \oplus H^1(\scrT_X/\bfR) \oplus H^0(\wedge^2 \scrT_X/\bfR)
\cong
H^1(\scrT_X/\bfR),
\end{align*}
every cocycle
$\phi \in HH^2(X/\bfR)^{\oplus l}$
defines an $\bfS$-deformation
$(X_\phi, i_\bfS)$
of
$(X, i_\bfR)$.
By Proposition
\pref{prop:QchCYDef}
we have
$\Qch(X, \phi) \simeq \Qch(X_\phi)$.
Along square zero extensions,
deforming Calabi--Yau manifolds
and
taking the category of quasi-coherent sheaves
are compatible in the following sense.

\begin{cor} \label{cor:Intertwine1}
Let
$X_0$
be a Calabi--Yau manifold with
$\dim X_0 > 2$,
$(X, i_\bfR)$
an $\bfR$-deformation of
$X_0$,
and
$(X_\phi, i_\bfS)$
the $\bfS$-deformation of
$(X, i_\bfR)$
corresponding to
$[\phi] \in HH^2(X/\bfR)^{\oplus l}$.
Then there is an equivalence
\begin{align*}
\Qch(X)_\phi \simeq \Qch(X_\phi)
\end{align*}
of $\bfS$-linear Grothendieck abelian categories,
where
$\Qch(X)_\phi$
is the flat abelian $\bfS$-deformation of
$\Qch(X)$
corresponding to
$[\phi]$
via the isomorphism
\pref{eq:HTab}.
\end{cor}

\begin{rmk}
Since we have
$HT^2(X_\phi/\bfS)
\cong
H^1(\scrT_{X_\phi}/\bfS)$
by Calabi--Yau condition
and
the finite affine open cover
$\frakU = \{ U_i \}^N_{i=1}$
of
$X$
closed under intersections canonically lifts to the locally trivial deformation
$\frakU \times_\bfR \bfS
=
\{ U_i \times_\bfR \bfS \}^N_{i=1}$,
one may iteratively use Corollary
\pref{cor:Intertwine1}
along a sequence of square zero extensions.  
\end{rmk}

\section{Deformations of the dg category of perfect complexes} \label{sec:dgDefPerf}
In this section,
we review the deformation theory of dg category following the exposition from
\cite{Low08}
and
\cite{KL}.
Based on the ideas thereof,
for a higher dimensional Calabi--Yau manifold
we prove the compatibility of deformations with taking the dg category of perfect complexes.
Namely,
the dg category of perfect complexes on a deformation
is Morita equivalent to
the corresponding dg deformation of a certain full dg category determined by the direction of the deformation.

\subsection{Curved $A_\infty$-categories}
In the sequel,
by a
\emph{quiver}
we will mean a $\bZ$-graded quiver. 
We choose shift functors
$\Sigma^k$
on the category
$G(\bfR)$
of $\bZ$-graded $\bfR$-modules.
Let
$\fraka$
be an $\bfR$-linear quiver.
Namely,
$\fraka$
consists of a set
$\Ob(\fraka)$
of objects
and
a $\bZ$-graded $\bfR$-module
$\fraka(A, A^\prime)$
for each pair
$A, A^\prime \in \Ob(\fraka)$.
The category of quivers with a fixed set of objects admits a tensor product
\begin{align*}
\fraka \otimes \frakb (A, A^\prime)
=
\bigoplus_{A^{\prime \prime}} \fraka (A^{\prime \prime}, A^{\prime}) \otimes_\bfR \frakb (A, A^{\prime \prime})
\end{align*}
and
an internal hom
\begin{align*}
[\fraka, \frakb] (A, A^\prime)
=
[\fraka (A, A^{\prime}), \frakb (A, A^{\prime})].
\end{align*}
Morphisms of degree
$k$
are elements of
$[\fraka, \frakb]^k = \prod_{A, A^\prime} [\fraka, \frakb] (A, A^\prime)^k$.

The
\emph{tensor cocategory}
$T(\fraka)$
of
$\fraka$
is the quiver
\begin{align*}
T(\fraka) = \bigoplus_{n \geq 0} \fraka^{\otimes n}
\end{align*}
equipped with the comultiplication
which separates tensors.
There is a natural brace algebra structure on
$[T(\fraka), \fraka] = \prod_{n \geq 0} [T(\fraka), \fraka]_n$,
where
\begin{align*}
[T(\fraka), \fraka]_n
=
[\fraka^{\otimes n}, \fraka]
=
\prod_{A_0, \ldots, A_n \in \fraka} [\fraka (A_{n-1}, A_n) \otimes_\bfR \cdots \otimes_\bfR \fraka (A_0, A_1), \fraka (A_0, A_n)].
\end{align*}
It is given by the operations
\begin{align*}
\begin{gathered}
{}[T(\fraka), \fraka]_n \otimes_\bfR [T(\fraka), \fraka]_{n_1} \otimes_\bfR \cdots \otimes_\bfR [T(\fraka), \fraka]_{n_i}
\to
[T(\fraka), \fraka]_{n-i+n_1+ \cdots + n_i}, \\
(\phi, \phi_1, \ldots, \phi_i) \mapsto \phi \{ \phi_1, \ldots, \phi_i \}
\end{gathered}
\end{align*}
with
\begin{align*}
\phi \{ \phi_1, \ldots, \phi_i \}
=
\sum \phi(1 \otimes \cdots \otimes \phi_1 \otimes 1 \otimes \cdots \otimes \phi_i \otimes 1 \otimes \cdots \otimes 1)
\end{align*}
satisfying
\begin{align*}
\phi \{ \phi_1, \ldots, \phi_i \} \{ \psi_1, \ldots, \psi_j \}
=
\sum (-1)^\alpha \phi \{\psi_1, \ldots, \phi_1 \{ \psi_{m_1}, \ldots  \}, \psi_{n_1}, \ldots, \phi_i \{ \psi_{m_i} \ldots \}, \psi_{n_i}, \ldots, \psi_j \},
\end{align*}
where
$\alpha = \sum^i_{k=1} |\phi_k| \sum^{m_k-1}_{l=1} |\psi_l|$.
We denote by
$B \fraka$
the Bar cocategory
$T(\Sigma \fraka)$
and
by
$\bfC_{br}(\fraka)$
the brace algebra
$[B \fraka, \Sigma \fraka]$.
The
\emph{associated Hochschild object}
is defined as
$\bfC(\fraka) = \Sigma^{-1} \bfC_{br}(\fraka)$.
By
\cite[Proposition 2.2]{Low08}
the tensor coalgebra
$T(\bfC_{br}(\fraka)) = B \bfC(\fraka)$
becomes a graded bialgebra with the associative multiplication defined by the composition.

A
\emph{curved $A_\infty$-structure}
on
$\fraka$
is an element
$b \in \bfC^1_{br}(\fraka)$
satisfying
$b \{ b \} = 0$.
The pair
$(\fraka, b)$
is called a
\emph{curved $A_\infty$-category}.
When the defining morphisms
$b_n \colon \Sigma \fraka^{\otimes n} \to \Sigma \fraka$
vanish for
$n \geq 3$,
we call 
$(\fraka, b)$
a
\emph{cdg category}.
The
\emph{curvature elements}
of
$(\fraka, b)$
is the morphism
$b_0$.
When it vanishes,
we drop ``curved'' and ``c'' from the notation.

\begin{dfn} {\rm{(}\cite[Definition 2.5]{Low08}\rm{)}}
For curved $A_\infty$-categories
$(\fraka, b), (\fraka^\prime, b^\prime)$
with
$\Ob(\fraka) = \Ob(\fraka^\prime)$
a
\emph{morphism}
is a fixed object morphism of quivers
$f \colon B \fraka \to B \fraka^\prime$,
which is determined by morphisms
$f_n \colon (\Sigma \fraka)^{\otimes n} \to \Sigma \fraka^\prime$
and
respects the comultiplications and the curved $A_\infty$-structures.
\end{dfn}

\subsection{Hochschild complexes of curved $A_\infty$-categories}
The
\emph{associated Lie bracket} 
with the brace algebra
$\bfC_{br}(\fraka)$
is defined as
\begin{align*}
\langle \phi, \psi \rangle
=
\phi \{ \psi \} - (-1)^{|\phi| |\psi|} \psi \{ \phi \}.
\end{align*}
Via an isomorphism
\begin{align*}
\bfC_{br}(\fraka) \cong \Coder(B\fraka, B \fraka) 
\end{align*}
of $\bZ$-graded $\bfR$-modules to coderivations between cocategories,
it corresponds to the commutator of coderivations.
For a curved $A_\infty$-structure
$b$
on
$\fraka$
the
\emph{Hochschild differential}
on
$\bfC_{br}(\fraka)$ 
is defined as
\begin{align*}
d_b
=
\langle b, - \rangle
\in
[\bfC_{br}(\fraka), \bfC_{br}(\fraka)]^1, \
\phi \mapsto \langle b, \phi \rangle.
\end{align*}
In particular,
$\bfC_{br}(\fraka)$
can be regarded as a dg Lie algebra.
Then
$\bfC(\fraka)$
is known to be isomorphic to the classical Hochschild complex of
$\fraka$,
whose definition we will review later.
Since
$b$
naturally belongs to
$B \bfC (\fraka)^1$,
it induces a differential
\begin{align*}
D_b
=
[b, - ]
\in
[B \bfC (\fraka), B \bfC (\fraka)]^1, \
\phi \mapsto [b, \phi],
\end{align*}
where
$[-, -]$
is the commutator of the multiplication given by
\cite[Proposition 2.2]{Low08}.
As
$D_b$
belongs to
$\Coder(B \bfC (\fraka), B \bfC (\fraka))$,
it defines a curved $A_\infty$- structure on
$\bfC(\fraka)$.
The differential
$D_b$
gives a dg bialgebra structure on
$B \bfC (\fraka)$
and
$\bfC(\fraka), \bfC_{br}(\fraka)$
become $B_\infty$-algebras
\cite[Definition 5.2]{GJ}.

We will use the same symbol
$\bfC(\fraka)$
to denote the bigraded object with
\begin{align*}
\bfC^{p, q}(\fraka)
=
\prod_{A_0, \ldots, A_q \in \fraka}[\fraka(A_{q-1}, A_q) \otimes_\bfR \cdots \otimes_\bfR \fraka(A_0, A_1), \fraka(A_0, A_q)]^p.
\end{align*}
An element
$\phi \in \bfC^{p,q}(\fraka)$
is said to have the
\emph{degree}
$|\phi| = p$,
the
\emph{arity}
$\AR(\phi) = q$,
and the
\emph{Hochschild degree}
$\deg(\phi) = n = p + q$.
The total complex of Hochschild degree
$n$
is defined as
$\bfC^n(\fraka) = \prod_{p+q=n}\bfC^{p,q}(\fraka)$.
Via the canonical isomorphisms
\begin{align*}
&\Sigma^{1-q} [\fraka(A_{q-1}, A_q) \otimes_\bfR \cdots \otimes_\bfR \fraka(A_0, A_1), \fraka(A_0, A_q)] \\
\to
[&\Sigma \fraka(A_{q-1}, A_q) \otimes_\bfR \cdots \otimes_\bfR \Sigma \fraka(A_0, A_1), \Sigma \fraka(A_0, A_q)],
\end{align*}
the $B_\infty$-structure on
$\bfC_{br}(\fraka)$
is translated in terms of
$\fraka$.
For instance,
the operation
\begin{align*}
\DOT
\colon
\bfC_{br}(\fraka)_q \otimes \bfC_{br}(\fraka)_s
\to
\bfC_{br}(\fraka)_{q+s-1}, \
(\phi, \psi)
\mapsto
\phi \{ \psi \}
\end{align*}
induces the classical ``dot product''
\begin{align*}
\bullet
\colon
\bfC^{p,q}(\fraka) \otimes \bfC^{r,s}(\fraka)
\to
\bfC^{p+r, q+s-1}(\fraka)
\end{align*}
on
$\bfC(\fraka)$
given by
\begin{align*}
\phi \bullet \psi
=
\sum^{q-1}_{k=0} (-1)^\beta \phi(1^{\otimes q-k-1} \otimes \psi \otimes 1^{\otimes k}),
\end{align*}
where
$\beta = (\deg(\phi) + k +1)(\AR(\psi) + 1)$.
We also call the bigraded object
$\bfC(\fraka)$
the Hochschild complex of
$\fraka$
and
its elements Hochschild cochains.
In the sequel,
curved $A_\infty$-structure on
$\fraka$
will often be translated into an element of
$\bfC^2(\fraka)$
without further comments.

\subsection{Curved dg category of precomplexes}
Let
$\fraka$
be an $\bfR$-linear category.
Consider the category
$\PCom (\fraka)$
of precomplexes of $\fraka$-objects.
A
\emph{precomplex}
of $\fraka$-objects is a $\bZ$-graded $\fraka$-objects
$C$
with
$C^i \in \fraka$
together with a
\emph{predifferential},
a $\bZ$-graded $\fraka$-morphism
$\delta_C \colon C \to C$
of degree 1.
For any
$C, D \in \PCom(\fraka)$
the Hom-set
$\PCom(\fraka)(C,D)$
is a $\bZ$-graded $\bfR$-module with
\begin{align*}
\PCom(\fraka)(C,D)^k = \prod_{i \in \bZ}\fraka(C^i ,D^{i + k}).
\end{align*}
The cdg structure
$\mu \in \bfC(\fraka)^2$
on
$\PCom(\fraka)$
consists of compositions
$m$,
differentials
$d$,
and
curvature elements
$c$,
where
\begin{align*}
\begin{gathered}
m
=
\mu_2 \in \prod_{C_0, C_1, C_2 \in \PCom(\fraka)}[\PCom(\fraka)(C_1, C_2) \otimes_\bfR \PCom(\fraka)(C_0, C_1), \PCom(\fraka)(C_0, C_2)]^0  ,\\
d
=
\mu_1 \in \prod_{C_0, C_1 \in \PCom(\fraka)}[\PCom(\fraka)(C_0, C_1), \PCom(\fraka)(C_0, C_1)]^1  ,\\
c
=
\mu_0 \in \prod_{C \in \PCom(\fraka)} \PCom(\fraka)(C, C)^2
\end{gathered}
\end{align*}
are given by
\begin{align*}
\begin{gathered}
m(g, f)_i = (gf)_i = g_{i + |f|}f_i \colon C^i_0 \to C^{i + |f| + |g|}_2, \\
d(f) = \delta_{C_1} f - (-1)^{|f|} f \delta_{C_0}, \\
c_C = - \delta^2_C
\end{gathered}
\end{align*}
for morphisms
$f \colon C_0 \to C_1, g \colon C_1 \to C_2$
in
$\PCom(\fraka)$.
One can check that
$m$,
$d$,
and
$c$
satisfy
\begin{align*}
\begin{gathered}
d(c) = 0, \\
d^2 = -m(c \otimes 1 - 1 \otimes c), \\
dm = m(d \otimes 1 + 1 \otimes d), \\
m(m \otimes 1) = m(1 \otimes m).
\end{gathered}
\end{align*}
We denote by
$\Com(\fraka)$
the full dg subcategoy of complexes of $\fraka$-objects,
where
$\delta_C$
become differentials.

Here,
we demonstrate
how the cdg structure is translated.
The differential
\begin{align*}
d_b
=
\langle b, - \rangle
=
\langle \Sigma c + d + \Sigma^{-1} m, - \rangle
\in
[\bfC_{br}(\PCom(\fraka)), \bfC_{br}(\PCom(\fraka))]^1
\end{align*}
on
$\bfC_{br}(\PCom(\fraka))$
sends
$\Sigma^{1-q} \phi \in \bfC_{br}(\PCom(\fraka))$
with
$\phi \in \bfC^{p,q}(\PCom(\fraka))$
to
\begin{align*}
\DOT(\Sigma c + d + \sigma^{-1} m, \Sigma^{1-q} \phi)
-
(-1)^{1-q+p}\DOT(\Sigma^{1-q} \phi, \Sigma c + d + \Sigma^{-1} m).
\end{align*}
In terms of
$\bfC(\fraka)$
the image corresponds to
$[c + d + 
m, \phi]$,
where
\begin{align*}
\begin{gathered}
{}[c, \phi]
=
\sum^{q-1}_{k=0}(-1)^{k+1} \phi(1^{\otimes q-k-1} \otimes c \otimes 1^{\otimes k}), \\
[d, \phi]
=
(-1)^{\AR(\phi) + 1} d(\phi)
+
\sum^{q-1}_{k=0}(-1)^{\deg(\phi)} \phi(1^{\otimes q-k-1} \otimes c \otimes 1^{\otimes k}), \\
[m, \phi]
=
m(\phi \otimes 1)
+
(-1)^{\AR(\phi) + 1} m(1 \otimes \phi)
+
\sum^{q-1}_{k=0}(-1)^{k+1} \phi(1^{\otimes q-k-1} \otimes c \otimes 1^{\otimes k}).
\end{gathered}
\end{align*}

\subsection{Curved dg deformations of dg categories}
Assume that
$\fraka$
is an $\bfR$-linear cdg category.
A
\emph{cdg $\bfS$-deformation}
of
$\fraka$
is an $\bfS$-linear cdg structure on an $\bfS$-linear quiver
$\frakb$
together with an isomorphism
$\frakb \to \fraka[\epsilon] = \fraka \otimes_\bfR \bfS$
of $\bfS$-linear quivers
whose reduction
$\frakb \otimes_\bfS \bfR \to \fraka$
induces an isomorphism of cdg categories.
Two cdg deformations
$\frakb, \frakb^\prime$
are
\emph{isomorphic}
if there is an isomorphism
$\frakb \to \frakb^\prime$
of cdg categories inducing the identity on
$\fraka$.
We denote by
$\Def^{cdg}_\fraka(\bfS)$
the set of isomorphism classes of cdg $\bfS$-deformations of
$\fraka$.

\begin{thm} {\rm{(}\cite[Theorem 4.11]{Low08}\rm{)}}
Let
$(\fraka, \mu)$
be an $\bfR$-linear cdg category.
Then there is a bijection
\begin{align} \label{eq:classifying}
H^2 \bfC(\fraka)^{\oplus l}
\to
\Def^{cdg}_\fraka(\bfS), \
\phi
\mapsto
(\fraka[\epsilon], \mu + \phi \epsilon), \
\phi \in Z^2 \bfC(\fraka)^{\oplus l}.
\end{align}
Another cocycle
$\phi^\prime \in Z^2 \bfC(\fraka)^{\oplus l}$
maps to an isomorphic cdg deformation
if and only if
there is an element
$h \in \bfC^1(\fraka)$
satisfying
$\phi^\prime - \phi = d_\mu (h)$.
\end{thm}

A
\emph{partial cdg $\bfS$-deformation}
of
$\fraka$
is a cdg $\bfS$-deformation of some full cdg subcategory
$\fraka^\prime$.
Two partial cdg deformations
$\frakb, \frakb^\prime$
are
\emph{isomorphic}
if there is an isomorphism
$\frakb \to \frakb^\prime$
of cdg categories inducing the identity on
$\fraka^\prime$.
A
\emph{morphism}
of partial cdg deformations
$\frakb, \frakb^\prime$
is an isomorphism of cdg deformations between
$\frakb$
and
a full cdg subcategory of
$\frakb^\prime$.
When every morphism of
$\frakb \to \frakb^\prime$
of partial cdg deformations is an isomorphism,
we call
$\frakb$
\emph{maximal}.
We denote by
$\Def^{p-cdg}_\fraka(\bfS)$
the set of morphism classes of partial cdg $\bfS$-deformations of
$\fraka$
and
by
$\Def^{mp-cdg}_\fraka(\bfS)$
the set of isomorphism classes of maximal partial cdg $\bfS$-deformations of
$\fraka$.

Assume further that
$\fraka$
is a dg category.
For
$\phi \in Z^2 \bfC(\fraka)$
the
\emph{$[\phi]-\infty$ part}
of
$\fraka$
is the full dg subcategory
$\fraka_{[\phi]-\infty}$
spanned by objects
$A \in \fraka$
satisfying
\begin{align*}
H^2(\pi_0)([\phi])_A = 0 \in H^2( \fraka(A,A) ),
\end{align*}
where
$\pi_0 \colon \bfC(\fraka) \to \bfC(\fraka)_0$
is the projection onto the
\emph{zero part}
\begin{align*}
\bfC(\fraka)_0
=
\Sigma^{-1}[T(\Sigma \fraka), \Sigma \fraka]_0
=
\prod_{A \in \fraka} \fraka(A,A).
\end{align*}
For a cdg $\bfS$-deformation
$\frakb = (\bar{\fraka}, \mu + \phi \epsilon)$
of
$\fraka$, 
the
\emph{$\infty$-part}
$\frakb_\infty$
is the full cdg subcategory
spanned by objects
$B \in \frakb$
satisfying
\begin{align*}
(\mu + \phi \epsilon)_{0,B} = 0 \in \frakb(B,B)^2.
\end{align*}
It is a partial dg deformation of
$\fraka$
and
a dg deformation of
$\fraka_{[\phi] - \infty}$.
More explicitly,
if we restrict
$\phi$
to
$\fraka_{[\phi]-\infty}$,
then 
$\phi_0$
becomes a coboundary
and
there is an element
\begin{align*}
h
\in
\prod_{A \in \fraka_{[\phi]-\infty}} \fraka(A, A)^1
\subset
\bfC^1(\fraka)
\end{align*}
with
$d_\mu(h) = (\phi|_{\fraka_{[\phi]-\infty}})_0$.
Thus
the cocycle
$\phi|_{\fraka_{[\phi]-\infty}} - d_\mu (h)$
has trivial curvature elements.

\begin{prop} {\rm{(}\cite[Proposition 4.14]{Low08}\rm{)}} 
Let
$(\fraka, \mu)$
be an $\bfR$-linear dg category.
Then there is a map
\begin{align} \label{eq:Hoch-pdg}
H^2 \bfC(\fraka)^{\oplus l}
\to
\Def^{p-dg}_\fraka(\bfS), \ 
\phi
\mapsto
(\fraka_{[\phi]-\infty}[\epsilon], \mu + (\phi|_{\fraka_{[\phi]-\infty}} - d_\mu (h)) \epsilon).
\end{align}
\end{prop}

\subsection{The characteristic morphism}
Let
$\fraki$
be an $\bfR$-linear category.
The canonical projection
\begin{align*} 
\pi
\colon
\bfC (\PCom(\fraki))
\to
\bfC (\fraki)
\end{align*}
has a $B_\infty$-section
\begin{align} \label{eq:embr}
\embr_\delta
\colon
\bfC(\fraki)
\to
\bfC(\PCom(\fraki)),
\end{align}
whose restriction to the full dg subcategory
$\Com^+(\fraki)$
of bounded below complexes of $\fraki$-objects
is an inverse in the homotopy category
$\Ho(B_\infty)$
of $B_\infty$-algebras
\cite[Theorem 3.21]{Low08}.
Consider the projection onto the zero part
\begin{align} \label{eq:proj-0}
\pi_0
\colon
\bfC (\Com(\fraki))
\to
\bfC (\Com(\fraki))_0
=
\prod_{C \in \Com(\frak i)} \Com(\fraki)(C,C).
\end{align}
Since
$\Com(\fraki)$
is uncurved,
$\pi_0$
induces a morphism of dg algebras 
\cite[Proposition 2.7]{Low08}.
Composing 
\pref{eq:embr}
and
\pref{eq:proj-0},
one obtains the
\emph{characteristic dg morphism}
\begin{align*}
\bfC (\fraki)
\to
\prod_{C \in \Com(\fraki)} \Com(\fraki)(C, C).
\end{align*}
On cohomology it induces the
\emph{characteristic morphism}
\begin{align*}
\chi_\fraki
\colon
H^\bullet \bfC(\fraki)
\to
\frakZ^\bullet K(\fraki)
\end{align*}
for a linear category
$\fraki$.
Here,
$\frakZ^\bullet K(\fraki)$
is the
\emph{graded center}
of
$K(\fraki)$,
i.e.,
the center
\begin{align*}
\frakZ (\Com (\fraki))
=
\Hom(1_{\Com (\fraki)}, 1_{\Com (\fraki)})
\end{align*}
of the graded category
$\Com (\fraki)$,
where
$\Hom$
denotes the graded $\bfR$-module of graded natural transformations
\cite[Remark 4.6]{Low08}.
The characteristic morphism can be interpreted
in terms of deformations of categories.

\begin{thm} {\rm{(}\cite[Theorem 4.8]{Low08}\rm{)}} \label{thm:obst} 
Let
$\fraki$
be an $\bfR$-linear category
and
$\fraki_\phi$
its $\bfS$-deformation along
$\phi \in Z^2 \bfC (\fraki)^{\oplus l}$.
Then
for each
$C \in K (\fraki)$
the element
$\chi^{\oplus l}_\fraki (\phi)_C \in K(\fraki)(C,C)[2]^{\oplus l}$
is the obstruction against deforming
$C$
to an object of
$K(\fraki_\phi)$.
\end{thm}

Let
$\scrC$
be an $\bfR$-linear abelian category with enough injectives.
Assume that
$\fraki$
is the subcategory
$\Inj(\scrC)$
of injective objects.
Taking cohomology of
\pref{eq:proj-0}
restricted to
$\Com^+(\fraki)$
and
composing with the isomorphism
$HH^\bullet_{ab}(\scrC)
\cong
H \bfC^\bullet (\Com^+(\fraki))$
induced by
\pref{eq:embr},
one obtains the
\emph{characteristic morphism}
\begin{align*}
\chi_\scrC
\colon
HH^\bullet_{ab}(\scrC)
\to
\frakZ^\bullet (D^+(\scrC))
\end{align*}
for an abelian category
$\scrC$.
Here,
we use the isomorphism
$HH^\bullet_{ab}(\scrC) \cong H\bfC^\bullet (\fraki)$
obtained from
\cite[Theorem 6.6]{LV06}.
Note that
the graded center
$\frakZ^\bullet (D^+(\scrC))$
of
$D^+(\scrC) \simeq K^+(\fraki)$
is given by the center
$\frakZ(\Com^+(\fraki))$.

\begin{cor} {\rm{(}\cite[Corollary 4.9]{Low08}\rm{)}} \label{cor:obst} 
Let
$\scrC$
be an $\bfR$-linear abelian category with enough injectives
and
$\scrC_\phi$
its abelian $\bfS$-deformation along a cocycle
$\phi \in HH^2_{ab}(\scrC)^{\oplus l}$.
Then
for each
$C \in D^+(\scrC)$
the element
$\chi^{\oplus l}_\scrC (\phi)_C \in \Ext^2_\scrC(C,C)^{\oplus l}$
is the obstruction against deforming
$C$
to an object of
$D^+(\scrC_\phi)$.
\end{cor}

\subsection{Maximal partial dg deformations of the dg category of bounded below complexes}
Consider the map
\begin{align} \label{eq:rho}
\rho^\prime
\colon
H^2 \bfC(\fraki)^{\oplus l}
\to
\Def^{p-dg}_{\Com^+(\fraki)}(\bfS), \
\phi
\mapsto
(\Com^+(\fraki)_{\embr_\delta(\phi)})_\infty
\end{align}
obtained from
\pref{eq:Hoch-pdg}
and
\pref{eq:embr}.
The partial dg deformation
$(\Com^+(\fraki)_{\embr_\delta(\phi)})_\infty$
of
$\Com^+(\fraka)$
coincides with a dg deformation
$(\Com^+(\fraki)_{[\embr_\delta(\phi)]-\infty})_{\embr_\delta(\phi)}$
of
$(\Com^+(\fraki)_{[\embr_\delta(\phi)]-\infty})$,
where
the cdg structure
$\embr_\delta(\phi)$
restricted to the
$[\embr_\delta(\phi)]-\infty$
part.

\begin{thm} {\rm{(}\cite[Theorem 4.15(iii)]{Low08}\rm{)}} \label{thm:lift} 
Let
$\fraki$
be an $\bfR$-linear category
and
$\fraki_\phi$
its $\bfS$-deformation along
$\phi \in Z^2 \bfC (\fraki)^{\oplus l}$.
Then,
for every collection of complexes
$\Gamma = \{ \bar{C} \}_{C \in \fraki_{[\phi]-\infty}}$
with
$\bar{C} \otimes_\bfS \bfR = C$,
the full dg subcategory
$\Com^+_\Gamma (\fraki_\phi)
\subset
\Com^+ (\fraki_\phi)$
spanned by
$\Gamma$
is a maximal partial dg $\bfS$-deformation of
$\Com^+(\fraki)$
representing
$\rho^\prime(\phi)$.
\end{thm}

From the proof,
one sees that
$\Com^+_\Gamma (\fraki_\phi)$
is a dg deformation of
$\Com^+ (\fraki)_{[\embr(\phi)] - \infty}$.
According to
\cite[Example 4.13]{Low08},
an object
$C \in \Com^+(\fraki)$
belongs to
$\Com^+ (\fraki)_{[\embr(\phi)] - \infty}$
if and only if
\begin{align*}
\chi^{\oplus l}_\fraki (\phi)_C = 0 \in K(\fraki)(C, C[2])^{\oplus l}.
\end{align*}
Since
$\chi^{\oplus l}_\fraki (\phi)_C$
is the obstruction against deforming
$C$
to an object of
$\Com^+(\fraki_\phi)$,
we may take
a collection of one chosen lift of each unobsteucted complexe in
$\Com^+(\fraki)$
as
$\Gamma$.
Clearly,
any dg subcategory of
$\Com^+ (\fraki)_{[\embr(\phi)] - \infty}$
dg deforms along the restriction of
$\embr_\delta(\phi)$.

\subsection{Hochschild cohomology of the dg category of perfect complexes}
We review the definition of the classical Hochschild complex of dg categories.
Assume that
$\fraka$
is a small $\bfR$-cofibrant dg category,
i.e.,
all Hom-sets are cofibrant in the dg category
$\Mod_{dg}(\bfR) = \Com(\Mod(\bfR))$
of complexes of $\bfR$-modules.
Recall that
$N \in \Mod_{dg}(\bfR)$
is cofibrant
if its terms are projective.
For an $\fraka$-bimodule
$M \colon \fraka^{op} \otimes \fraka \to \Mod_{dg}(\bfR)$,
the
\emph{Hochschild complex}
$\bfC(\fraka, M)$
of
$\fraka$
with coefficients in
$M$
is the total complex of the double complex
whose $q$-th columns are given by
\begin{align*}
\prod_{A_0, \ldots, A_q \in \fraka} \Hom(\fraka(A_{q-1}, A_q) \otimes_\bfR \cdots \otimes_\bfR \fraka(A_1, A_0), M(A_0, A_q))
\end{align*}
with horizontal differentials
$d^q_{Hoch}$.
When
$M = \fraka$
we call
$\bfC(\fraka) = \bfC(\fraka, \fraka)$
the
\emph{Hochschild complex}
of
$\fraka$.
The Hochschild complex satisfies a ``limited functoriality'' property.
Namely,
if
$j \colon \fraka \hookrightarrow \frakb$
is a fully faithful dg functor to a small $\bfR$-cofibrant dg category
$\frakb$,
then there is an associated map between Hochschild complexes
\begin{align*}
j^* \colon \bfC(\frakb) \to \bfC(\fraka)
\end{align*} 
given by restriction.
As mentioned above,
the Hochschild complex is isomorphic to the associated Hochschild object
$\bfC(\fraka) = \Sigma^{-1} \bfC_{br}(\fraka)$.
In particular,
the Hochschild complex
$\bfC(\fraka)$
has a $B_\infty$-algebra structure compatible with the map
$j^*$.

The definition of the Hochschild complex was modified by Shukla and Quillen
\cite{Shu, Qui}
to general dg categories.
Now,
we drop the assumption on
$\fraka$
to be $\bfR$-cofibrant
and
fix a good $\bfR$-cofibrant resolution
$\bar{\fraka} \to \fraka$,
which is a quasi-equivalence with an $\bfR$-cofibrant dg category
$\bar{\fraka}$
inducing surjection of Hom-sets in the graded category
\cite[Proposition-Definition 2.3.2]{LV06a}. 
The
\emph{Shukla complex}
of
$\fraka$
with coefficient
$M$
is defined as
\begin{align*}
\bfC_{sh}(\fraka, M) = \bfC(\bar{\fraka}, M).
\end{align*} 
According to
\cite[Section 4.2]{LV06a},
which in turn is attributed to
\cite{Kel},
the assignment
\begin{align*}
\bfC_{sh} \colon \fraka \mapsto \bfC_{sh}(\fraka)
\end{align*}
defines up to canonical natural isomorphism a contravariant functor 
on a suitable category of small dg categories
with values in
$\Ho (B_\infty)$.
In particular,
$\bfC_{sh}(\fraka)$
does not depend on the choice of good $\bfR$-cofibrant resolutions of
$\fraka$
up to canonical isomorphism.
The functor
$\bfC_{sh}$
satisfies some extended ``limited functoriality''.
Let
$j \colon \fraka \hookrightarrow \frakb$
be a fully faithful dg functor to a small dg category
$\frakb$
with a good $\bfR$-cofibrant resolution
$\bar{\frakb} \to \frakb$.
One may restrict the resolution to a good $\bfR$-cofibrant resolution
$\bar{\fraka} \to \fraka$
of
$\fraka$.
Then the restriction along the extended fully faithful dg functor
$\bar{\fraka} \hookrightarrow \bar{\frakb}$
defines a morphism of Shukla complexes
\begin{align*}
\bfC_{sh}(\frakb) \to \bfC_{sh}(\fraka)
\end{align*}
still denoted by
$j^*$.
In the sequel,
we write
$\bfC(\fraka)$
for  
$\bfC_{sh}(\fraka)$.

Now,
we return to our setting.
Let
$X_0$
be a Calabi--Yau manifold with
$\dim X_0 > 2$.
We denote by
$D_{dg}(\Qch (X_0))$
the dg category of unbounded complexes of quasi-coherent sheaves on
$X_0$.
In our setting,
the full dg subcategory
$\Perf_{dg}(X_0)$
of compact objects consists of perfect complexes on
$X_0$.
The canonical embedding
$\Perf_{dg}(X_0) \hookrightarrow D_{dg}(\Qch (X_0))$
factorizes through the dg category
$D^+_{dg}(\Qch (X_0))$
of bounded below complexes of quasi-coherent sheaves on
$X_0$.
Let
$(X, i_\bfR)$
be an $\bfR$-deformation of
$X_0$
and
$\fraki = \Inj(\Qch(X))$.
As explained above,
$X$
is smooth projective over
$\bfR$.
The Hochschild cohomology of
$\Perf_{dg}(X)$
can be expressed in terms of
$\fraki$.

\begin{lem} \label{lem:Hochschild}
There is an isomorphism
\begin{align} \label{eq:Hochschild}
H \bfC^\bullet (\Com^+ (\fraki)) \to H \bfC^\bullet (\Perf_{dg}(X)).
\end{align}
\end{lem}
\begin{proof}
Consider the quasi-fully faithful functor
\begin{align} \label{eq:qff1}
\Perf_{dg}(X)
\hookrightarrow
D^+_{dg}(\Qch(X))
\to
\Com^+(\fraki),
\end{align}
where
the first functor is the canonical embedding
and
the second functor is a quasi-equivalence induced by the canonical equivalence
\begin{align*}
D^+(\Qch(X)) \to K^+(\fraki)
\end{align*}
of their homotopy categories
\cite[Theorem A]{CNS}.
The functor
\pref{eq:qff1}
induces a morphism
\begin{align} \label{eq:iqff1}
\bfC(\Com^+(\fraki))
\to
\bfC(D^+_{dg}(\Qch(X)))
\to
\bfC(\Perf_{dg} (X))
\end{align}
of $B_\infty$-algebra,
which in turn induces a morphism
\begin{align} \label{eq:iiqff1}
H^\bullet \bfC(\Com^+(\fraki))
\to
H^\bullet \bfC(D^+_{dg}(\Qch(X)))
\to
H^\bullet \bfC(\Perf_{dg} (X))
\end{align}
of Hochschild cohomology.
Since quasi-equivalences preserve Hochschld cohomology,
the first arrow in
\pref{eq:iiqff1}
is an isomorphism.

It remains to show that
the second arrow in
\pref{eq:iiqff1}
is an isomorphism.
We claim that
the restriction
\begin{align*}
\bfC (D_{dg}(\Qch(X))) \to \bfC (\Perf_{dg}(X))
\end{align*}
is an isomorphism in
$\Ho (B_\infty)$. 
Fix a good $\bfR$-cofibrant resolution
\begin{align*}
\overline{D_{dg}(\Qch(X))} \to D_{dg}(\Qch(X)).
\end{align*}
Via the canonical embedding
$\Perf_{dg}(X) \hookrightarrow D_{dg}(\Qch(X))$
it induces a good $\bfR$-cofibrant resolution
$\overline{\Perf_{dg}(X)} \to \Perf_{dg}(X)$
and
the fully faithful embedding
\begin{align*}
\bar{j}
\colon
\overline{\Perf_{dg}(X)}
\hookrightarrow
\overline{D_{dg}(\Qch(X))}.
\end{align*}
We denote by
$j$
the induced fully faithful functor on the homotopy categories.
One can apply
\cite[Theorem 1.2]{Por}
to see that
the functor
\begin{align*}
D(\Qch(X)) \to \Mod(\Perf(X)), \
F \mapsto \Hom_{D(\Qch(X))}(j(-), F)
\end{align*}
lifts to a localization
$D(\Qch(X)) \to \bfD (\Mod_{dg}(\Perf_{dg}(X))$,
where
$\bfD (\Mod_{dg}(\Perf_{dg}(X))$
is the derived category of right dg modules over the dg category 
$\Mod_{dg}(\Perf_{dg}(X))$
of right dg modules over
$\Perf_{dg}(X)$.
In particular,
the lift 
is fully faithful.
Then the claim follows from
\cite[Proposition 5.1]{DL}.
Similarly,
one can show that
the restriction
\begin{align*}
\bfC (D_{dg}(\Qch(X))) \to \bfC (D^+_{dg}(\Qch(X)))
\end{align*}
is an isomorphism in
$\Ho (B_\infty)$. 
Hence the restriction
\begin{align*}
\bfC (D^+_{dg}(\Qch(X))) \to \bfC (\Perf_{dg}(X))
\end{align*}
is an isomorphism in
$\Ho (B_\infty)$,
which induces the desired isomorphism on Hochschild cohomology.
\end{proof}

\subsection{Morita deformations of the dg category of perfect complexes}
Let
$\dgCat_\bfR$
be the category of
small $\bfR$-linear dg categories
and
dg functors.
The category
$\dgCat_\bfR$
has two model structures,
so called
the Dwyer--Kan model structure
and
the Morita model structure,
constructed by Tabuada  respectively in
\cite{Tab}
and
\cite{Tab05}.
On the Dwyer--Kan model structure,
weak equivalences are given by quasi-equivalences of dg categories.
On the Morita model structure,
weak equivalences are given by Morita morphisms. 
Recall that
a dg functor in
$\dgCat_\bfR$
is a
\emph{Morita morphism}
if it induces an derived equivalence.
Also recall that
for each object
$\fraka \in \dgCat_\bfR$
the derived category
$\bfD(\fraka)$
of right dg modules over
$\fraka$ 
is defined as the Verdier quotient
\begin{align*}
[\Mod_{dg}(\fraka)] / [\Acycl(\fraka)]
\end{align*}
of the homotopy category of the dg category
$\Mod_{dg}(\fraka)$
of right dg modules over
$\fraka$
by the homotopy category of the full dg subcategory
$\Acycl(\fraka)$
of acyclic right dg modules.

We denote by
$\Ho_\bfR$
the localization of
$\dgCat_\bfR$
by weak equivalences in the Dwyer--Kan model structure
and
by
$\Hmo_\bfR$
the localization of
$\dgCat_\bfR$
by weak equivalences in the Morita model structure.
Passing to
$\Ho_\bfR$,
two quasi-equivalent $\bfR$-linear dg categories
$\fraka, \frakb$
get identified.
Passing to
$\Hmo_\bfR$,
two Morita equivalent $\bfR$-linear dg categories
$\fraka, \frakb$
get identified.
Recall that
two $\bfR$-linear dg categories
$\fraka, \frakb$
are said to be
\emph{Morita equivalent} 
if they are connected by a Morita morphism.
By
\cite[Proposition 7]{Toe}
or
\cite[Exercise 28]{Toe}
for $\bfR$-linear triangulated dg categories
Morita equivalences
coincide with
quasi-equivalences.
The dg category
$\Perf_{dg}(X)$
is triangulated,
as it is pretriangulated
and
closed under homotopy direct summands.

Let
$\fraka$
be a small $\bfR$-linear dg category.
A
\emph{Morita $\bfS$-deformation}
of
$\fraka$
is an $\bfS$-linear dg category
$\frakb$
together with a Morita equivalence
$\frakb \otimes^L_\bfS \bfR \to \fraka$,
where
$- \otimes^L_\bfS \bfR \colon \Hmo_\bfS \to \Hmo_\bfR$
is the derived functor of the base change
$- \otimes_\bfS \bfR \colon \dgCat_\bfS \to \dgCat_\bfR$.
Two Morita deformations
$\frakb, \frakb^\prime$
are
\emph{isomorphic}
if there is a Morita equivalence
$\frakb \to \frakb^\prime$
inducing the identity on
$\fraka$.
We denote by
$\Def^{mo}_\fraka(\bfS)$
the set of isomorphism classes of Morita $\bfS$-deformations of
$\fraka$.
By
\cite[Proposition 3.3]{KL}
there is a canonical map
\begin{align} \label{eq:Morita}
\Def^{mo}_\fraka(\bfS) \to H^2\bfC(\fraka)^{\oplus l}
\end{align}
defined as follows.
Any Morita $\bfS$-deformation
$\frakb$
of a small $\bfR$-linear dg category
$\fraka$
can be represented by a $h$-flat resolution
$\overline{\frakb} \to \frakb$,
which defines a dg $\bfS$-deformation of 
$\bar{\frakb} \otimes_\bfS \bfR$.
Let
$\phi_{\overline{\frakb}} \in Z^2 \bfC(\bar{\frakb} \otimes_\bfS \bfR)^{\oplus l}$
be a Hochschild cocycle representing
$\overline{\frakb}$
via the bijection
\pref{eq:classifying}.
The map
\pref{eq:Morita}
sends
$\frakb$
to the image
$\phi_\frakb$
of
$\phi_{\overline{\frakb}}$
under the isomorphism
$H^2 \bfC(\bar{\frakb} \otimes_\bfS \bfR)^{\oplus l}
\to
H^2 \bfC(\fraka)^{\oplus l}$
induced by the Morita equivalence
$\bar{\frakb} \otimes_\bfS \bfR \to \fraka$.

\begin{thm} \label{thm:mocdg}
The composition
\begin{align} \label{eq:mocdg}
\Def^{mo}_{\Perf_{dg}(X)}(\bfS) \to \Def^{cdg}_{\Perf_{dg}(X)}(\bfS)
\end{align}
of
\pref{eq:Morita}
with the inverse of
\pref{eq:classifying}
is bijective.
\end{thm}
\begin{proof}
To show the injectivity,
by
\cite[Proposition 3.7]{KL}
it suffices to check that
$\Perf_{dg}(X)$
has bounded above cohomology,
i.e.,
the dg module
$\Perf_{dg}(X)(E, F)$
has bounded above cohomology for each
$E, F \in \Perf_{dg}(X)$.
Consider the spectral sequence
\begin{align*}
E^{p,q}_2
=
H^p(X, \underline{\Ext}^q_X(E, F))
\Rightarrow
\Ext^{p+q}_X(E, F).
\end{align*}
Since we have
$\underline{\Ext}^q_X(E, F) \cong \cH^q(E^\vee \otimes_{\scrO_X} F)$,
the cohomology
\begin{align*}
\cH^{p+q}(\Perf_{dg}(X)(E, F))
\cong
\Ext^{p+q}_X(E, F)
\end{align*}
vanishes
whenever
$p, q$
are sufficiently large.

To show the surjectivity,
consider the characteristic dg morphism
\begin{align} \label{eq:kaibar}
\bfC (\Perf_{dg}(X))
\to
\bfC (\Tw(\Perf_{dg}(X)))
\to
\prod_{E \in \Tw(\Perf_{dg}(X))} \Tw(\Perf_{dg}(X))(E, E)
\end{align}
for the dg category
$\Perf_{dg}(X)$
\cite[Theorem 3.19]{Low08},
where
the first arrow is given by the $B_\infty$-section of the canonical projection
$\bfC (\Tw(\Perf_{dg}(X))) \to \bfC (\Perf_{dg}(X))$.
Here,
\begin{align*}
\Tw(\Perf_{dg}(X)) = \Tw_{\text{ilnil}}(\Perf_{dg}(X))_\infty
\end{align*}
is the $\infty$-part of the dg category of locally nilpotent twisted object over
$\Perf_{dg}(X)$.
It is a dg enhancement of the derived category
$\bfD(\Perf_{dg}(X))$
of right dg modules over
$\Perf_{dg}(X)$.
We denote by
$\Tw(\Perf_{dg}(X))^c$
the full dg subcategory of
$\Tw(\Perf_{dg}(X))$
spanned by compact objects.
Its homotopy category
$\bfD(\Perf_{dg}(X))^c$,
the full triangulated subcategory of
$\bfD(\Perf_{dg}(X))$
spanned by compact objects,
get identified with
$\Perf_{dg}(X)$
via the Yoneda embedding
as
$\Perf_{dg}(X)$
is pretriangulated
and
closed under homotopy direct summands.
When restricted to
$\Tw(\Perf_{dg}(X))^c$,
on cohomology
\pref{eq:kaibar}
induces the characteristic morphism
\begin{align*}
\chi_{\Perf_{dg}(X), E}
\colon
H^\bullet \bfC(\Perf_{dg}(X))
\to
\Ext^\bullet_X(E, E).
\end{align*}

Consider the quasi-fully faithful functor
\pref{eq:qff1}.
One can apply
\cite[Proposition 2.3]{KL}
to obtain a commutative diagram
\begin{align} \label{eq:Keller1} 
\begin{gathered}
\xymatrix{
H^\bullet \bfC(\Perf_{dg}(X)) \ar[r]^{\chi_{\Perf_{dg}(X), E}} \ar[d] & \Ext^\bullet_X(E, E) \ar[d] \\
H^\bullet \bfC(D^+_{dg}(\Qch(X)) \ar[r]^{\chi_{D^+_{dg}(\Qch(X), E}} \ar[d] & \Ext^\bullet_X(E, E) \ar[d] \\
H^\bullet \bfC(\Com^+(\fraki)) \ar[r]^{\chi_{\Com^+(\fraki), E}} & \Ext^\bullet_X(E, E),
}
\end{gathered}
\end{align}
whose horizontal arrows are the characteristic morphisms
and
whose vertical arrows are the canonical isomorphisms induced by quasi-equivalences.
Applying
\cite[Proposition 2.3]{KL}
to the quasi-fully faithful functor,
\begin{align} \label{eq:qff2}
\fraki \to \Com^+(\fraki)
\end{align}
we obtain another commutative diagram
\begin{align} \label{eq:Keller2} 
\begin{gathered}
\xymatrix{
H^\bullet \bfC(\Com^+(\fraki)) \ar[r]^{\chi_{\Com^+(\fraki), E}} \ar[d] & \Ext^\bullet_X(E, E) \ar[d] \\
H^\bullet \bfC(\fraki) \ar[r]^{\chi_{\fraki, E}} & \Ext^\bullet_X(E, E)
}
\end{gathered}
\end{align}
whose horizontal arrows are the characteristic morphisms
and
whose vertical arrows are the canonical isomorphisms,
as the morphism
$\bfC(\Com^+(\fraki)) \to \bfC(\fraki)$
induced by the quasi-fully faithful functor
\pref{eq:qff2}
coincides with the canonical projection
and
by
\cite[Proposition 3.20]{Low08}
its inverse in
$\Ho(B_\infty)$
is the $B_\infty$-section
\pref{eq:embr},
which induces an isomorphism on cohomology.
Note that
$\Tw(\fraki)$
is precisely
$\Com(\fraki)$.

By
\pref{lem:Hochschild}
any element of
$H^2 \bfC(\Perf_{dg}(X))^{\oplus l}$
can be represented by the image
$\embr_\delta(\phi)$
of
$\phi \in Z^2 \bfC(\fraki)^{\oplus l}$
under the $B_\infty$-section
\pref{eq:embr}.
We use the same symbol to denote the image under the direct sum of
\pref{eq:Hochschild}.
Since
$\Perf_{dg}(X)$
has bounded above cohomology,
by
\cite[Proposition 3.12]{KL}
the map
\pref{eq:mocdg}
is surjective
if
there exists a full dg subcategory
$\frakm(\phi) \subset \Perf_{dg}(X)$
which is Morita equivalent to
$\Perf_{dg}(X)$
such that
$\chi^{\oplus l}_{\Perf_{dg}(X)} (\embr_\delta(\phi))_E = 0$
for any
$E \in \frakm(\phi)$
and
cocycle
$\embr_\delta(\phi) \in Z^2 \bfC (\Perf_{dg}(X))^{\oplus l}$.
The composition of vertical arrows in
\pref{eq:Keller1}
and
\pref{eq:Keller2}
maps
$\embr_\delta(\phi)$
to
$\phi$
up to coboundary.
By Corollary
\pref{cor:obst}
the image of
$\phi$
under
$\chi^{\oplus l}_{\fraki, E}$
is the obstruction against deforming
$E$
to an object of
$D^+(\Qch(X_\phi))$.
Here,
for a cocycle
$\phi \in HH^2_{ab}(\Qch(X))^{\oplus l}$
we denote by
$(X_\phi, i_\bfS)$
the $\bfS$-deformation of
$(X, i_\bfR)$
along
$\phi$
via the bijection
\pref{eq:HTab}.
Let
$\overline{\Perf_{dg}(X_\phi)}$
be an $h$-flat resolution of
$\Perf_{dg}(X_\phi)$.
Then the dg category
$\frakm(\phi) = \overline{\Perf_{dg}(X_\phi)} \otimes_\bfS \bfR$
is a full dg subcategory of
$\Perf_{dg}(X)$
with a Morita equivalence
\begin{align*}
\frakm(\phi)
\simeq_{mo}
\Perf_{dg}(X_\phi) \otimes^L_\bfS \bfR
\simeq_{mo}
\Perf_{dg}(X)
\end{align*}
from
\cite[Theorem 1.2]{BFN}.
As each object
$E \in \frakm(\phi)$
lifts to an object of
$\Perf(X_\phi)$,
we obtain
$\chi^{\oplus l}_{\fraki}(\phi)_E = 0$.
Chasing the diagrams
\pref{eq:Keller1}
and
\pref{eq:Keller2},
we obtain
$\chi^{\oplus l}_{\Perf_{dg}(X)} (\embr_\delta(\phi))_E = 0$
which completes the proof.
\end{proof}

\begin{rmk} \label{rmk:BFN1}
The canonical equivalence
\begin{align*}
\Perf_{\infty}(X_\phi) \otimes_{\Perf_{\infty}(\bfS)} \Perf_{\infty}(\bfR)
\simeq_{\infty}
\Perf_{\infty}(X)
\end{align*}
of corresponding $\infty$-categories from
\cite[Theorem 1.2]{BFN}
translates via
\cite[Corollary 5.7]{Coh}
into a Morita equivalence
\begin{align} \label{eq:dgbasechange}
\Perf_{dg}(X_\phi) \otimes^L_{\Perf_{dg}(\bfS)} \Perf_{dg}(\bfR)
\simeq_{mo}
\Perf_{dg}(X)
\end{align}
of dg categories,
where
$- \otimes^L - \colon \Hmo_\bfS \times \Hmo_\bfS \to \Hmo_\bfS$
is the derived pointwise tensor product of dg categories.
The left hand side of
\pref{eq:dgbasechange}
is a triangulated dg category split-generated by objects of the form
$E_\phi \boxtimes M$
for
$E_\phi \in \Perf_{dg}(X_\phi)$
and
$M \in \Perf_{dg}(\bfR)$,
which maps to
$E \otimes_{\scrO_X} \pi^*_\bfR M \in \Perf_{dg}(X)$
via the Morita equivalence.
Here,
$E = E_\phi \otimes_\bfS \bfR$
and
$\pi_\bfR \colon X \to \Spec \bfR$
is the structure morphism.
Hence we obtain a Morita equivalence 
\begin{align*}
\Perf_{dg}(X_\phi) \otimes^L_\bfS \bfR
\simeq_{mo}
\Perf_{dg}(X_\phi) \otimes^L_{\Perf_{dg}(\bfS)} \Perf_{dg}(\bfR)
\simeq_{mo}
\Perf_{dg}(X)
\end{align*}
used in the above proof.
\end{rmk}

\subsection{Maximal partial dg deformations of the dg category of perfect complexes}
As explained above,
the category
$\Perf_{dg}(X_\phi)$
defines a Morita $\bfS$-deformation of
$\Perf_{dg}(X)$.
Let
$\overline{\Perf_{dg} (X_\phi)}$
be the $h$-flat resolution of
$\Perf_{dg} (X_\phi)$
from
\cite[Proposition 3.10]{CNS}.
Choose one lift for each object of
$\frakm(\phi)
=
\overline{\Perf_{dg} (X_\phi)} \otimes_\bfS \bfR$
to 
$\overline{\Perf_{dg}(X_\phi)}$.
Let
$\Gamma$
be the collection of the choices
and
$\overline{\Perf_{dg, \Gamma}(X_\phi)}$ 
the full dg subcategory spanned by objects belonging to 
$\Gamma$.
It gives a dg $\bfS$-deformation of
$\frakm(\phi)$
representing the Morita $\bfS$-deformation
$\Perf_{dg}(X_\phi)$.
On the other hand,
$\frakm(\phi)$
admits a dg deformation
\begin{align*}
\frakm(\phi)_{\embr_\delta(\phi)}
=
\left( \frakm(\phi)[\epsilon] = \frakm(\phi) \otimes_\bfR \bfS, \embr_\delta(m) + \embr_\delta(\phi) \epsilon \right),
\end{align*}
where
$\embr_\delta (m) \in  Z^2 \bfC(\overline{\Perf_{dg}(X)})$
and
$\embr_\delta (\phi) \in Z^2 \bfC(\overline{\Perf_{dg}(X)})^{\oplus l}$
are respectively the images
under the $B_\infty$-section
\pref{eq:embr}
and
its direct sum 
of
the compositions
$m$
in
$\fraki$
and
the cocycle
$\phi \in Z\bfC^2(\fraki)^{\oplus l}$
corresponding to
$\phi \in HH^2_{ab}(\Qch(X))^{\oplus l}$
via the isomorphism obtained from
\cite[Theorem 6.6]{LV06}.
Here,
as above
we use the same symbol to denote the images under the compositions of
the bijections induced by the quasi-equivalences with
\pref{eq:Hochschild}
and
its direct sum respectively.
Also,
we use the same symbol to denote the images under the morphism
$\bfC(\Perf_{dg}(X)) \to \bfC(\frakm(\phi))$
of $B_\infty$-algebras induced by the Morita equivalence
$\frakm(\phi) \to \Perf_{dg}(X)$.

\begin{thm} \label{thm:Intertwine2}
There is an isomorphism
\begin{align*}
\overline{\Perf_{dg, \Gamma} (X_\phi)}
\simeq
\frakm(\phi)_{\embr_\delta(\phi)}
\end{align*}
of dg $\bfS$-deformations of
$\frakm(\phi)$.
In particular,
the Morita $\bfS$-deformations
$\Perf_{dg} (X_\phi)$
defines a maximal partial dg $\bfS$-deformaiton of
$\Perf_{dg} (X)$
along
$\embr_\delta(\phi)$.
\end{thm}
\begin{proof}
Since both dg deformations share their underlying quiver
$\frakm(\phi)[\epsilon]$,
it suffices to show the coincidence of their dg structures up to coboundary.
The dg structure on
$\frakm(\phi)_{\embr_\delta(\phi)}$
is
$\embr_\delta(m) + \embr_\delta (\phi) \epsilon$.
Let
$\overline{D^+_{dg}(\Qch(X_\phi))}$
be the $h$-flat resolution of
$D^+_{dg}(\Qch(X_\phi))$
from
\cite[Proposition 3.10]{CNS}.
There is a canonical dg functor
\begin{align*}
\overline{\Perf_{dg}(X_\phi)}
\hookrightarrow
\overline{D_{dg}(\Qch(X_\phi))}
\end{align*}
extending the canonical embedding
$\Perf_{dg}(X_\phi) \hookrightarrow D_{dg}(\Qch(X_\phi))$.
The dg structure on
$\overline{\Perf_{dg, \Gamma} (X_\phi)}$
is the restriction of that on
$\overline{D^+_{dg}(\Qch(X_\phi))}$
by
\cite[Proposition 2.6]{Low08}.

We compute the dg structure on
$\overline{\Perf_{dg}(X_\phi)}$.
Consider the quasi-fully faithful functor
\begin{align} \label{eq:qff}
\Perf_{dg} (X_\phi)
\hookrightarrow
D^+_{dg}(\Qch(X_\phi))
\to
\Com^+(\fraki_\phi)
\end{align}
where
the first functor is the canonical embedding
and
the second functor is a quasi-equivalence induced by the canonical equivalence
\begin{align} \label{eq:canonical}
D^+(\Qch(X_\phi)) \to K^+(\fraki_\phi)
\end{align}
by
\cite[Theorem A]{CNS}.
The functor
\pref{eq:qff}
canonically extends to that
\begin{align} \label{eq:resolvedqff}
\overline{\Perf_{dg}(X_\phi)}
\hookrightarrow
\overline{D^+_{dg}(\Qch(X_\phi))}
\to
\overline{\Com^+(\fraki_\phi)}
\end{align}
of the $h$-flat resolutions from
\cite[Proposition 3.10]{CNS}.
It induces a morphism
\begin{align} \label{eq:iqff}
\bfC(\overline{\Com^+(\fraki_\phi)})
\to
\bfC(\overline{D^+_{dg}(\Qch(X_\phi))})
\to
\bfC(\overline{\Perf_{dg}(X_\phi)})
\end{align}
of $B_\infty$-algebras,
which in turn induces an isomorphism
\begin{align} \label{eq:iiqff}
H^\bullet \bfC(\overline{\Com^+(\fraki_\phi)})
\to
H^\bullet \bfC(D^+_{dg}(\overline{\Qch(X_\phi))})
\to
H^\bullet \bfC(\overline{\Perf_{dg}(X_\phi)})
\end{align}
of Hochschild cohomology by
\pref{lem:Hochschild}.

Recall that
$\delta \in \bfC^1(\Com^+(\fraki))$
be the differentials of objects in
$\Com^+(\fraki)$.
Let
$\delta + \delta^\prime \epsilon \in \bfC^1(\Com^+(\fraki_\phi))$
be the differentials of objects in
$\Com^+(\fraki_\phi)$
with
$\delta^\prime
=
(\delta^\prime_1, \ldots, \delta^\prime_l)
\in
\bfC^1(\Com(\fraki))^{\oplus l}$.
Then the dg structure on
$\Com^+(\fraki_\phi)$
is
\begin{align*}
\embr_{\delta + \delta^\prime \epsilon}(m + \phi \epsilon)
=
(m + \phi \epsilon) \{ \delta + \delta^\prime \epsilon, \delta + \delta^\prime \epsilon \}
+
(m + \phi \epsilon) \{ \delta + \delta^\prime \epsilon \}
+
(m + \phi \epsilon).
\end{align*}
We use the same symbol to denote the images under the composition of
\pref{eq:iqff}
with the morphism
$\bfC(\Com^+(\fraki_\phi)) \to \bfC(\overline{\Com^+(\fraki_\phi)})$
induced by the $h$-flat resolution.
Note that
$\delta \in \bfC^1(\overline{\Perf_{dg}(X)})$
is the differentials of objects in
$\overline{\Perf_{dg}(X)}$
and
$\delta + \delta^\prime \epsilon \in \bfC^1(\overline{\Perf_{dg}(X_\phi)})$
is the differentials of objects in
$\overline{\Perf_{dg}(X_\phi)}$
with
$\delta^\prime
=
(\delta^\prime_1, \ldots, \delta^\prime_l)
\in
\bfC^1(\overline{\Perf_{dg}(X)})^{\oplus l}$,
as
\pref{eq:iqff}
is a morphism of $B_\infty$-algebras induced by the canonical equivalence
\pref{eq:canonical}.
One can apply the same argument as in the proof of
\cite[Theorem 4.15]{Low08}
to obtain
\begin{align*}
\embr_{\delta + \delta^\prime \epsilon}(m + \phi \epsilon)
=
\embr_\delta(m) + \embr_\delta (\phi) \epsilon
+
d_{\embr_\delta(m)}(\delta^\prime_1) \epsilon_1
+
\cdots
+
d_{\embr_\delta(m)}(\delta^\prime_l) \epsilon_l
\end{align*} 
on
$\overline{\Perf_{dg, \Gamma}(X_\phi)}$.
Note that
up to coboundary
the image of
\begin{align*}
d_{\embr_\delta(m)}(\delta^\prime_1) \epsilon_1
+
\cdots
+
d_{\embr_\delta(m)}(\delta^\prime_l) \epsilon_l
\in
\bfC(\overline{\Com^+(\fraki_\phi)})^{\oplus l}
\end{align*}
under
\pref{eq:iqff}
coincide with the images of 
\begin{align*}
\delta^\prime_1 \epsilon_1
+
\cdots
+
\delta^\prime_l \epsilon_l \in \bfC(\Perf_{dg}(X_\phi))^{\oplus l}
\end{align*}
under Hochschild differential
$d_{\embr_\delta(m)}$
on
$\bfC(\Perf_{dg}(X))$.
Hence we obtain an isomorphism
\begin{align*}
\overline{\Perf_{dg, \Gamma}(X_\phi)}
&=
(\frakm(\phi)[\epsilon], \embr_{\delta + \delta^\prime \epsilon}(m + \phi \epsilon)) \\
&\simeq
(\frakm(\phi)[\epsilon], \embr_\delta(m) + \embr_\delta (\phi) \epsilon) \\
&=
\frakm(\phi)_{\embr_\delta(\phi)}.
\end{align*}
of dg deformaitons of
$\frakm(\phi)$
from
\pref{eq:classifying}.
\end{proof}

\begin{rmk}
From the above theorem it follows that
the image of
$\Perf_{dg}(X_\phi)$
under the map
\pref{eq:Morita}
is represented by
$\embr_\delta(\phi)$.
\end{rmk}

\begin{rmk}
In general,
$\frakm(\phi)$
is strictly smaller than
$\Perf_{dg}(X)$.
For instance,
let
$X_0$
be a quintic $3$-fold of Fermat type.
By
\cite[Proposition 2.1]{AK}
for any general first order deformation
$X_1$
of
$X_0$,
there is no line in
$X_0$
which lifts to a closed subvariety of
$X_1$.
Hence deformations of any perfect complex quasi-isomorphic to the pushforward of the structure sheaf of a line in
$X_0$
is obstructed.
This example was informed to the author by Yukinobu Toda.
\end{rmk}

\section{Deformations of higher dimensional Calabi--Yau manifolds revisited}
Now,
we are ready to prove our first main result.
Consider the functor
\begin{align*}
\Def^{mo}_{\Perf_{dg}(X_0)}
\colon
\Art
\to
\Set
\end{align*}
which sends each
$A \in \Art$
to the set of isomorphism classes of Morita $A$-deformations of
$\Perf_{dg}(X_0)$
and each morphism
$B \to A$
in
$\Art$
to the map
$\Def^{mo}_{\Perf_{dg}(X_0)}(B) \to \Def^{mo}_{\Perf_{dg}(X_0)}(A)$
induced by
$- \otimes^L_\bfS \bfR$.
The deformation theory for
$X_0$
is equivalent to
that for
$\Perf_{dg}(X_0)$
in the following sense.

\begin{thm} \label{thm:natisom}
There is a natural isomorphism
\begin{align} \label{eq:schmo}
\zeta \colon \Def_{X_0} \to \Def^{mo}_{\Perf_{dg}(X_0)}
\end{align}
of deformation functors.
\end{thm}  
\begin{proof}
We show that
the assignment
\begin{align*}
(X_A, i_A) \mapsto (\Perf_{dg}(X_A), i^*_A)
\end{align*}
for each
$A \in \Art$
defines a natural transformation.
Here,
we use the same symbol
$i^*_A$
to denote both the derived pullback functor
\begin{align*}
i^*_A \colon \Perf_{dg}(X_A) \to \Perf_{dg}(X_0)
\end{align*}
and
the induced Morita equivalence
\begin{align*}
\Perf_{dg}(X_A) \otimes^L_A \bfk \simeq_{mo} \Perf_{dg}(X_0).
\end{align*}
The surjection
$A \to \bfk$
factorizes through a sequence
\begin{align*}
A = A_m \to A_{m-1} \to \cdots \to A_1 \to \bfk 
\end{align*}
of small extensions.
Pullback of
$X_A$
yields a sequence
\begin{align*}
(X_{A_m}, i_{A_m})
\mapsto
(X_{A_{m-1}}, i_{A_{m-1}})
\mapsto
\cdots
\mapsto
(X_{A_1}, i_{A_1})
\mapsto
X_0 
\end{align*}
of deformations of
$X_0$.
Let
$\phi_{A_1} \in HH^2(X_0) = \text{H}^1(\scrT_{X_0})$
be the cocycle representing
$(X_{A_1}, i_{A_1})$.

By 
\pref{thm:Intertwine2}
the Morita deformation
$\Perf_{dg}(X_{A_1})$
of
$\Perf_{dg}(X)$
corresponds to
$\embr_{\delta_0}(\phi_{A_1})$
via the bijection
\pref{eq:mocdg}.
Here,
$\embr_{\delta_0}(\phi_{A_1})$
denotes the image under the composition of
\begin{align*}
H^2\bfC(\Com^+(\Inj(\Qch(X_0))))
\cong
H^2\bfC(\Perf_{dg}(X_0))
\end{align*}
with the induced isomorphism
\begin{align*}
HH^2(X_0)
\cong
H^2\bfC(\Inj(\Qch(X_0)))
\cong
H^2\bfC(\Com^+(\Inj(\Qch(X_0))))
\end{align*}
by the $B_\infty$-section of the canonical projection
$\bfC(\Inj(\Qch(X_0))) \to \bfC(\Com^+(\Inj(\Qch(X_0))))$.
Induction yields a sequence
\begin{align*}
[\Perf_{dg}(X_{A_m}), i^*_{A_m}]
\mapsto
[\Perf_{dg}(X_{A_{m-1}}), i^*_{A_{m-1}}]
\mapsto
\cdots
\mapsto
[\Perf_{dg}(X_{A_1}), i^*_{A_1}]
\mapsto
\Perf_{dg}(X_0) 
\end{align*}
of Morita deformations of
$\Perf_{dg}(X_0)$.
In particular,
$(\Perf_{dg}(X_{A_m}), i^*_{A_m})$
is a Morita deformation of
$\Perf_{dg}(X_0)$
corresponding to the collection
$\{ \embr_{\delta_{n-1}}(\phi_{A_n}) \}^m_{n = 1}$.
Here,
each
$\embr_{\delta_{n-1}}(\phi_{A_n})$
denotes the image of
$\phi_{A_n} \in H^1(\scrT_{X_{A_{n-1}}} / A_{n-1})^{\oplus l_{n-1}}$,
where
$l_{n-1}$
is the rank of the kernel of square zero extension
$A_n \to A_{n-1}$
as a free $A_{n-1}$-module,
under the composition of
\begin{align*}
H^2\bfC(\Com^+(\Inj(\Qch(X_{A_{n-1}}))))^{\oplus l_{n-1}}
\cong
H^2\bfC(\Perf_{dg}(X_{A_{n-1}}))^{\oplus l_{n-1}}
\end{align*}
with the induced isomorphism
\begin{align*}
HH^2(X_{A_{n-1}} / A_{n-1})^{\oplus l_{n-1}}
\cong
H^2\bfC(\Inj(\Qch(X_{A_{n-1}})))^{\oplus l_{n-1}}
\cong
H^2\bfC(\Com^+(\Inj(\Qch(X_{A_{n-1}}))))^{\oplus l_{n-1}}
\end{align*}
by the $B_\infty$-section of the canonical projection
$\bfC(\Inj(\Qch(X_{A_{n-1}}))) \to \bfC(\Com^+(\Inj(\Qch(X_{A_{n-1}}))))$.

It follows that
the assignment defines a map
\begin{align*}
\zeta_A \colon \Def_{X_0}(A) \to \Def^{mo}_{\Perf_{dg}(X_0)}(A), \
(X_A, i_A) \mapsto [\Perf_{dg}(X_A), i^*_A].
\end{align*}
For each morphism
$f \colon B \to A$
in
$\Art$
the diagram
\begin{align} \label{eq:commutativity} 
\begin{gathered}
\xymatrix{
\Def_{X_0}(B) \ar[r] \ar[d]_{\Def_{X_0}(f)} & \Def^{mo}_{\Perf_{dg}(X_0)}(B) \ar[d]^{\Def^{mo}_{\Perf_{dg}(X_0)}(f)} \\
\Def_{X_0}(A) \ar[r] & \Def^{mo}_{\Perf_{dg}(X_0)}(A)
}
\end{gathered}
\end{align}
commutes.
To see this,
we may assume that
$f$
is a square zero extension
$B = \bfS$
of
$A = \bfR$.
Then
for any
$(X_\bfS, i_\bfS) \in \Def_{X_0}(\bfS)$
with
$X_\bfS \times_\bfS \bfR \cong X_\bfR$
we have
$X_\bfS \cong (X_\bfR)_\phi$
for some cocycle
$\phi \in H^1 (\scrT_{X_\bfR} / \bfR)^{\oplus l}$.
We already know that
$(\Perf_{dg}(X_\bfS), i^*_\bfS)$
is the Morita deformation of
$\Perf_{dg}(X_\bfR)$.
Since
$(X_\bfS, i_\bfS)$
maps to
$(X_\bfR, i_\bfR)$,
the derived pullback functor
$i^*_\bfS$
factorizes through
$i^*_\bfR$.
Thus the assignments
$\{ \zeta_A \}_{A \in \Art}$
defines a natural transformation
$\zeta \colon \Def_{X_0} \to \Def^{mo}_{\Perf_{dg}(X_0)}$.

It remains to show that
$\zeta_A$
is bijective for each
$A \in \Art$.
We will proceed by induction.
Now,
assume that
$\zeta_{A_i}$
are bijective for all 
$1 \leq i \leq n$.
In order to show the surjectivity of
$\zeta_{A_{n+1}}$,
take any element
$[\fraka_{A_{n+1}}, u^*_{A_{n+1}}] \in \Def^{mo}_{\Perf_{dg}(X_0)}(A_{n+1})$. 
By the assumption of induction,
the reduction
$[\fraka_{A_n}, u^*_{A_n}] \in \Def^{mo}_{\Perf_{dg}(X_0)}(A_n)$
is equal to
$[\Perf_{dg}(Y_{A_n}), j^*_{A_n}]$
for some
$(Y_{A_n}, j_{A_n}) \in \Def_{X_0}(A_n)$.
Combining
\pref{thm:mocdg}
with
\pref{thm:Intertwine2},
one sees that
the Morita $A_{n+1}$-deformation
$(\fraka_{A_{n+1}}, u^*_{A_{n+1}})$
of
$\Perf_{dg}(Y_{A_n})$
is represented by
$\embr_{\delta_n}(\phi_{A_{n+1}})$
for some cocycle
$\phi_{A_{n+1}} \in H^1(\scrT_{Y_{A_n} / A_n})^{\oplus l_n}$.
Then we have
\begin{align*}
[\fraka_{A_{n+1}}, u^*_{A_{n+1}}]
=
[\Perf_{dg}(Y_{A_n, \phi_{A_{n+1}}}), j^*_{A_{n+1}}].
\end{align*}

In order to show the injectivity,
suppose that we have
\begin{align} \label{eq:equality}
[\Perf_{dg}(X_{A_{n+1}}), i^*_{A_{n+1}}]
=
[\Perf_{dg}(Y_{A_{n+1}}), j^*_{A_{n+1}}],
\end{align}
i.e.,
there is a Morita equivalence
$\Perf_{dg}(X_{A_{n+1}}) \simeq_{mo} \Perf_{dg}(Y_{A_{n+1}})$
reducing to the identity on
$\Perf_{dg}(X_0)$.
Combining the above argument with the commutative diagram
\pref{eq:commutativity},
we have
\begin{align*}
[\Perf_{dg}(X_{A_{n+1}}), i^*_{A_{n+1}}]
=
[\Perf_{dg}(X_{{A_n}, \phi_{A_n}}), i^*_{A_{n+1}}], \
[\Perf_{dg}(Y_{A_{n+1}}), j^*_{A_{n+1}}]
=
[\Perf_{dg}(Y_{{A_n}, \psi_{A_n}}), j^*_{A_{n+1}}]
\end{align*}
for some elements
\begin{align*}
(X_{A_n}, i_{A_n}), (Y_{A_n}, j_{A_n}) \in \Def_{X_0}(A_n)
\end{align*}
and
cocycles
\begin{align*}
\phi_{A_n} \in H^1(\scrT_{X_{A_n} / A_n})^{\oplus l_n}, \
\psi_{A_n} \in H^1(\scrT_{Y_{A_n} / A_n})^{\oplus l_n}.
\end{align*}
Applying
$- \otimes^L_{A_{n+1}} A_n$,
we obtain a Morita equivalence
$\Perf_{dg}(X_{A_n}) \simeq_{mo} \Perf_{dg}(Y_{A_n})$
reducing to the identity on
$\Perf_{dg}(X_0)$.
By the assumption of induction
there is an isomorphim
$X_{A_n} \cong Y_{A_n}$
reducing to the identity on
$X_0$.
Then 
\pref{eq:equality}
implies
$[\embr_{\delta_{n-1}}(\phi_{A_n})] = [\embr_{\delta_{n-1}}(\psi_{A_n})]$,
which in turn implies
$[\phi_{A_n}] = [\psi_{A_n}]$.
Thus we obtain an isomorphism
$X_{A_{n+1}} \cong Y_{A_{n+1}}$
reducing to the identity on
$X_0$. 
\end{proof}

\begin{rmk} \label{rmk:BFN2}
Consider the functor
\begin{align*}
\widetilde{\Def}^{mo}_{\Perf_{dg}(X_0)}
\colon
\Art
\to
\Set
\end{align*}
which sends each
$A \in \Art$
to the set of isomorphism classes of Morita $A$-deformations of
$\Perf_{dg}(X_0)$
and each morphism
$B \to A$
in
$\Art$
to the map
$\Def^{mo}_{\Perf_{dg}(X_0)}(B) \to \Def^{mo}_{\Perf_{dg}(X_0)}(A)$
induced by the derived pointwise tensor product with
$\Perf_{dg}(A)$
over
$\Perf_{dg}(B)$.
Based on Remark
\pref{rmk:BFN1},
one can rewrite the proof of
\pref{thm:natisom}
in terms of
$\widetilde{\Def}^{mo}_{\Perf_{dg}(X_0)}$
to obtain a natural isomorphism
\begin{align*}
\tilde{\zeta} \colon \Def_{X_0} \to \widetilde{\Def}^{mo}_{\Perf_{dg}(X_0)}
\end{align*}
of deformation functors.
In the sequel,
we will identify the deformation functors
\begin{align*}
\Def^{mo}_{\Perf_{dg}(X_0)},
\widetilde{\Def}^{mo}_{\Perf_{dg}(X_0)}
\colon
\Art
\to
\Set
\end{align*}
without further comments.
\end{rmk}

\begin{rmk}
\pref{thm:natisom}
tells us that
infinitesimal deformations of
$\Perf_{dg}(X_0)$
is controlled by the Kodaira--Spencer differential graded Lie algebra
$\KS_{X_0}$
of
$X_0$.
Consider the functor
$\Def_{\KS_{X_0}} \colon \Art \to \Set$
defined as
\begin{align}
\Def_{\KS_{X_0}} (A) = \frac{\MC_{\KS_{X_0}} (A)}{gauge \ equivalence}
\end{align}
for each
$A \in \Art$,
where
\begin{align}
\MC_{\KS_{X_0}} (A)
=
\Bigl\{
x \in \KS^1_{X_0} \otimes_{\bfk} \frakm_A \ | \ dx + \frac{1}{2} [x, x] = 0 
\Bigr\}.
\end{align}
Recall that
given a differential graded Lie algebra
$L$
and
a commutative $\bfk$-algebra
$\frakm$
there exists a natural structure of differential graded Lie algebra on the tensor product
$L \otimes_\bfk \frakm$
given by
\begin{align*}
d( x \otimes_\bfk r) = dx \otimes_\bfk r, \
[x \otimes_\bfk r, y \otimes_\bfk s] = [x,y] \otimes_ \bfk rs, \
x,y \in L, \
r,s \in \frakm.
\end{align*}
For every surjection
$A \to \bfk [t] / t^2$
in
$\Art$
the set
$\Def_{\KS_{X_0}} (A)$
consists of solutions of the extended Maurer--Cartan equation to
$\frakm_A$.
Giving higher order deformations of
$X_0$
is equivalent to giving solutions of the extended Maurer--Cartan equation. 
Indeed,
we have
$\Def_{\KS_{X_0}} \simeq \Def_{X_0}$
by
\cite[Example 2.3]{Man}.
\end{rmk}

\begin{cor} \label{cor:prorep}
The functor
$\Def^{mo}_{\Perf_{dg}(X_0)}$
is prorepresented by
$R$.
\end{cor}  
\begin{proof}
This follows immediately as
$\Def_{X_0}$
is prorepresented by
$R$.
\end{proof}

\begin{cor} \label{cor:effective}
The functor
$\Def^{mo}_{\Perf_{dg}(X_0)}$
has an effective universal formal family.
\end{cor}  
\begin{proof}
Let
$(R, \tilde{\xi})$
be a universal formal family for
$\Def^{mo}_{\Perf_{dg}(X_0)}$,
where
$\tilde{\xi} = \{ \tilde{\xi}_n \}_n$
belongs to the limit 
\begin{align*}
\widehat{\Def}^{mo}_{\Perf_{dg}(X_0)} (R)
=
\displaystyle \lim_{\longleftarrow} \Def^{mo}_{\Perf_{dg}(X_0)} (R / \frakm^n_R)
\end{align*}
of the inverse system
\begin{align*}
\cdots
\to
\Def^{mo}_{\Perf_{dg}(X_0)} (R/\frakm_R^{n+2})
\to
\Def^{mo}_{\Perf_{dg}(X_0)} (R/\frakm_R^{n+1})
\to
\Def^{mo}_{\Perf_{dg}(X_0)} (R/\frakm_R^n)
\to
\cdots
\end{align*}
induced by the natural quotient maps
$R/\frakm^{n+1}_R \to R/\frakm^n_R$.
Recall that
for the universal formal family
$(R, \xi)$
there is a noetherian formal scheme
$\scrX$
over
$R$
such that
$X_n \cong \scrX \times_R R / \frakm^{n+1}_R$
for each
$n$,
where 
$(X_n, i_n)$
are $R_n$-deformations of
$X_0$
defining
$\xi_n$.
By
\cite[Theorem III5.4.5]{GD61}
there exists a scheme
$X_R$
flat projective over
$R$
whose formal completion along the closed fiber
$X_0$
is isomorphic to
$\scrX$.
From the proof of
\pref{thm:natisom}
it follows that
$(\Perf_{dg}(X_n), i^*_n)$
defines
$\tilde{\xi_n}$.
Then by
\cite[Theorem 1.2]{BFN}
the $R$-linear dg category
$\Perf_{dg}(X_R)$
yields the compatible system
$\{ \tilde{\xi}_n\}_n$
via reduction along the natural quotient maps
$R/\frakm^{n+1}_R \to R/\frakm^n_R$,
which means
$\tilde{\xi}$
is effective.
\end{proof}

\begin{rmk}
Recall that
the Dwyer--Kan model structure on
$\dgCat_\bfk$
has a natural simplicial enrichment
\cite[Section 5]{Toe}.
We denote by
$\dgCat^\infty_\bfk$
the underlying $\infty$-category.
There is a notion of limits in $\infty$-categories
that behaves similarly to the classical one
\cite[Chapter 4]{Lur}.
As
$\dgCat^\infty_\bfk$
is the underlying $\infty$-category of a simplicial model category,
it admits limits
\cite[Corollary 4.2.4.8]{Lur}.
Hence we obtain a limit
\begin{align*}
\widehat{\Perf}_{dg}(X_R)
=
\displaystyle \lim_{\longleftarrow} \Perf_{dg}(X_n)
\end{align*}
of the inverse system
\begin{align*}
\cdots
\to
\Perf_{dg}(X_{n+2})
\to
\Perf_{dg}(X_{n+1})
\to
\Perf_{dg}(X_n)
\to
\cdots
\end{align*}
of small $\bfk$-linear dg categories induced by the natural quotient maps
$R/\frakm^{n+1}_R \to R/\frakm^n_R$.

We claim that
the limit is quasi-equivalent to
$\Perf_{dg}(\scrX)$.
By
\cite[Corollary 5.1.3]{GD61}
the canonical map
\begin{align*}
\Hom_{X_R}(\scrE, \scrF)
\to
\Hom_{\scrX}(\hat{\scrE}, \hat{\scrF})
\end{align*}
defined by taking the formal completion of each morphism along the closed fiber 
is an isomorphism for all coherent sheaves
$\scrE, \scrF$
on
$X_R$.
In particular,
we may write
\begin{align*}
\widehat{\Hom_{X_R}(\scrE, \scrF)}
=
\Hom_{\scrX}(\hat{\scrE}, \hat{\scrF}).
\end{align*}
Since
$X_R$
is projective over a complete local noetherian ring
$R$,
by
\cite[Corollary III5.1.6]{GD61}
the functor
\begin{align*}
\coh (X_R) \to \coh (\scrX),
\end{align*}
which sends each coherent sheaf
$\scrF$
on
$X_R$
to its formal completion
$\hat{\scrF}$
along the closed fiber
is an equivalence of abelian categories.
We obtain the induced derived equivalence
\begin{align*}
\Perf(X_R) \simeq D^b (X_R) \simeq D^b (\scrX) \simeq \Perf(\scrX).
\end{align*}
Hence for
$E, F \in \Perf_{dg}(X_R)$
with formal completions
$\hat{E}, \hat{F} \in \Perf_{dg}(\scrX)$
we may write
\begin{align*}
\widehat{\Ext^i_{X_R}(E, F)}
=
\Ext^i_{\scrX}(\hat{E}, \hat{F}).
\end{align*}
Now,
one sees that
the objects and morphisms in
$\Perf(\scrX)$
satisfy universality with respect to the induced inverse system on homotopy categories. 
Thus the dg functor
\begin{align*}
\Perf_{dg}(\scrX) \to \widehat{\Perf}_{dg}(X_R)
\end{align*}
uniquely determined by universality of the limit is a quasi-equivalence.
Namely,
the formal completion of
$\Perf_{dg}(X_R)$
is quasi-equivalent to
$\Perf_{dg}(\scrX)$.
\end{rmk}

\begin{cor} \label{cor:algebraizable}
Any effective universal formal family for
$\Def^{mo}_{\Perf_{dg}(X_0)}$
is algebraizable.
In particular,
an algebraization is given by
$\Perf_{dg}(X_S)$
where
$X_S$
is a versal deformation of
$X_0$.
\end{cor}  
\begin{proof}
Consider the triple
$(\Spec S, s, \Perf_{dg}(X_S))$.
Since the reduction of
$X_S$
along the natural quotient maps
$S/\frakm^{n+1}_S \to S/\frakm^n_S$
yields a compatible system isomorphic to 
$\xi$,
the reduction of
$\Perf_{dg}(X_S)$
yields a compatible system isomorphic to 
$\tilde{\xi}$.
Thus
$(\Spec S, s, \Perf_{dg}(X_S))$
gives a versal Morita deformation of
$\Perf_{dg}(X_0)$.
\end{proof}

\begin{prop} \label{prop:dgCGF}
There is a quasi-equivalence
\begin{align*}
\Perf_{dg}(X_S) / \Perf_{dg}(X_S)_0 \simeq_{qeq} \Perf_{dg}(X_{Q(S)})
\end{align*}
where
$Q(S)$
is the quotient field of
$S$
and
$X_{Q(S)}$
is the generic fiber of
$X_S$.
\end{prop}  
\begin{proof}
By
\cite[Theorem 3.4]{Dri}
and
\cite[Theorem 1.1]{Morb}
we have an equivalence
\begin{align*}
[\Perf_{dg}(X_S) / \Perf_{dg}(X_S)_0]
\simeq
\Perf(X_S) / \Perf(X_S)_0
\simeq
\Perf(X_{Q(S)})
\end{align*}
of idempotent complete triangulated categories,
where
the middle category is the Verdier quotient by the full triangulated subcategory
$\Perf(X_S)_0 \subset \Perf(X_S)$
of perfect complexes with $S$-torsion cohomology.
Then the claim follows from
\cite[Theorem B]{CNS}.
\end{proof}

\begin{rmk}
From the proof
one sees that 
the dg categorical generic fiber is a natural dg enhancement of the categorical generic fiber introduced in
\cite{Morb},
which is in turn based on the categorical general fiber by Huybrechts--Macr\`i--Stellari
\cite{HMS11}.
\end{rmk}

\section{Independence from geometric realizations}\label{sec:Independence}

Due to Corollary
\pref{cor:algebraizable},
a versal Morita deformation of
$\Perf_{dg}(X_0)$
is given by
$\Perf_{dg}(X_S)$
where
$X_S$
is a versal deformatiion of
$X_0$.
Suppose that
there is another Calabi--Yau manifold
$X^\prime_0$
derived-equivalent to
$X_0$.
Since by
\cite[Theorem B]{CNS}
dg enhancements of
\begin{align*}
\Perf(X_0)
\simeq
D^b(X_0)
\simeq
D^b(X^\prime_0)
\simeq
\Perf(X^\prime_0)
\end{align*}
are unique,
we obtain a quasi-equivalence
\begin{align*}
\Perf_{dg}(X_0) \to \Perf_{dg}(X^\prime_0).
\end{align*}
Hence
$\Perf_{dg}(X_S)$
gives also a versal Morita deformation of
$\Perf_{dg}(X^\prime_0)$.

By
\pref{lem:smprojVD}
we may assume
$X_S$
to be smooth projective over
$S$.
Then one finds a smooth projective versal deformation
$X^\prime_S$
over the same base.
The construction requires the deformation theory of Fourier--Mukai kernels,
which we briefly review below.
It passes through effectivizations,
i.e.,
there are effectivizations
$X_R, X^\prime_R$
of
$X_0, X^\prime_0$
over the same regular affine scheme
$\Spec R$.
Applying
\cite[Corollary 4.2]{Mora}
and
\cite[Theorem B]{CNS},
we obtain a quasi-equivalence
\begin{align} \label{eq:qeqeff}
\Perf_{dg}(X_R) \simeq_{qeq} \Perf_{dg}(X^\prime_R).
\end{align}

Unwinding the construction of versal deformations recalled in
Section
\pref{subsec:DefCY},
one sees that,
up to equivalence of deformations,
the ambiguity of
$X_S$ 
essentially stems from the choice of indices
$i \in I$
of the filtered inductive system
$\{ R_i \}_{i \in I}$,
where
$R_i$
are finitely generated $T$-subalgebras of
$R$
whose colimit is
$R$.
The versal deformations
$X_S, X^\prime_S$
over the same base are obtained by choosing the same sufficiently large index. 
From this observation combined with
\pref{thm:natisom}
and the quasi-equivalence
\pref{eq:qeqeff},
it is natural to expect that
the versal Morita deformations
$\Perf_{dg}(X_S), \Perf_{dg}(X^\prime_S)$
become quasi-equivalent close to effectivizations.
In this section,
we prove our second main result
which yields the quasi-equivalence as a corollary.

\subsection{Deformations of Fourier--Mukai kernels}
Suppose that
the derived equivalence of
$X_0, X^\prime_0$
is given by a Fourier--Mukai kernel
$\cP_0 \in D^b (X_0 \times X^\prime_0)$.
In order to define a relative integral functor from
$D^b (X_S)$
to
$D^b (X^\prime_S)$,
we deform
$\cP_0$
to a perfect complex
$\cP_S$
on
$X_S \times_S X^\prime_S$.
Here,
for a deformation
$[X_\bfP, i_\bfP] \in \Def_X((\bfP, \frakm_\bfP))$
of a $\bfk$-scheme
$X$,
by a deformation of
$E \in \Perf(X)$
over
$(\bfP, \frakm_\bfP)$
we mean a pair
$(E_\bfP, u_\bfP)$,
where
$E_\bfP \in \Perf(X_\bfP)$
and
$u_\bfP \colon E_\bfP \otimes^L_\bfP \bfk \to E$
is an isomorphism.
Two deformations
$(E_\bfP, u_\bfP)$
and
$(F_\bfP, v_\bfP)$
are equivalent
if there is an isomorphism
$E_\bfP \to F_\bfP$
reducing to an isomorphism of
$E$.

The $R_n$-deformations
$X_n, X^\prime_n$
of
$X_0, X^\prime_0$
and
their fiber product
$X_n \times_{R_n} X^\prime_n$
form the diagram
\begin{align*}
\begin{gathered}
\xymatrix{
&
X_n \times_{R_n} X^\prime_n
\ar _{ q_n }[dl]
\ar ^{ p_n }[dr]
&\\
X_n
&
&
X^\prime_n}
\end{gathered}
\end{align*}
with the natural projections
$q_n$
and
$p_n$.
For any perfect complex
$\cP_n$
on
$X_n \times_{R_n} X^\prime_n$,
the relative integral functor
\begin{align*}
\Phi_{\cP_n} \left( - \right)  = Rp_{n *} \left( \cP_n \otimes^L q^*_n \left( - \right) \right)
\end{align*}
sends each object of
$D^b(X_n)$
to
$D^b(X^\prime_n)$.
Due to the Grothendieck--Verdier duality
the functor
$\Phi_{\cP_n}$
admits the right adjoint
$\Phi^R_{\cP_n} = \Phi_{(\cP_n)_R}$
with kernel
$(\cP_n)_R
=
\cP^\vee_n \otimes p^*_n \omega_{\pi^\prime_n}[\dim X_0]$,
where
$\omega_{\pi^\prime_n}$
is the determinant of the relative cotangent sheaf associated with the natural projection
$\pi^\prime_n \colon X^\prime_n \to \Spec R_n$.

\begin{lem} {\rm{(}\cite[Lemma 3.1, 3.2]{Mora}\rm{)}} \label{lem:OBSf}
Assume that
$\Phi_{\cP_n}$
is an equivalence.
Then for any thickening
$X_n \hookrightarrow X_{n+1}$
there exist a thickening
$X^\prime_n \hookrightarrow X^\prime_{n+1}$
and
a perfect complex
$\cP_{n+1}$
on
$X_{n+1} \times_{R_{n+1}} X^\prime_{n+1}$
with an isomorphism
$\cP_{n+1} \otimes^L_{R_{n+1}} R_n \cong \cP_n$
such that
the integral functor
$\Phi_{\cP_{n+1}} \colon D^b(X_{n+1}) \to D^b(X^\prime_{n+1})$
is an equivalence.
\end{lem}

Iterative application of
\pref{lem:OBSf}
allows us to deform the Fourier--Mukai kernel
$\cP_0 \in D^b(X_0 \times X^\prime_0)$
to some Fourier--Mukai kernel
$\cP_n \in \Perf(X_n \times_{R_n} X^\prime_n)$
for arbitrary order
$n$.   
We obtain a system of deformations
$\cP_n \in \Perf (X_n \times_{R_n} X^\prime_n)$
of
$\cP_0$
with compatible isomorphisms
$\cP_{n + 1} \otimes^L_{R_{n + 1}} R_n \to \cP_n$. 
According to
\cite[Proposition 3.6.1]{Lie}
there exists an effectivization,
i.e.,
a perfect complex $\cP_R$ on $X_R \times_R X^\prime_R$
with compatible isomorphisms
$\cP_R \otimes^L_R R_n \to \cP_n$. 
Recall that
to algebrize
$\scrX$
we used a filtered inductive system
$\{ R_i \}_{i \in I}$
of finitely generated $T$-subalgebras of
$R$
whose colimit is
$R$.
Taking an index
$i$
sufficiently large,
one finds smooth projective $R_i$-deformations
$X_{R_i}, X^\prime_{R_i}$
of
$X_0, X^\prime_0$
whose pullback along the canonical homomorphism
$R_i \hookrightarrow R$
are
$X_R, X^\prime_R$.
Since we have
$X_R \times_R X^\prime_R
\cong
\left( X_{R_i} \times_{R_i} X^\prime_{R_i} \right) \times_{R_i} R$,
by
\cite[Proposition 2.2.1]{Lie}
there exists a perfect complex
$\cP_{R_i}$
on
$X_{R_i} \times_{R_i} X^\prime_{R_i}$
with an isomorphism
$\cP_{R_i} \otimes^L_{R_i} R \to \cP_R$.
Finally,
the derived pullback
$\cP_S \in \Perf (X_S \times_S X^\prime_S)$
along
$R_i \to S$
yields a deformation of
$\cP_0$.

\begin{lem} {\rm{(}\cite[Proposition 3.3]{Mora}\rm{)}}
Let
$\cP_0$
be a Fourier--Mukai kernel 
defining the derived equivalence of Calabi--Yau manifolds
$X_0, X^\prime_0$
of dimension more than two.
Then there exists a perfect complex
$\cP_S$
on the fiber product
$X_S \times_S X^\prime_S$
of smooth projective versal deformations
with an isomorphism
$\cP_S \otimes^L_S \bfk \to \cP_0$.
\end{lem}

\subsection{Inherited equivalences}
The schemes
$X_{R_i}, X^\prime_{R_i}$,
and
their fiber product
$X_{R_i} \times_{R_i} X^\prime_{R_i}$
together with the pullbacks along $T$-algebra homomorphisms
$R_i \to R_j \to R$
for
$i \leq j$
form the commutative diagram  
\begin{align*} 
\begin{gathered}
\xymatrix{
& X_R \times_R X^\prime_R \ar[dl]_{q} \ar[d]^{f^{\prime \prime}_j} \ar[dr]^{p} & \\
X_R \ar[d]^{f_j} & X_{R_j} \times_{R_j} X^\prime_{R_j} \ar[dl]_{q_j} \ar[d]^{f^{\prime \prime}_{ij}} \ar[dr]^{p_j} & X^\prime_R \ar[d]^{f^\prime_j} \\
X_{R_j} \ar[d]^{f_{ij}} & X_{R_i} \times_{R_i} X^\prime_{R_i} \ar[dl]_{q_i} \ar[dr]^{p_i} & X^\prime_{R_j} \ar[d]^{f^\prime_{ij}} \\
X_{R_i} & & X^\prime_{R_i},
}
\end{gathered}
\end{align*}
where
$q_i, p_i$
are smooth projective of relative dimension
$\dim X_0$.
Given a collection
$\{ \cP_i \}_{i \in I}$
with
$\cP_i \in \Perf (X_{R_i} \times_{R_i} X^\prime_{R_i})$
satisfying
$\cP_j \cong \cP_{R_i} \otimes^L_{R_i} R_j$
and
$\cP_R \cong \cP_{R_j} \otimes^L_{R_j} R$
for all $i \leq j$,
consider the relative integral functors
\begin{align*}
\Phi_{\cP_i} 
=
Rp_{i*} \left( \cP_i \otimes^L q^*_i \left( - \right) \right)
\colon
D^b (X_{R_i})
\to
D^b (X^\prime_{R_i}).
\end{align*}
Since
$p_i$
is projective
and
$\cP_i$ is of finite homological dimension,
i.e.,
$\cP_i \otimes^L q^*_i F_{R_i}$
are bounded for each object
$F_{R_i} \in D^b (X_{R_i})$,
one can apply
\cite[Lemma 1.8]{LST}
to see that
$\Phi_{\cP_i}$
send perfect complexes to perfect complexes.
We use the same symbol to denote the restricted functor.

\begin{thm} \label{thm:inherit}
There exists an index
$j \in I$
such that
for all
$k \geq j$
the functors
\begin{align*}
\Phi_{\cP_k}
\colon
\Perf(X_{R_k})
\to
\Perf(X^\prime_{R_k})
\end{align*}
are equivalences of triangulated categories of perfect complexes.
In particular,
the dg categories
$\Perf_{dg}(X_{R_k}), \Perf_{dg}(X^\prime_{R_k})$
of perfect complexes are quasi-equivalent.
\end{thm}
\begin{proof}
Under the assumption
one always finds deformations
$[X_{R_j}, i_{R_j}], [X^\prime_{R_j}, i^\prime_{R_j}]$
smooth projective over
$(R_j, \frakm_{R_j})$
for sufficiently large index
$j \in I$.
Moreover,
the pullbacks along
$R_j \to R$
and
$R_j \to S$
yield respectively
effectivizations
$X_R, X^\prime_R$
of universal formal families
$\xi, \xi^\prime$
and
versal deformations
$(\Spec S, s, X_S), (\Spec S, s, X^\prime_S)$
of
$X_0, X^\prime_0$.
Recall that
$(\Spec S, s)$
is an \'etale neighborhood of
$t$
in
$\Spec T$
with
$t$
corresponding to the maximal ideal
$(t_1, \ldots, t_d) \subset T$,
and
the formal completions of
$X_S, X^\prime_S$
along the closed fibers over
$s$
are isomorphic to
$\hat{X}_R, \hat{X}^\prime_R$. 
In summary,
we have the pullback diagrams
\begin{align*}
\begin{gathered}
\xymatrix{
X_0  \ar@{^{(}->}[r]^-{} \ar_{\pi}[d] & X_S \ar[r]^-{f_{jS}} \ar_{\pi_S}[d] & X_{R_j} \ar_{\pi_{R_j}}[d] & X_R \ar[l]_-{f_j} \ar_{\pi_R}[d] \\
\Spec \bfk \ar@{^{(}->}[r]^-{} & \Spec S \ar[r]_-{} & \Spec R_j & \Spec R \ar[l] \\
X^\prime_0 \ar@{^{(}->}[r]^-{} \ar^{\pi^\prime}[u] & X^\prime_S \ar[r]^-{f^\prime_{jS}} \ar^{\pi^\prime_S}[u] & X^\prime_{R_j} \ar^{\pi^\prime_{R_j}}[u] & X^\prime_R. \ar[l]_-{f^\prime_j} \ar^{\pi^\prime_R}[u]
}
\end{gathered}
\end{align*}

Let
$\cP_0 \in D^b(X_0 \times_\bfk X^\prime_0)$
be a Fourier--Mukai kernel defining the derived equivalence.
As explained above,
one can deform
$\cP_0$
to a perfect complex
$\cP_j \in D^b(X_{R_j} \times_{R_j} X^\prime_{R_j})$.
Due to the Grothendieck--Verdier duality
the functor
$\Phi_{\cP_j}$
admits a left adjoint
$\Phi^L_{\cP_j} = \Phi_{(\cP_i)_L}$
with kernel
$(\cP_j)_L
=
\cP^\vee_{R_j} \otimes p^*_j \omega_{\pi^\prime_{R_j}}[\dim X_0]$.
By
\cite[Corollary 3.1.2]{BV}
the category
$\Perf (X_{R_j})$
is generated by some single object
$E_{R_j}$.
Namely,
each object
$F_{R_j} \in \Perf(X_{R_j})$
can be obtained from
$E_{R_j}$
by taking
isomorphisms,
finite direct sums,
direct summands,
shifts,
and
bounded number of cones.
The counit morphism
$\eta_j \colon \Phi^L_{\cP_j} \circ \Phi_{\cP_j} \to \id_{\Perf (X_{R_j})}$
gives the distinguished triangle
\begin{align*} 
\Phi^L_{\cP_j} \circ \Phi_{\cP_j} (E_{R_j})
\xrightarrow{\eta_j(E_{R_j})}
E_{R_j}
\to
C(E_{R_j}) \coloneqq \Cone (\eta_j ( E_{R_j})).
\end{align*}
For sufficiently large
$k \geq j$
we will show that
$\eta_k(E_{R_k})$
is an isomorphism
and then
$\Phi_{\cP_k}$
is fully faithful.
Similarly,
one can show that
$\Phi^L_{\cP_k}$
is also fully faithful.
Thus
$\Phi_{\cP_k}$
is an equivalence,
as it is a fully faithful functor admitting a fully faithful left adjoint.

Pullback along
$R_j \subset R_k$
yields
\begin{align} \label{eq:DT}
\Phi^L_{\cP_k} \circ \Phi_{\cP_k} (E_{R_k})
\xrightarrow{f^*_{jk} \eta_j(E_{R_j})}
E_{R_k}
\to
f^*_{jk}C(E_{R_j})
\end{align}
with
$E_k = f^*_{jk} E_j$
and
$\cP_k = (f_{jk} \times f^\prime_{jk})^* \cP_j$.
Further pullback along
$R_k \subset R$
yields
\begin{align*} 
\Phi^L_{\cP_R} \circ \Phi_{\cP_R} (E_R)
\xrightarrow{f^*_j \eta_j(E_{R_j})}
E_R
\to
f^*_j C(E_{R_j}),
\end{align*}
where
$f_j \colon X_R \to X_{R_j}$
satisfies
$f_{jk} \circ f_k = f_j$.
Restriction to the closed fiber
$X_0$
yields
\begin{align*} 
\Phi^L_{\cP_0} \circ \Phi_{\cP_0} (E_R |_{X_0})
\xrightarrow{\eta_j(E_{R_j} |_{X_0})}
E_R |_{X_0}
\to
(f^*_j C(E_{R_j})) |_{X_0}.
\end{align*}
Note that
since
$f^{-1}_j(X_0) = X_0$
and
the restriction of the counit morphism is the counit morphism,
we have
$(f^*_j \eta(E_{R_j})) |_{X_0} = \eta(E_{R_j} |_{X_0})$.
Each term in the above distinguished triangle is perfect 
so that
we may consider the restriction to the closed fiber.
Since
$\Phi_{\cP_0}$
is an equivalence,
$\eta_j(E_{R_j} |_{X_0})$
is an isomorphism
and
we obtain a quasi-isomorphism
$f^*_j C(E_{R_j}) |_{X_0} \cong 0$.
Then the support of
$f^*_j C(E_{R_j})$
is a proper closed subscheme of
$X_R$
which does not contain any closed point of
$X_R$.
Thus the quasi-isomorphism extends to
$f^*_j C(E_{R_j}) \cong 0$. 
From
\cite[Proposition 2.2.1]{Lie}
it follows
$f^*_{jk} C(E_{R_j}) \cong 0$
when
$k \in I$
is sufficiently large.

Take any closed point
$u \in \Spec R_j$
whose inverse image by
$g_{jk} \colon \Spec R_k \to \Spec R_j$
is not empty.
We have the pullback diagrams
\begin{align*}
\begin{gathered}
\xymatrix{
X_{R_j} \ar_{\pi_{R_j}}[d]  & X_u \ar@{_{(}->}[l]^-{} \ar[d] & f^{-1}_{jk}(X_u) \ar_{f_{u, jk}}[l] \ar@{^{(}->}[r]^-{} \ar[d] & X_{R_k} \ar_{\pi_{R_k}}[d] \\
\Spec R_j & \Spec \bfk \ar@{_{(}->}[l]^-{} & g^{-1}_{jk}(u) \ar_{g_{u, jk}}[l] \ar@{^{(}->}[r]^-{} & \Spec R_k.
}
\end{gathered}
\end{align*}
Note that
$f_{u, jk}$
is surjective by construction
and
flat
as
$g_{u, jk}$
is flat. 
The restriction of
\pref{eq:DT}
to
$f^{-1}_{jk} (X_u)$
yields
\begin{align*}
\Phi^L_{\cP_{u, jk}}
\circ
\Phi_{\cP_{u, jk}} (E_{R_k} |_{f^{-1}_{jk} (X_u)})
\xrightarrow{f^*_{u, jk} \eta_j (E_{R_j} |_{X_u})}
E_{R_k} |_{f^{-1}_{jk} (X_u)}
\to
f^*_{u, jk} C(E_{R_j} |_{X_u})
\cong
0,
\end{align*}
where
$\cP_{u, jk} = \cP_k |_{f^{-1}_{jk} (X_u) \times_{g^{-1}_{jk}(u)} f^{-1}_{jk} (X^\prime_u)}$.
It follows
$C(E_{R_j}|_{X_u}) \cong 0$
and
$\eta_j(E_{R_j} |_{X_u})$
is an isomorphism.

By
\cite[Lemma 3.4.1]{BV}
the restriction
$E_{R_j} |_{X_u}$
is a generator of
$\Perf(X_u)$.
Then each object
$F_u \in \Perf(X_u)$
can be obtained from
$E_{R_j} |_{X_u}$
by taking
isomorphisms,
finite direct sums,
direct summands,
shifts,
and
bounded number of cones.
We may assume that
$E_{R_j} |_{X_u}$
has no nontrivial direct summands,
as
$\Phi^L_{\cP_{u, j}}$
and
$\Phi_{\cP_{u, j}}$
commute with direct sums on
$\Perf(X^\prime_u)$
and
$\Perf(X_u)$
respectively
with
$\cP_{u, j} = \cP_j |_{X_u \times X^\prime_u}$
\cite[Corollary 3.3.4]{BV}.
One inductively sees that 
the counit morphism
$\Phi^L_{\cP_{u, j}} \circ \Phi_{\cP_{u, j}}(F_u) \to F_u$
is an isomorphism. 
In other words,
the restriction
$\Phi_{\cP_{u, j}}$
of
$\Phi_{\cP_j}$
to
$X_u$
is fully faithful.
Similarly,
the restriction
$\Phi^L_{\cP_{u, j}}$
of
$\Phi^L_{\cP_j}$
to
$X^\prime_u$
is also fully faithful.
Thus
$\Phi_{\cP_{u, j}}$
is an equivalence.

Since
$X_u$
is a smooth projective $\bfk$-variety,
\begin{align*}
\Phi^L_{\cP_{u, j}} \circ \Phi_{\cP_{u, j}} \cong \id_{\Perf (X_u)}, \
\Phi_{\cP_{u, j}} \circ \Phi^L_{\cP_{u, j}} \cong \id_{\Perf (X^\prime_u)}
\end{align*}
imply
\begin{align*}
\cP_{u, j} * (\cP_{u, j})_L \cong \scrO_{\Delta_{u,j}}, \
(\cP_{u, j})_L * \cP_{u, j} \cong \scrO_{\Delta^\prime_{u,j}}
\end{align*}
where
\begin{align*}
\Delta_{u,j} \colon X_u \hookrightarrow X_u \times X_u, \
\Delta^\prime_{u,j} \colon X^\prime_u \hookrightarrow X^\prime_u \times X^\prime_u
\end{align*}
are the diagonal embeddings.
Pullback by
$f_{u, jk}$
yields
\begin{align*}
\cP_{u, jk} * (\cP_{u, jk})_L \cong \scrO_{\Delta_{u,jk}}, \
(\cP_{u, jk})_L * \cP_{u, jk} \cong \scrO_{\Delta^\prime_{u,jk}}
\end{align*}
where
\begin{align*}
\Delta_{u,jk} \colon f^{-1}_{jk}(X_u) \hookrightarrow f^{-1}_{jk}(X_u) \times_{g^{-1}_{jk}(u)} f^{-1}_{jk}(X_u), \
\Delta^\prime_{u,jk} \colon (f^\prime_{jk})^{-1}(X^\prime_u) \hookrightarrow (f^\prime_{jk})^{-1}(X^\prime_u) \times_{g^{-1}_{jk}(u)} (f^\prime_{jk})^{-1}(X^\prime_u)
\end{align*}
are the relative diagonal embeddings.
Thus
$\Phi_{\cP_{u, jk}}$
is an equivalence.
Since
$X_{R_k}$
is covered by the collection
$\{ f^{-1}_{jk} (X_u) \}_u$
with
$u$
running through all the closed points of
$\Spec R_j$,
from
\cite[Proposition 1.3]{LST}
it follows that
$\Phi_{\cP_k}$
is an equivalence.
By the same argument,
we conclude that
$\Phi_{\cP_l}$
are equivalences for all
$l \geq k$.
Applying
\cite[Theorem B]{CNS},
we obtain a quasi-equivalence
$\Perf_{dg}(X_{R_l}) \simeq_{qeq} \Perf_{dg}(X^\prime_{R_l})$
for all
$l \geq k$.
\end{proof}

\begin{cor} \label{cor:inherit}
Let
$X_0, X^\prime_0$
be derived-equivalent Calabi--Yau manifolds of dimension more than two
and
$X_S, X^\prime_S$
their smooth projective versal deformations
over a common nonsingular affine $\bfk$-variety
$\Spec S$.
Assume that
$X_S, X^\prime_S$
correspond to a first order approximation
$R_j \to S$
of
$R_j \hookrightarrow R$
for sufficiently large
$j \in I$.
Then
$X_S, X^\prime_S$
are derived-equivalent.
In particular,
the dg categories
$\Perf_{dg}(X_S), \Perf_{dg}(X^\prime_S)$
of perfect complexes are quasi-equivalent.
\end{cor}
\begin{proof}
By
assumption
one can apply
\pref{thm:inherit}
to find an index
$j \in I$
such that
$X_S, X^\prime_S$
are the pullbacks of smooth projective families
$X_{R_j}, X^\prime_{R_j}$
over
$R_j$
satisfying
$\Perf_{dg}(X_{R_j}) \simeq_{qeq} \Perf_{dg}(X^\prime_{R_j})$.
Consider the distinguished triangle
\begin{align*} 
\Phi^L_{\cP_j} \circ \Phi_{\cP_j} (E_{R_j})
\xrightarrow{\eta_j(E_{R_j})}
E_{R_j}
\to
C(\eta_j(E_{R_j}))
\cong
0.
\end{align*}
Applying the same argument in the above proof to
$X_S \to X_{R_j}$
instead of
$X_{R_k} \to X_{R_j}$,
one sees that
$\Phi_{\cP_S} \colon D^b(X_S) \to D^b(X^\prime_S)$
is an equivalence
with
$\cP_S = (f_{jS} \times f^\prime_{jS})^* \cP_j$.
\end{proof}

\begin{prop}
Let
$X_0, X^\prime_0$
be derived-equivalent Calabi--Yau manifolds of dimension more than two
and
$X_S, X^\prime_S$
smooth projective versal deformations over a common nonsingular affine variety
$\Spec S$.
Then the dg categorical generic fibers are quasi-equivalent.
\end{prop}  
\begin{proof}
We have
\begin{align*}
\Perf_{dg}(X_S) / \Perf_{dg}(X_S)_0
\simeq_{qeq}
\Perf_{dg}(X_{Q(S)})
\simeq_{qeq}
\Perf_{dg}(X^\prime_{Q(S)})
\simeq_{qeq}
\Perf_{dg}(X_S) / \Perf_{dg}(X_S)_0,
\end{align*}
where
the first
and
the tirhd quasi-equivalences follow from Proposition
\pref{prop:dgCGF}.
The second quasi-equivalence follows from
the above corollary,
\cite[Theorem 1.1]{Mora},
\cite[Corollary 4.2]{Morb},
and
\cite[Theorem B]{CNS}.
\end{proof}


\end{document}